\documentclass[12pt]{article}
\usepackage{amsmath, amssymb, amsthm, graphicx, enumitem, color,xcolor}
\usepackage{etoolbox, mathabx}
\usepackage{relsize}
\usepackage{cancel}
\usepackage{todonotes}

\newtheorem{thm}{Theorem}[section]
\newtheorem{conj}[thm]{Conjecture}
\newtheorem{cor}[thm]{Corollary}

\newtheorem{lem}[thm]{Lemma}
\theoremstyle{definition}
\newtheorem{defi}[thm]{Definition}
\newtheorem{example}[thm]{Example}

\newcommand{\C}{\mathbb{C}}
\newcommand{\Z}{\mathbb{Z}}

\newcommand{\wheelpoly}[1]{W_{#1}[\mathbf{z}]}

\newcommand{\asm}{\textnormal{ASM}}
\newcommand{\cat}{\operatorname{Cat}}
\newcommand{\prob}{\mathbb{P}}

\newcommand{\lpbracket}[1]{\wideparen{\ #1\ }}
\newcommand{\CT}{\textnormal{CT}}

\newcommand{\lpgraph}{\textnormal{LP}}
\newcommand{\lpedges}{V(\lpgraph)}

\newcommand{\matched}[1]{\stackrel{#1}{\longleftrightarrow}}
\newcommand{\notmatched}[1]{
\textrm{\raisebox{4pt}{$\begin{array}{c}\scriptstyle #1 \\[-5pt] \longleftrightarrow \hspace{-16pt} \scalebox{0.8}{/} \hspace{10pt} \end{array}$}}
}
\newcommand{\notmatchedvariant}[1]{
\textrm{\raisebox{4pt}{$\hspace{-3pt}\begin{array}{c}\scriptstyle #1 \\[-7pt] \longleftrightarrow \hspace{-12pt} \scalebox{0.8}{/} \hspace{4pt} \end{array}$}}
}
\newcommand{\nonc}{\textnormal{NC}_\Z}
\newcommand{\noncn}[1]{\textnormal{NC}_{#1}}
\newcommand{\pistar}{\Pi_*}
\newcommand{\pistarn}{\Pi_*^{(n)}}
\newcommand{\halflat}{\Z^2_{\textrm{NE}}}

\newcommand{\fpl}{\textrm{FPL}}
\newcommand{\anticluster}{\textnormal{AC}}
\newcommand{\submatching}{\mathrel{\lhd}}

\title{Connectivity patterns in loop percolation I: the rationality phenomenon and constant term identities}
\author{Dan Romik}

\begin{document}
\maketitle

\begin{abstract}
Loop percolation, also known as the dense $O(1)$ loop model, is a variant of critical bond percolation in the square lattice $\mathbb{Z}^2$ whose graph structure consists of a disjoint union of cycles. We study its connectivity pattern, which is a random noncrossing matching associated with a loop percolation configuration. These connectivity patterns exhibit a striking rationality property whereby probabilities of naturally-occurring events are dyadic rational numbers or rational functions of a size parameter $n$, but the reasons for this are not completely understood. We prove the rationality phenomenon in a few cases and prove an explicit formula expressing the probabilities in the ``cylindrical geometry'' as coefficients in certain multivariate polynomials. This reduces the rationality problem in the general case to that of proving a family of conjectural constant term identities generalizing an identity due to Di~Francesco and Zinn-Justin. Our results make use of, and extend, algebraic techniques related to the quantum Knizhnik-Zamolodchikov equation.
\end{abstract}

\bigskip

\vfill


\section{Introduction}

\subsection{Critical bond percolation and loop percolation}

Critical bond percolation on $\Z^2$ is the most natural model of a random subgraph of the square lattice:  each edge of the lattice is included with probability $1/2$, independently of all other edges. Despite the simplicity of its definition, it is difficult to analyze rigorously. Though the spectacular recent results of Smirnov \cite{smirnov} and Lawler-Schramm-Werner \cite{lawler-etal} related to critical exponents, conformal invariance and the Schramm-Loewner Evolution (SLE) have resolved many of the outstanding open problems for the related model of critical site percolation on the triangular lattice, the same problems are still unsolved in the case of bond percolation on $\Z^2$. 

In this paper we will study another percolation model, which is related to critical bond percolation on $\Z^2$ and can be thought of as a variant of it. We will refer to the model as \textbf{loop percolation on $\Z^2$}; an equivalent model has been studied in the statistical physics literature under different names such as the \textbf{dense $O(1)$ loop model} \cite{mitra-etal} or 
\textbf{completely packed loops} \cite{zinn-justin}. Another closely related model is the special case of the \textbf{XXZ spin chain} in which the so-called anisotropy parameter $\Delta$ takes the value $-1/2$; see \cite{batchelor-etal, difran-zj-zuber}.

Our motivation in studying this model was twofold: first, to point out, and attempt to systematically exploit, the inherent interest and approachability of the model from the point of view of probability theory---qualities which have not been emphasized in previous studies. Second, to use and enhance some of the deep algebraic tools (referred to under the broad heading of the \textbf{quantum Knizhnik-Zamolodchikov equation}) that were developed in recent years to compute explicit formulas for several quantities of interest in the model. 
In this paper we focus on algebraic techniques. The follow-up paper \cite{romik-pipes} will contain additional results of a more probabilistic flavor.

Let us start with the definition of the model. We define loop percolation in the entire plane, but later we will also consider the model on a half-plane and on a semi-infinite cylinder.

\begin{defi}[Loop percolation]
Consider the square lattice $\Z^2$ equipped with a checkerboard coloring of the squares of the dual lattice, such that the square whose bottom-left corner is $(0,0)$ is white. \textbf{Loop percolation} is the random subgraph $\lpgraph=(\Z^2,\lpedges)$ of $\Z^2$ obtained by tossing a fair coin for each white square \raisebox{-2.8pt}{\scalebox{2.0}{$\square$}} in the dual lattice, independently of all other coins, and including in $\lpedges$ either the bottom and top edges of \raisebox{-2.8pt}{\scalebox{2.0}{$\square$}} or the left and right edges of \raisebox{-2.8pt}{\scalebox{2.0}{$\square$}} according to the result of the coin toss (Fig.~\ref{fig:lp-two-edge-confs}).
\end{defi}

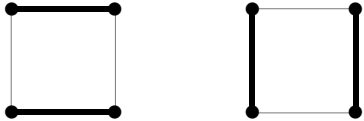
\begin{figure}[h]
\begin{center}
\setlength{\unitlength}{0.18in}
\begin{picture}(10,3)(0,0)
\color{gray}
\put(0,0){\framebox(3,3){}}
\put(7,0){\framebox(3,3){}}
\color{black}
\put(0,0){\circle*{0.35}}
\put(3,0){\circle*{0.35}}
\put(0,3){\circle*{0.35}}
\put(3,3){\circle*{0.35}}
\put(7,0){\circle*{0.35}}
\put(10,0){\circle*{0.35}}
\put(7,3){\circle*{0.35}}
\put(10,3){\circle*{0.35}}
\linethickness{2pt}
\put(0,0){\line(1,0){3}}
\put(0,3){\line(1,0){3}}
\put(7,0){\line(0,1){3}}
\put(10,0){\line(0,1){3}}
\end{picture}
\caption{The two possible edge configurations around a white lattice square.}
\label{fig:lp-two-edge-confs}
\end{center}
\end{figure}

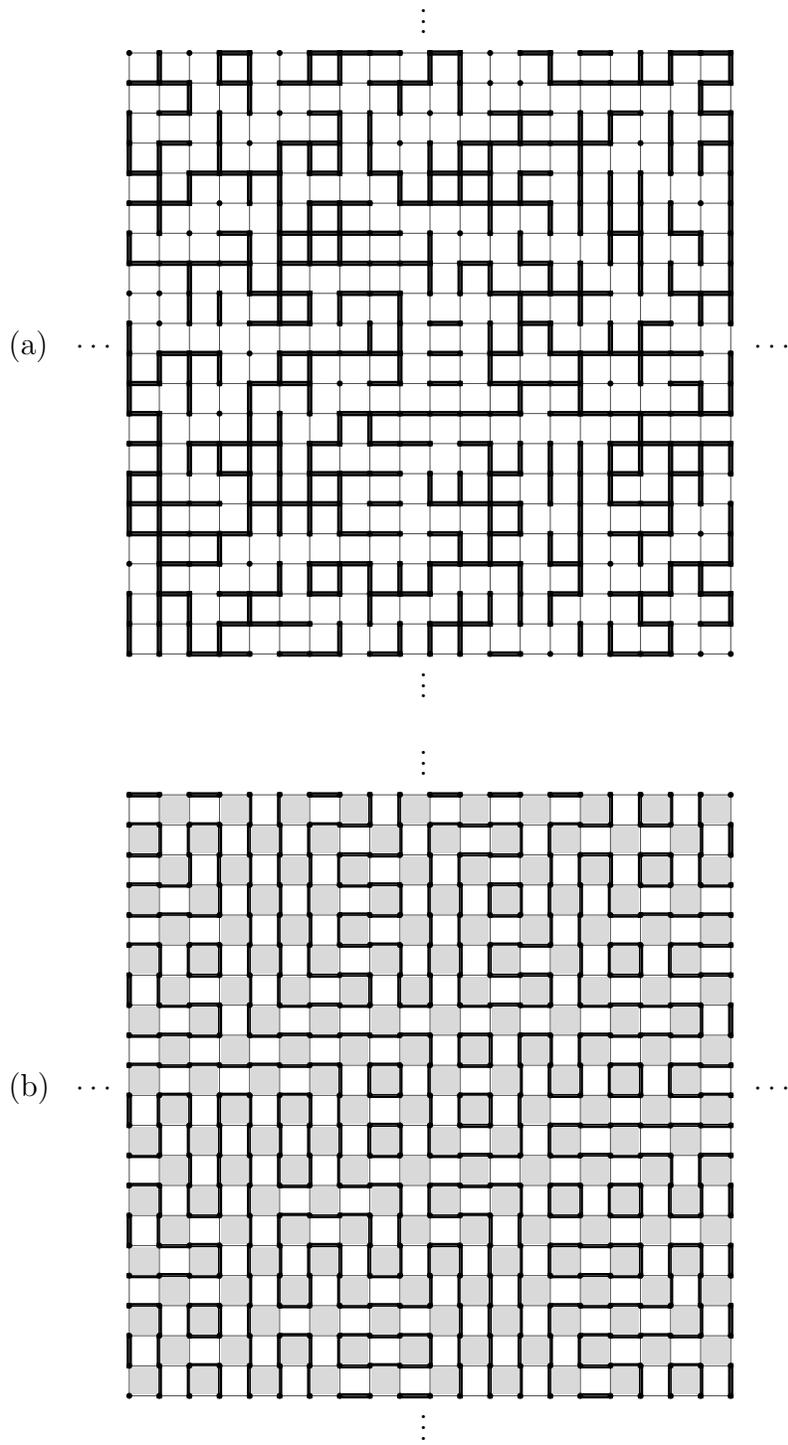
\begin{figure}
\begin{center}
\begin{tabular}{cc}
\setlength{\unitlength}{1in}
\begin{picture}(3.7,3.4)(0,0.1)
\put(-0.35,1.8){(a)}
\put(0,1.8){$\cdots$}
\put(3.55,1.8){$\cdots$}
\put(1.8,0){$\vdots$}
\put(1.8,3.48){$\vdots$}
\put(0.2,0.15){\scalebox{0.92}{\includegraphics{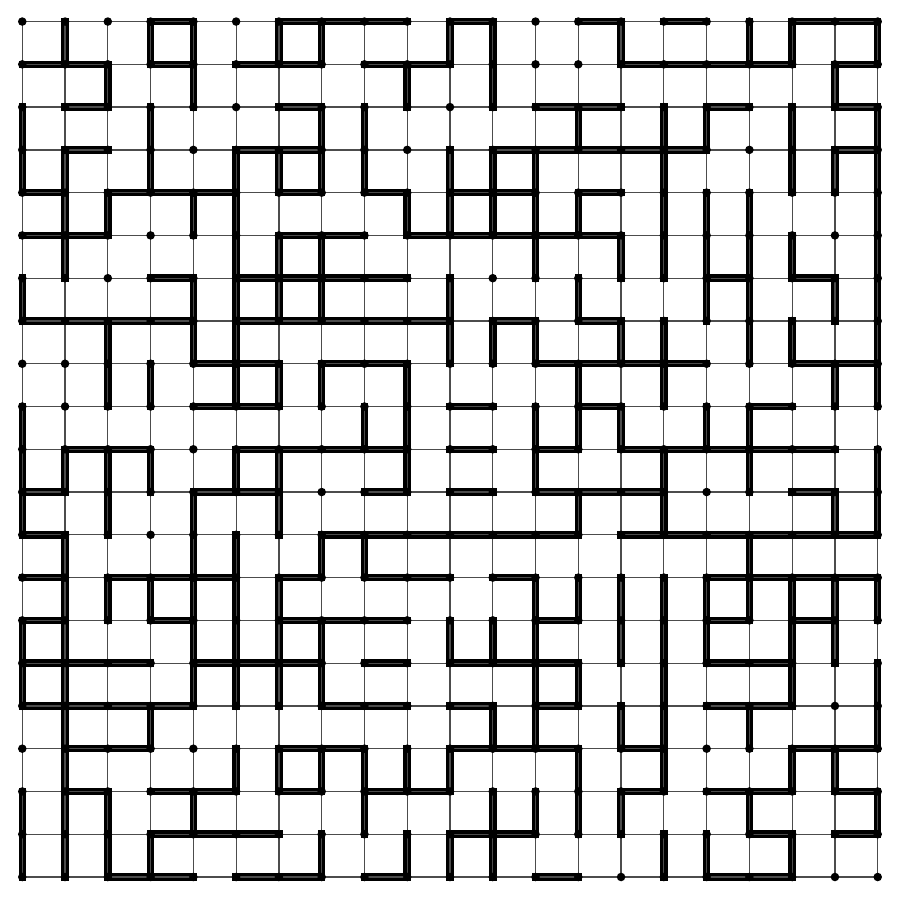}}}
\end{picture}
\\[38pt]
\setlength{\unitlength}{1in}
\begin{picture}(3.7,3.4)(0,0)
\put(-0.35,1.8){(b)}
\put(0,1.8){$\cdots$}
\put(3.55,1.8){$\cdots$}
\put(1.8,0){$\vdots$}
\put(1.8,3.48){$\vdots$}
\put(0.2,0.15){\scalebox{0.92}{\includegraphics{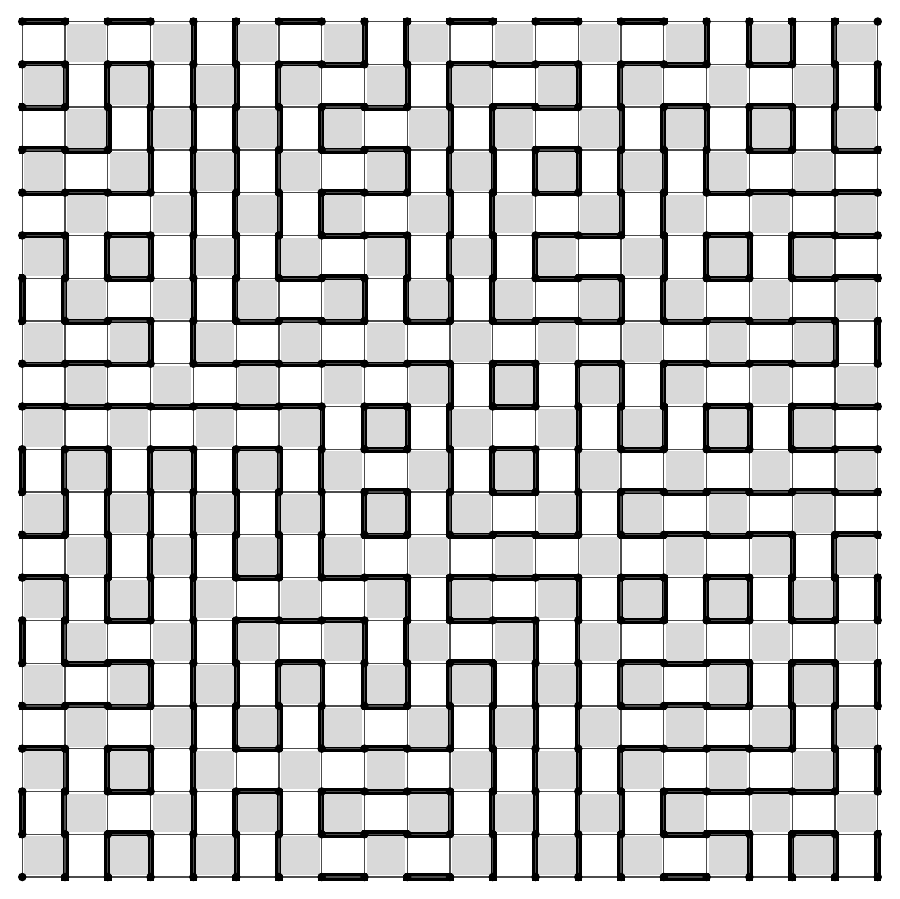}}}
\end{picture}
\end{tabular}
\caption{(a) Critical bond percolation. (b) Loop percolation.}
\label{fig:bondloop}
\end{center}

\end{figure}

Fig.~\ref{fig:bondloop} shows portions of a critical bond percolation configuration and an (unrelated) loop percolation configuration. In fact, the two models are equivalent, in the following sense. The black squares form the vertices of a graph where two black squares are adjacent if their centers differ by one of the vectors $(0,\pm 2), (\pm 2, 0)$. Identifying each black square with its center, this graph decomposes into disjoint components of ``blue'' and ``red'' sites according to parity, each component being isomorphic to $\Z^2$. Define a subgraph $G$ of the blue component whose edges are precisely those that do not cross an edge of the loop percolation graph. Then it is easy to see that $G$ is a critical bond percolation graph; see Fig.~\ref{fig:bond-loop-connection}. Conversely, it is clear that starting from the critical bond percolation graph one can reconstruct the associated loop percolation graph.
\begin{figure}[h]
\begin{center}
\begin{tabular}{ccc}
& $\vdots$ & \\
\raisebox{95pt}{$\cdots$} \hspace{-15pt} &
\scalebox{0.8}{\includegraphics{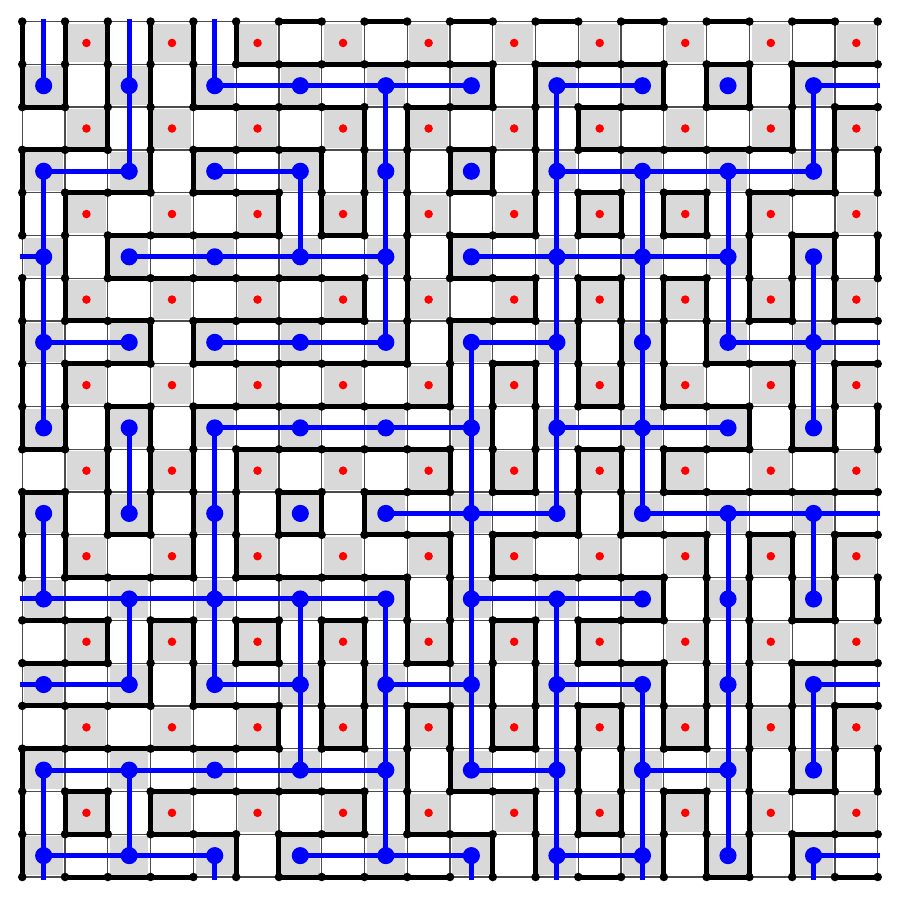}} & 
\hspace{-15pt} \raisebox{95pt}{$\cdots$} \\[-5pt] & $\vdots$
\end{tabular}
\caption{The critical bond percolation graph associated with a loop percolation graph.}
\label{fig:bond-loop-connection}
\end{center}
\end{figure}

As an immediate consequence of the above discussion, we see that the loop percolation graph almost surely decomposes into a disjoint union of cycles, or ``loops.'' Indeed, since each site of $\Z^2$ is incident to precisely two white squares, it has degree $2$, so the connected components of the graph are either loops or infinite paths; however, the presence of an infinite path would imply the existence of an infinite connected component in the ``blue'' critical bond percolation graph containing those blue sites adjacent to an edge in the infinite path, in contradiction to the well-known fact \cite[Lemma~11.12]{grimmett} 
that critical bond percolation on $\Z^2$ almost surely has no such infinite component.

Given the equivalence between loop percolation and critical bond percolation described above, one may ask why it is necessary to study loop percolation separately from critical bond percolation. We believe that there is much to be gained in doing so. In particular, the existing research on loop percolation and connections with other natural statistical physics models such as the XXZ spin chain and Fully Packed Loops indicate that studying this type of percolation as a model in its own right suggests a variety of mathematically interesting questions that one might not otherwise be led to consider by thinking directly of critical bond percolation. 

Furthermore, the techniques developed for studying loop percolation have shown that it belongs to the (loosely defined) class of so-called ``exactly solvable'' models, for which certain remarkable algebraic properties can be used to get precise formulas for various quantities associated with the model. The new results presented in this paper will provide another strong illustration of this phenomenon and give further credence to the notion that loop percolation is quite worthy of independent study. Ultimately, we hope that this line of investigation may lead to new insights that could be used to attack some of the important open problems concerning critical bond percolation.

\subsection{Noncrossing matchings}

Our discussion will focus on certain combinatorial objects known as \textbf{noncrossing matchings} that are associated with loop percolation configurations. Let us recall the relevant definitions. For two integers $a<b$, let $[a,b]$ denote the discrete interval $\{a,a+1,\ldots,b\}$. Recall that a \textbf{noncrossing matching of order $n$} is a perfect matching of the numbers $1,\ldots,2n$ (which we encode formally as a function $\pi:[1,2n]\to[1,2n]$ such that $\pi\circ\pi = \textrm{id}$ and $\pi(k)\neq k$ for all $k\in[1,2n]$; $\pi(k)$ represents the number matched to $k$), which has the additional property that there do not exist numbers $a<b<c<d$ in $[1,2n]$ such that $\pi(a)=c, \pi(b)=d$.
If $\pi(j)=k$ we say that \textbf{$j$ and $k$ are matched under $\pi$} and denote $j \matched{\pi} k$. Denote by $\noncn{n}$ the set of noncrossing matchings of order $n$. It is well-known that the number $|\noncn{n}|$ of elements in $\noncn{n}$ is $\cat(n)=\frac{1}{n+1}\binom{2n}{n}$, the $n$th Catalan number.

Similarly, an \textbf{infinite noncrossing matching} is a one-to-one and onto function $\pi:\Z\to\Z$ that satisfies $\pi\circ\pi = \textrm{id}$ and $\pi(k)\neq k$ for all $k\in\Z$, such that there do not exist integers $a<b<c<d$ for which $\pi$ matches $a$ to $c$ and $b$ to $d$. Denote by $\nonc$ the set of noncrossing matchings on $\Z$.

Both finite and infinite noncrossing matchings can be represented graphically in a diagram in which noncrossing edges, or \textbf{arcs}, are drawn between the elements of any matched pair $j\matched{\pi} k$; see Fig.~\ref{fig:example-noncrossing-matching}. As illustrated in the figure, in the case of a finite noncrossing matchings there are two equivalent representations, as a matching of $2n$ points arranged on a line or on a circle.

\begin{figure}[h]
\begin{center}
\setlength{\unitlength}{1in}
\begin{tabular}{cc}
\raisebox{20pt}{\scalebox{0.8}{\includegraphics{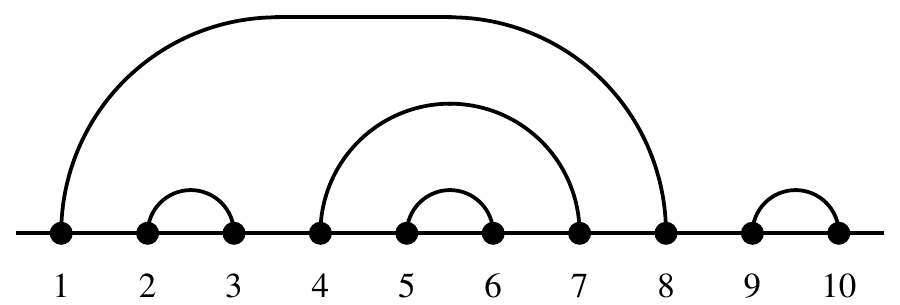}}} &
\scalebox{0.5}{\includegraphics{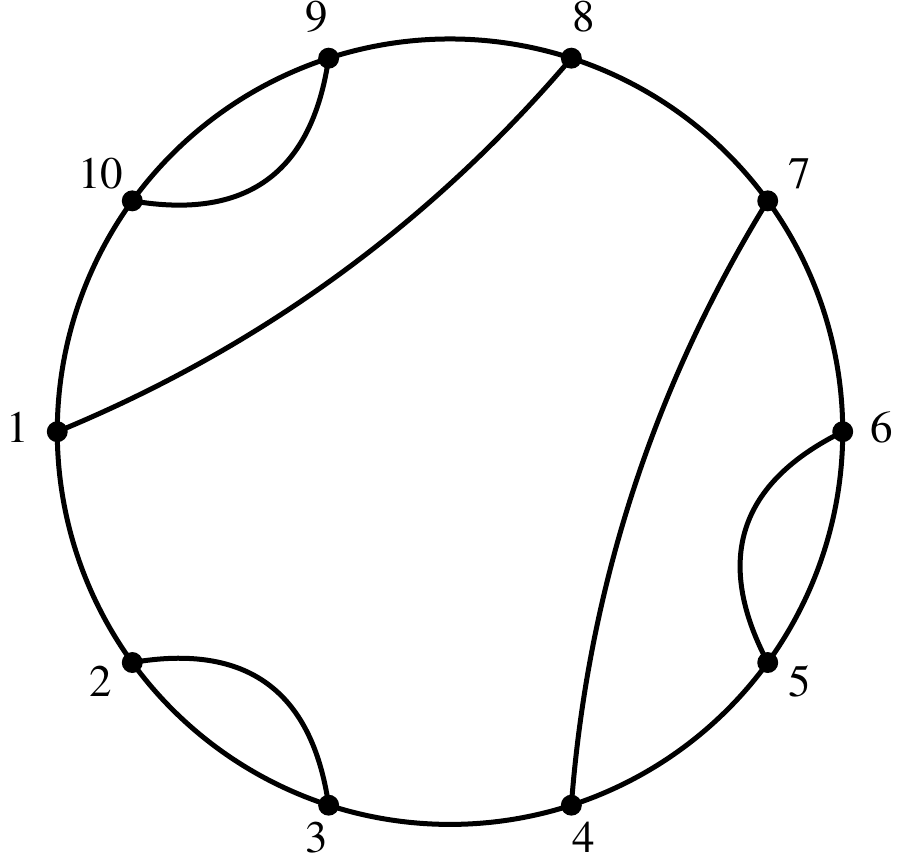}} \\ (a) & (b)
\end{tabular}

\medskip
\begin{tabular}{c}
\begin{picture}(4,1)(0,0)
\put(0.2,0){\includegraphics{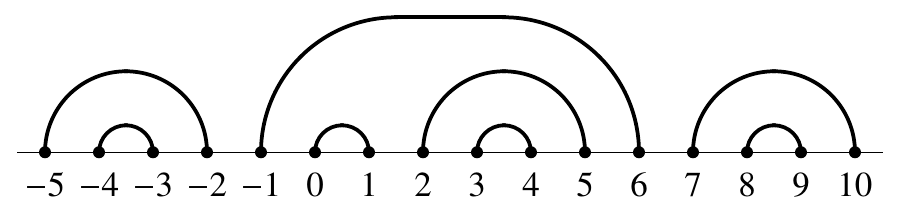}}
\put(0,0.2){\ldots}
\put(3.8,0.2){\ldots}
\end{picture} \\ (c)
\end{tabular}
\caption{(a--b) a finite noncrossing matching shown as a matching of points on a line or on a circle; (c) an infinite noncrossing matching.}
\label{fig:example-noncrossing-matching}
\end{center}
\end{figure}

\subsection{The connectivity pattern of loop percolation: the infinite case}

\label{sec:connectivity-infinite}

We now associate a random noncrossing matching called the \textbf{connectivity pattern} with loop percolation configurations. There will be two variants of the problem according to the type of region being considered, resulting in infinite or finite noncrossing matchings. We start the discussion with the infinite case. To the best of our knowledge, this variant of the problem has not been previously considered in the literature, but nonetheless the rationality phenomenon we will discuss appears most striking in this setting.

\begin{defi}[Half-planar loop percolation]
Denote by $\halflat$ the half-lattice $\halflat = \{ (m,n)\,:\, m+n\ge 0 \}$, and consider a loop percolation graph $\lpgraph_{\textrm{NE}}$ defined as before but only on $\halflat$ instead of on the entire plane. The connectivity pattern associated with it is an infinite noncrossing matching, which we denote by $\pistar$. It is defined as follows: for any $j\in\Z$, to compute $\pistar(j)$, start at the vertex $(j,-j)$. According to our convention regarding the coloring of the squares of the dual lattice, $(j,-j)$ is at the bottom-left corner of a white square, so $(j,-j)$ is incident to precisely one edge in the configuration. Follow that edge, and, continuing along the path of loop percolation edges leading out from $(j,-j)$, one eventually ends up at a vertex $(k,-k)$ on the boundary of the half-lattice (since if we consider the configuration as part of a configuration on the entire plane, the path must be a loop as noted above). In this case, we say that $j$ and $k$ are matched, and denote $j\matched{\pistar}k$; see Fig.~\ref{fig:loopperc-matching} for an illustration.
\end{defi}

\begin{figure}
\begin{center}

\vspace{-20.0pt}

\begin{tabular}{c}
\hspace{-30.0pt}
\scalebox{1}{\includegraphics{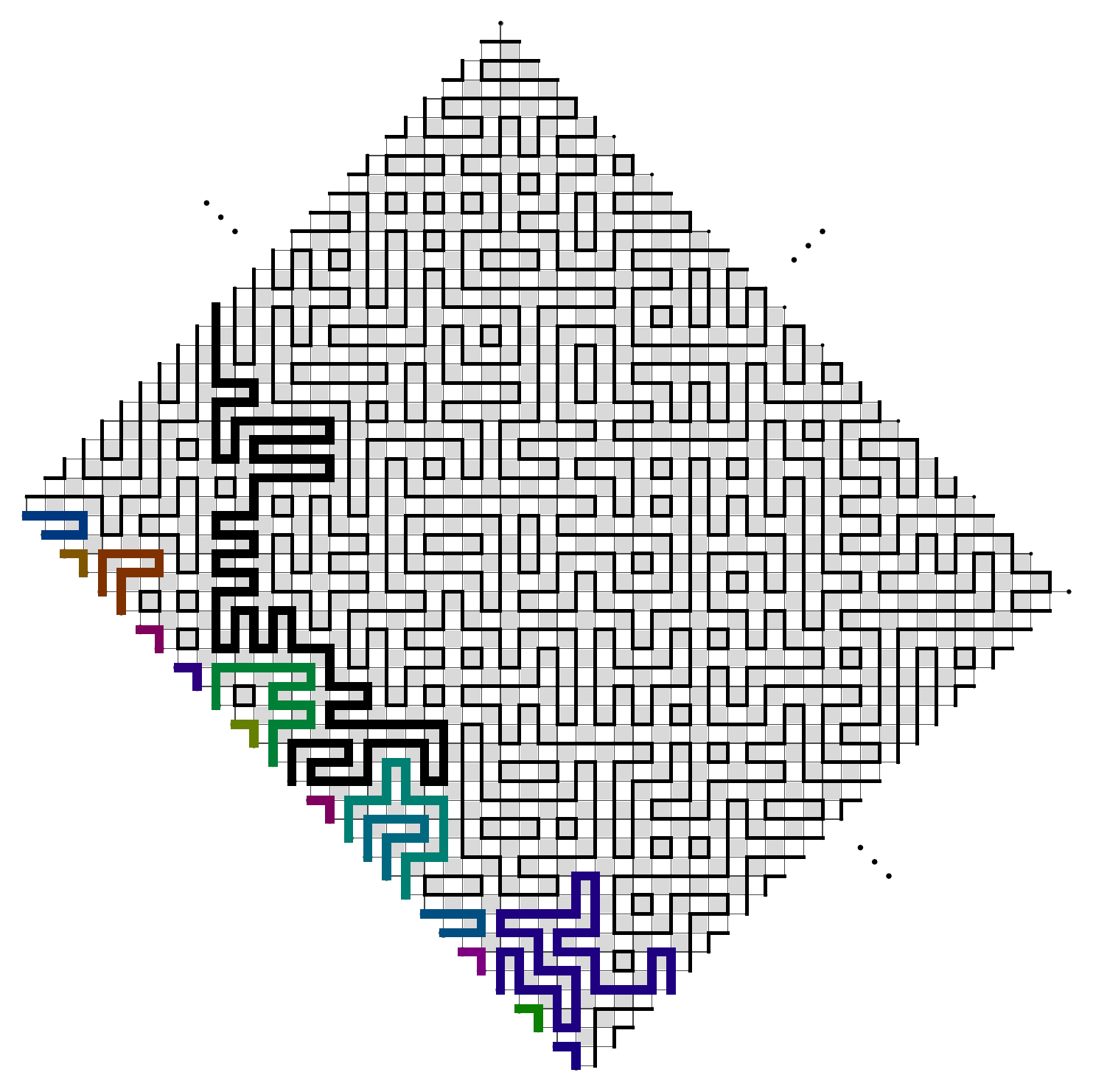}}
\\ (a) \\[15pt]
\hspace{-25pt} \raisebox{5pt}{$\cdots$}\hspace{-5pt} \scalebox{0.76}{\includegraphics{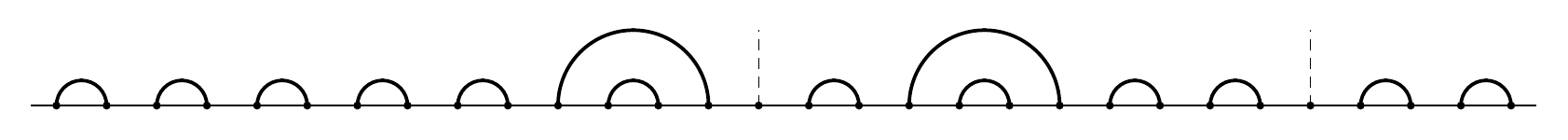}}
\hspace{-5pt}\raisebox{5pt}{$\cdots$}
\\ (b)
\end{tabular}
\caption{(a) A loop percolation configuration (the paths leading out of the vertices $(n,-n)$ are highlighted and colored to emphasize the connectivities); (b) the associated connectivity pattern.}
\label{fig:loopperc-matching}
\end{center}
\end{figure}

It is clear from elementary topological considerations that $\pistar$ is a (random) noncrossing matching. It turns out to have some remarkable distributional properties. In particular, a main motivation for the current work was the observation that the probabilities of many events associated with $\pistar$ turn out rather unexpectedly to be rational numbers which can be computed explicitly. We refer to this as the \textbf{rationality phenomenon}. We will prove it rigorously in a few cases, and conjecture it for a large family of events. Although we have not been able to prove the conjecture in this generality, we also derive
several results that establish strong empirical and theoretical evidence for the correctness of the conjecture and reduce it to a much more explicit algebraic conjecture concerning the Taylor coefficients of a certain family of multivariate polynomials.

Given a finite noncrossing matching $\pi_0\in\noncn{n}$ and an infinite noncrossing matching $\pi\in\nonc$, we say that $\pi_0$ \textbf{is a submatching of $\pi$} if $\pi_{\raisebox{2pt}{\big|} [1,2n]} \equiv \pi_0$, and in this case denote $\pi_0 \submatching \pi$.
We refer to an event of the form $\{ \pi_0 \submatching \pistar \}$ as a \textbf{submatching event}.
We will often represent a submatching event schematically by drawing the diagram associated with the submatching $\pi_0$; for example, the diagram ``\raisebox{-6pt}{\scalebox{0.2}{\includegraphics{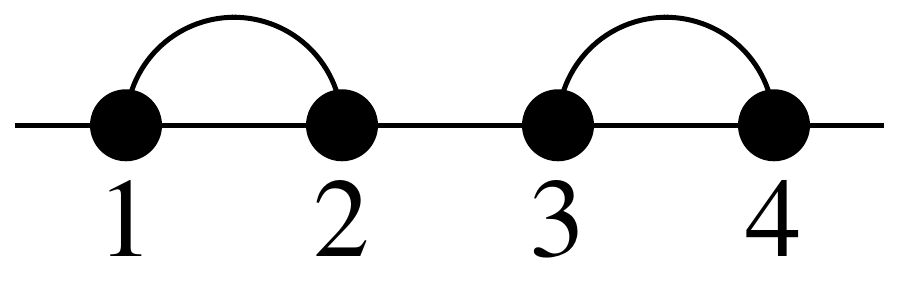}}}'' corresponds to the submatching event $\{ (\Pi_*)_{[1,4]} = (2,1,4,3) \}$. 

\vbox{
\begin{conj}[The rationality phenomenon for submatching events]
\label{conj:rationality}
For any $\pi_0\in\noncn{n}$, the probability $\prob( \pi_0 \submatching \pistar )$ of the associated submatching event is a dyadic rational number that is computable by an explicit algorithm (Algorithm~C described in Appendix~A).
\end{conj}
}

\begin{table}[h]
\begin{tabular}{ccc}
\raisebox{51.5pt}{
\begin{tabular}{c|c}
Event & \textrm{Probability} \\
\hline \\[-1ex]
\scalebox{0.14}{\includegraphics{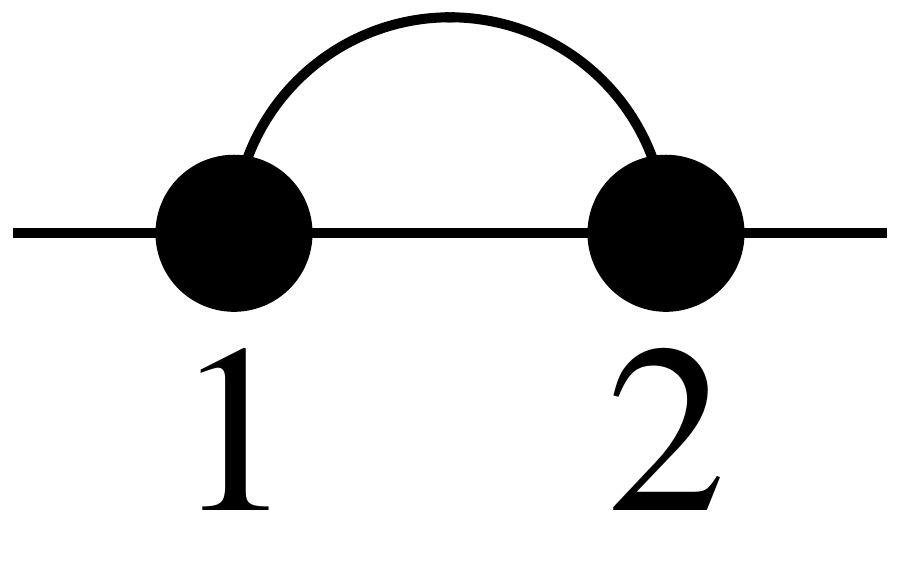}} & 
\raisebox{8pt}{$\displaystyle \frac{3}{8}$}
\\[10pt]
\scalebox{0.3}{\includegraphics{event12-34}} & 
\raisebox{11pt}{$\displaystyle \frac{97}{512}$}
\\[10pt]
\scalebox{0.3}{\includegraphics{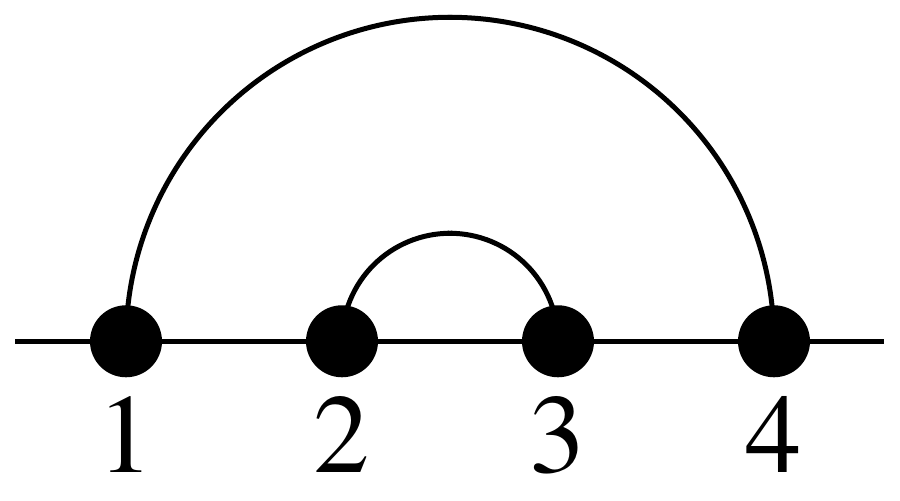}} & 
\raisebox{12pt}{$\displaystyle \frac{59}{1024}$}
\end{tabular}
}
& &
\begin{tabular}{c|c}
Event & \textrm{Probability} \\
\hline \\
\scalebox{0.3}{\includegraphics{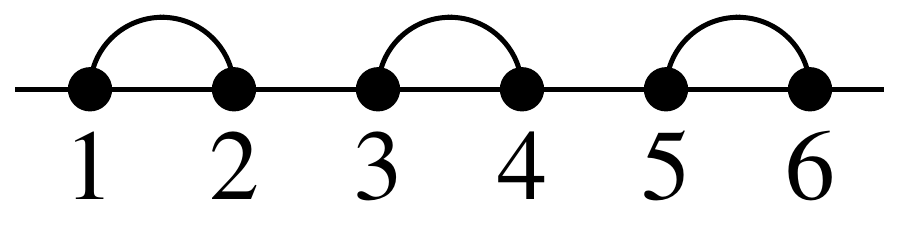}} & 
\raisebox{8pt}{$\displaystyle \frac{214093}{2^{21}}$}
\\[10pt]
\scalebox{0.3}{\includegraphics{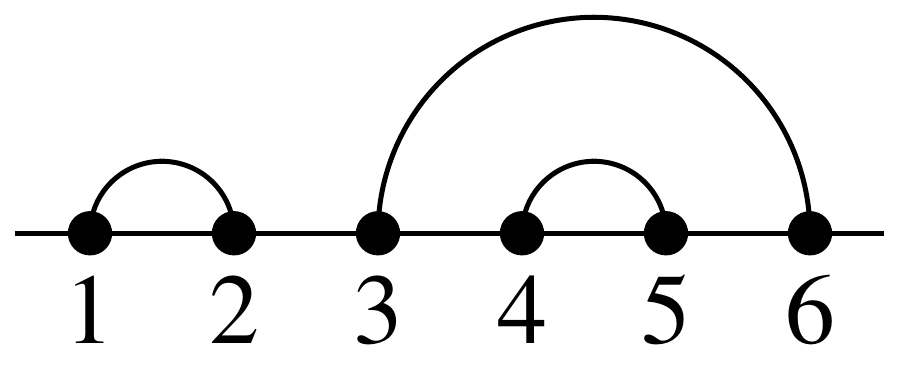}} & 
\raisebox{8pt}{$\displaystyle \frac{69693}{2^{21}}$}
\\[10pt]
\scalebox{0.3}{\includegraphics{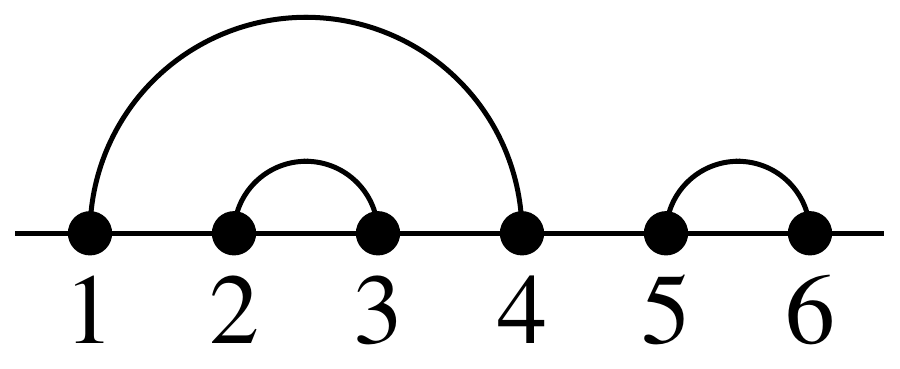}} & 
\raisebox{8pt}{$\displaystyle \frac{69693}{2^{21}}$}
\\[10pt]
\scalebox{0.3}{\includegraphics{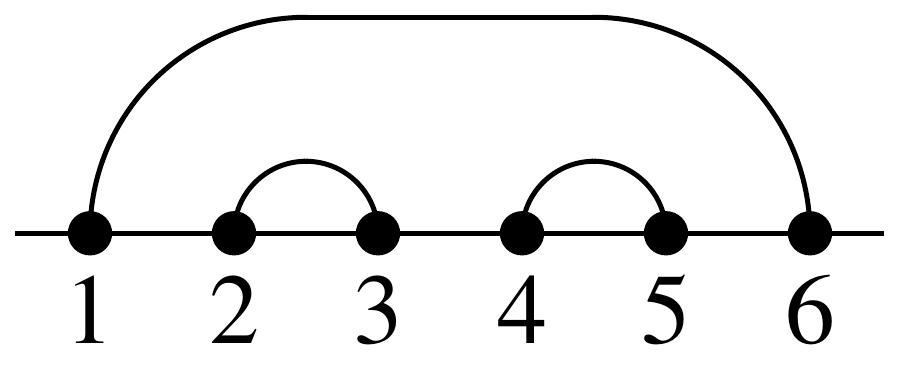}} & 
\raisebox{8pt}{$\displaystyle \frac{37893}{2^{21}}$}
\\[10pt]
\scalebox{0.3}{\includegraphics{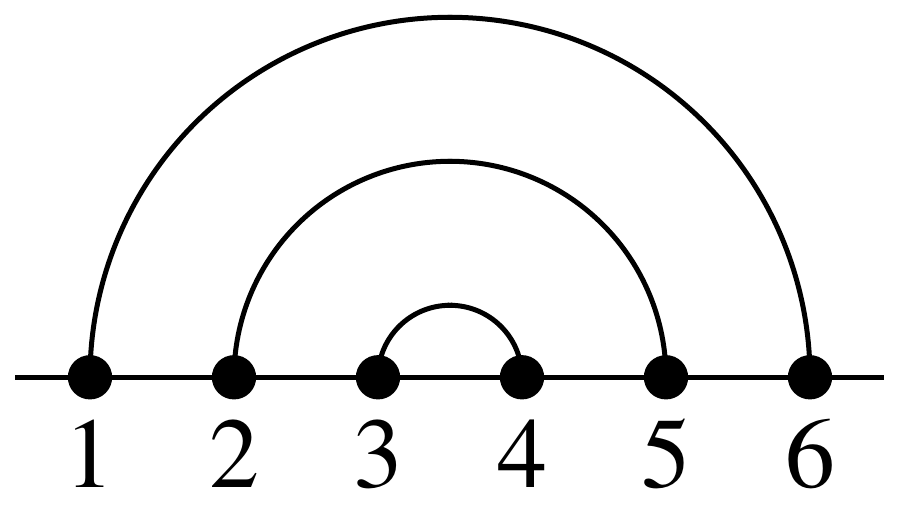}} & 
\raisebox{14pt}{$\displaystyle \frac{7737}{2^{21}}$}
\end{tabular}
\end{tabular}

\vspace{10pt}
\caption{The probabilities of some submatching events in the noncrossing matching $\pistar$. The first two are established rigorously (Theorem~\ref{thm:two-explicit-cases}), others are conjectured.}
\label{table:prob-submatching}
\end{table}

Table~\ref{table:prob-submatching} lists some of the simplest submatching events and their conjectured probabilities. We can prove the first two cases.

\begin{thm}
\label{thm:two-explicit-cases}
We have
\begin{align}
\prob\left( 
\raisebox{-8pt}{\scalebox{0.1}{\includegraphics{event12}}}
\submatching \pistar
\right) &= \frac{3}{8}, \label{eq:pistar-three-eighths} \\[3pt]
\prob\left( 
\raisebox{-6pt}{\scalebox{0.2}{\includegraphics{event12-34}}}
\submatching \pistar
\right) &= \frac{97}{512}.
\label{eq:pistar-ninety-seven}
\end{align}
\end{thm}

As we explain in Subsection~\ref{sec:consequences-halfplane}, the relation \eqref{eq:pistar-three-eighths} will follow as an immediate corollary of a result due to Fonseca and Zinn-Justin (Theorem~\ref{thm:fonseca-zinn-justin} in Subsection~\ref{sec:loop-perc-cylinder}). The second relation \eqref{eq:pistar-ninety-seven} can also be derived from the same theorem but in a slightly less trivial manner; see Theorem~\ref{thm:ninety-seven-finiten} in Subsection~\ref{sec:cylinder-new-results}.

We note that the rationality phenomenon is not restricted to submatching events; a similar statement to Conjecture~\ref{conj:rationality} also appears to hold for a larger family of \textbf{finite connectivity events}, which are events of the form $\bigcap_{ (j,k)\in A } \left\{ j\matched{\pistar} k \right\}$ for some finite set $A\subset \Z\times \Z$. For example, in Theorem~\ref{thm:event135} in Subsection~\ref{sec:consequences-halfplane} we derive rigorously the value $135/1024$ as the probability of a certain finite connectivity event that is not a submatching event. However, we mostly focus on submatching events since our main results (Theorems~\ref{thm:explicit-formulas-finite-n} and~\ref{thm:submatching-event-expansion}) pertain to them and as a result provide strong theoretical evidence for a rationality phenomenon in this setting.

\subsection{Loop percolation on a cylinder: background}

\label{sec:cylinder-background}

Our study of half-planar loop percolation and its connectivity pattern is based on the fact that the model can be approached in a fairly straightforward manner as a limit of a model with a different geometry of a ``semi-infinite cylinder.'' This corresponds to loop percolation in a semi-infinite diagonal strip whose two infinite boundary edges are identified. The precise definition is as follows.

\begin{defi}[Cylindrical loop percolation] For any $n\ge 1$, the cylindrical loop percolation graph $\lpgraph_n$ has vertex set 
\begin{equation} \label{eq:v-lpgraph-n}
V(\lpgraph_n) = \{ (x,y)\in\Z^2\,:\, x+y\ge 0, -n+1 \le x-y\le n \},
\end{equation}
where for any $k\ge 0$ we identify the two vertices $(-n+1+k,n-1+k)$ and $(n+k,-n+k)$. The edges are sampled randomly in the same manner as the usual loop percolation.
\end{defi}

There is a convenient way of representing the cylindrical model as a tiling of a strip of the form $[0,2n]\times[0,\infty)$ with two kinds of square tiles, known as \textbf{plaquettes}, which are shown in Fig.~\ref{fig:plaquettes}. To see the correspondence, first rotate the diagonal strip in the original loop percolation configuration counter-clockwise by 45 degrees, then replace each (rotated) white lattice square with a plaquette whose connectivities correspond to the percolation edges in the white square. Once the plaquette tiling representation is drawn, one can wrap it around a cylinder to obtain a three-dimensional picture; see Fig.~\ref{fig:cylindrical-loop-perc}. The representation using plaquettes is the one traditionally used in most of the existing literature.

\begin{figure}[h]
\begin{center}
\begin{tabular}{ccc}
\scalebox{0.2}{\includegraphics{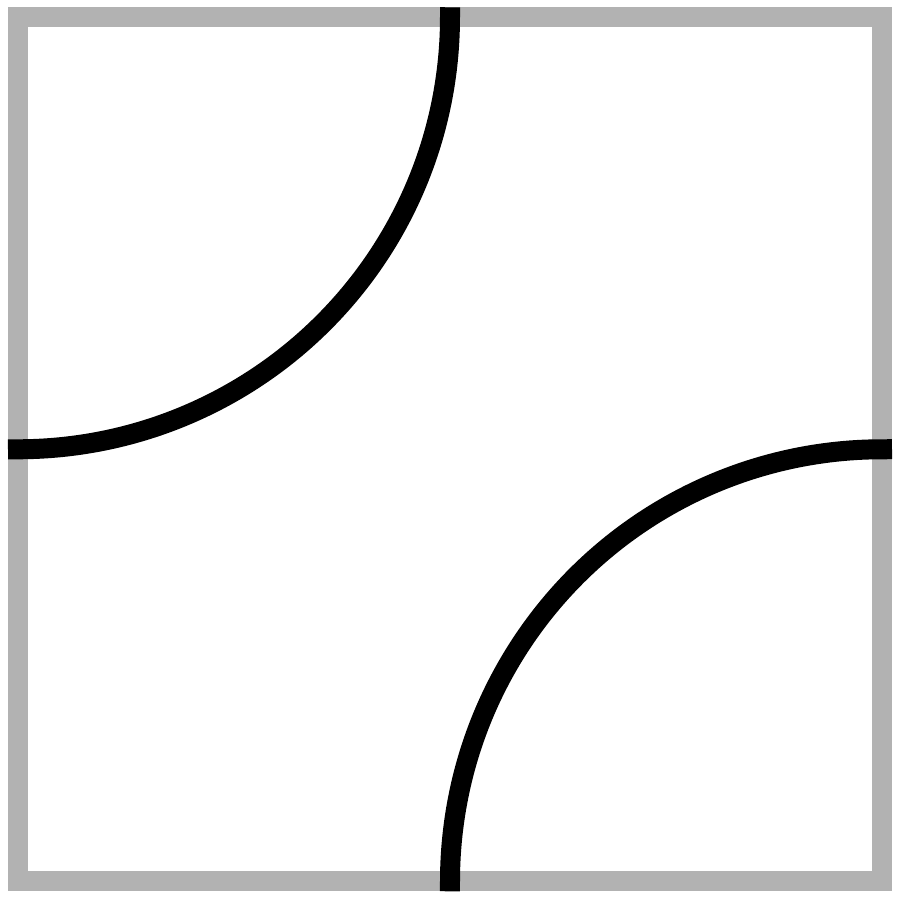}}
& \hspace{30.0pt} &
\scalebox{0.2}{\includegraphics{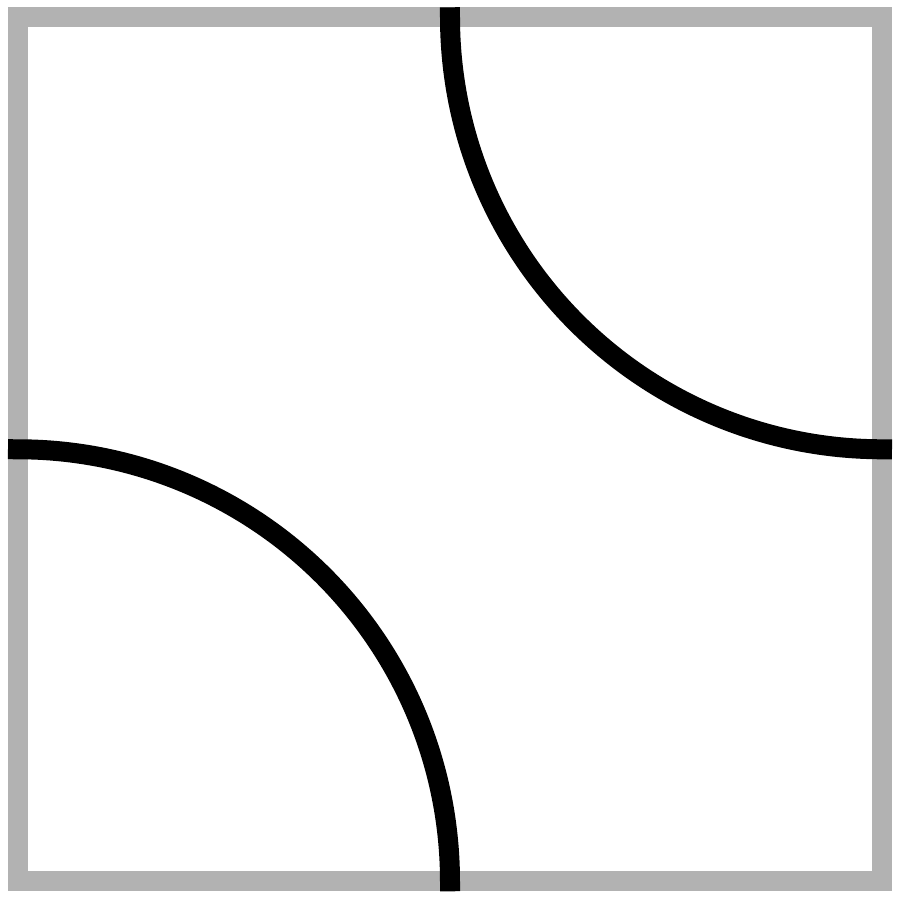}}
\end{tabular}
\caption{The two types of plaquettes.}
\label{fig:plaquettes}
\end{center}
\end{figure}

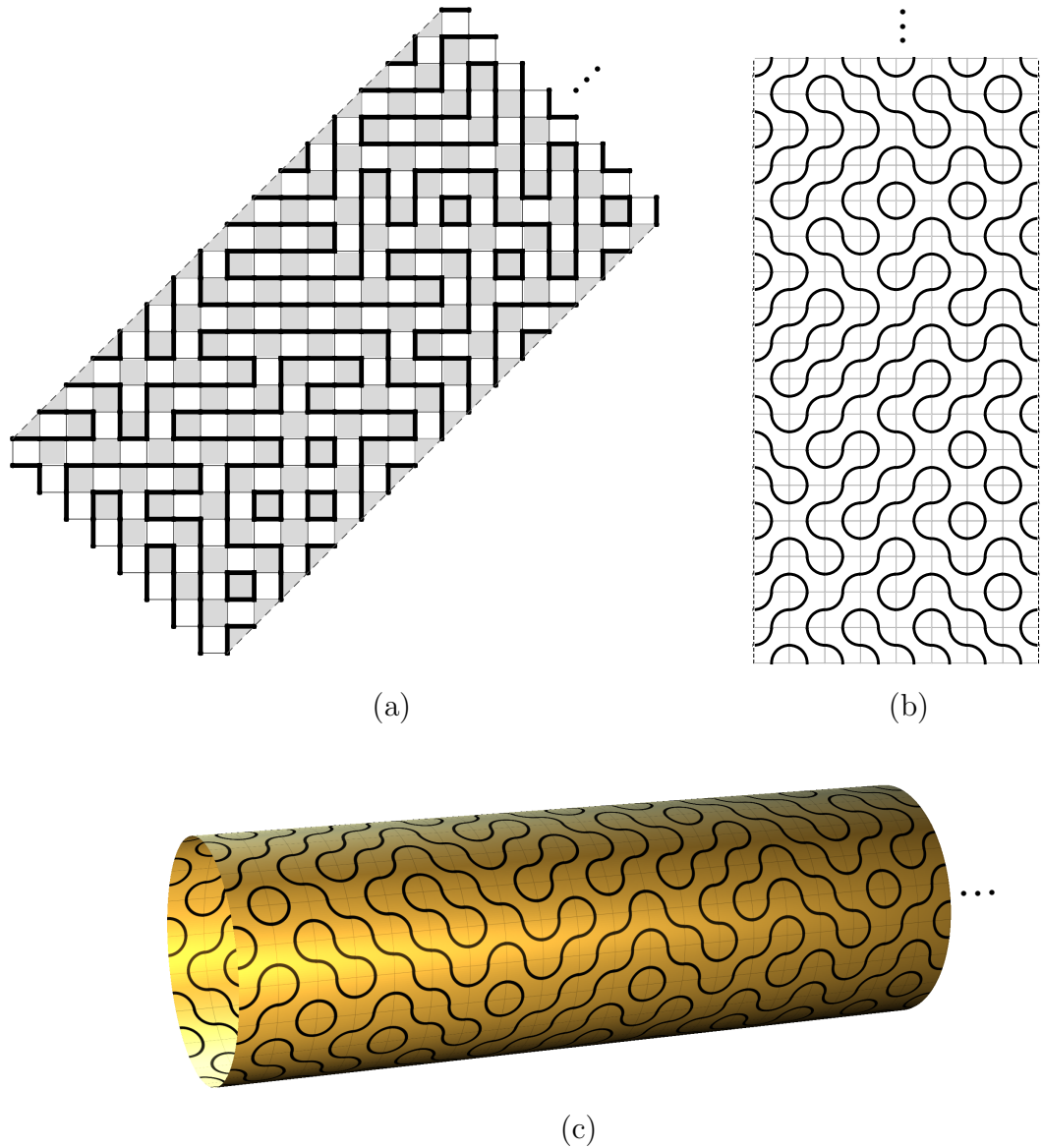
\begin{figure}
\begin{center}
\setlength{\unitlength}{0.5in}
\begin{picture}(10,12.5)(1,0)
\put(0,5){\scalebox{1}{\includegraphics{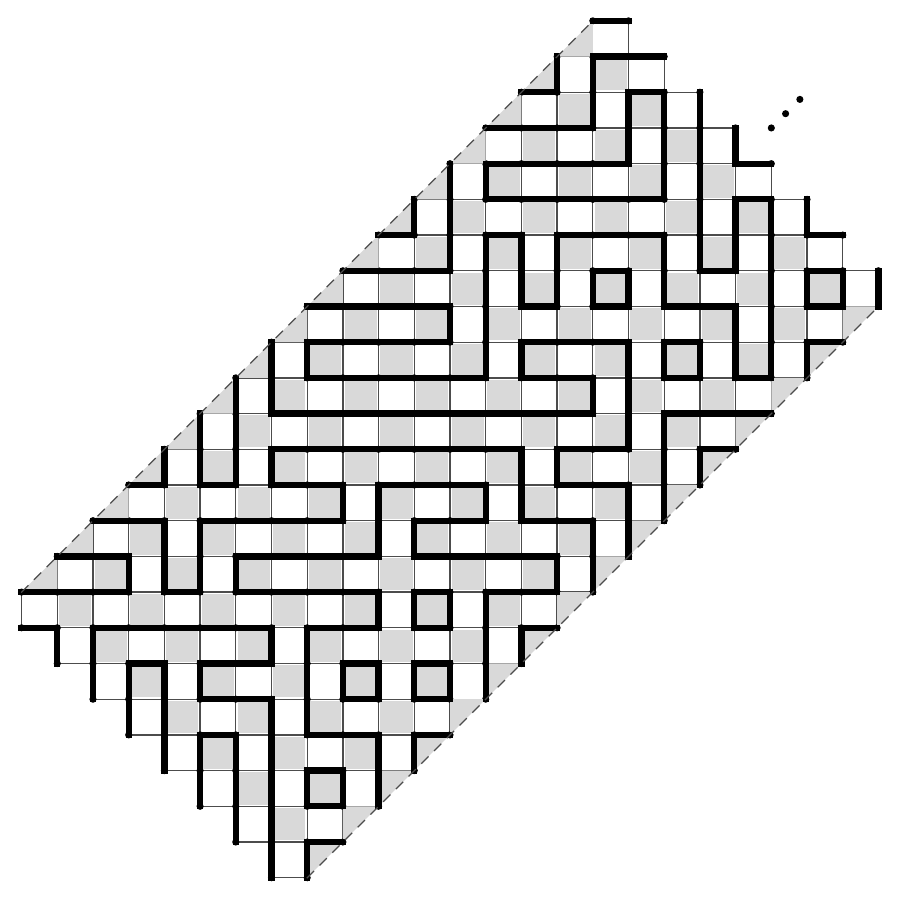}}}
\put(8,5){\scalebox{0.76}{\includegraphics{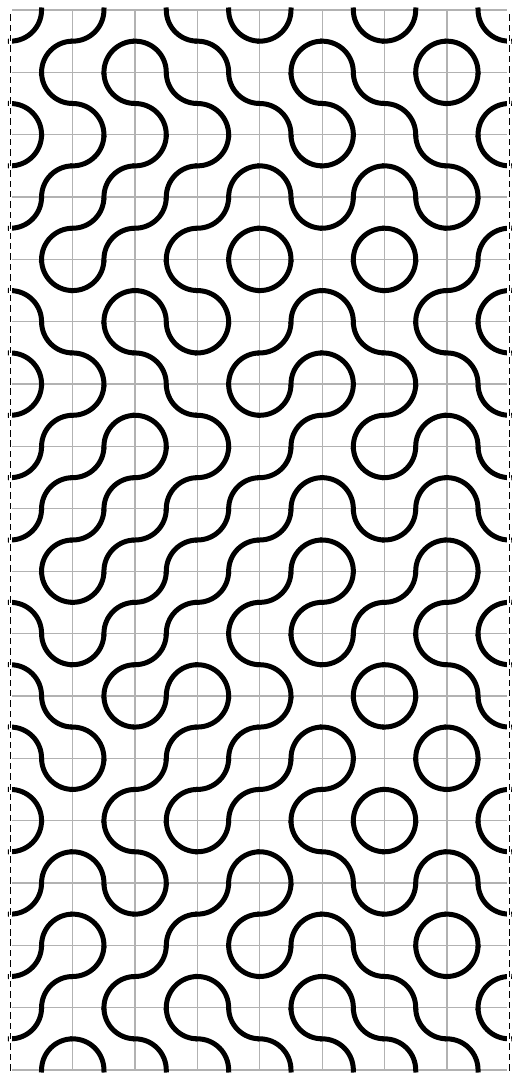}}}
\put(9.65,11.7){\circle*{0.05}}
\put(9.65,11.85){\circle*{0.05}}
\put(9.65,12){\circle*{0.05}}
\put(1.5,0.05){\scalebox{0.16}{\includegraphics{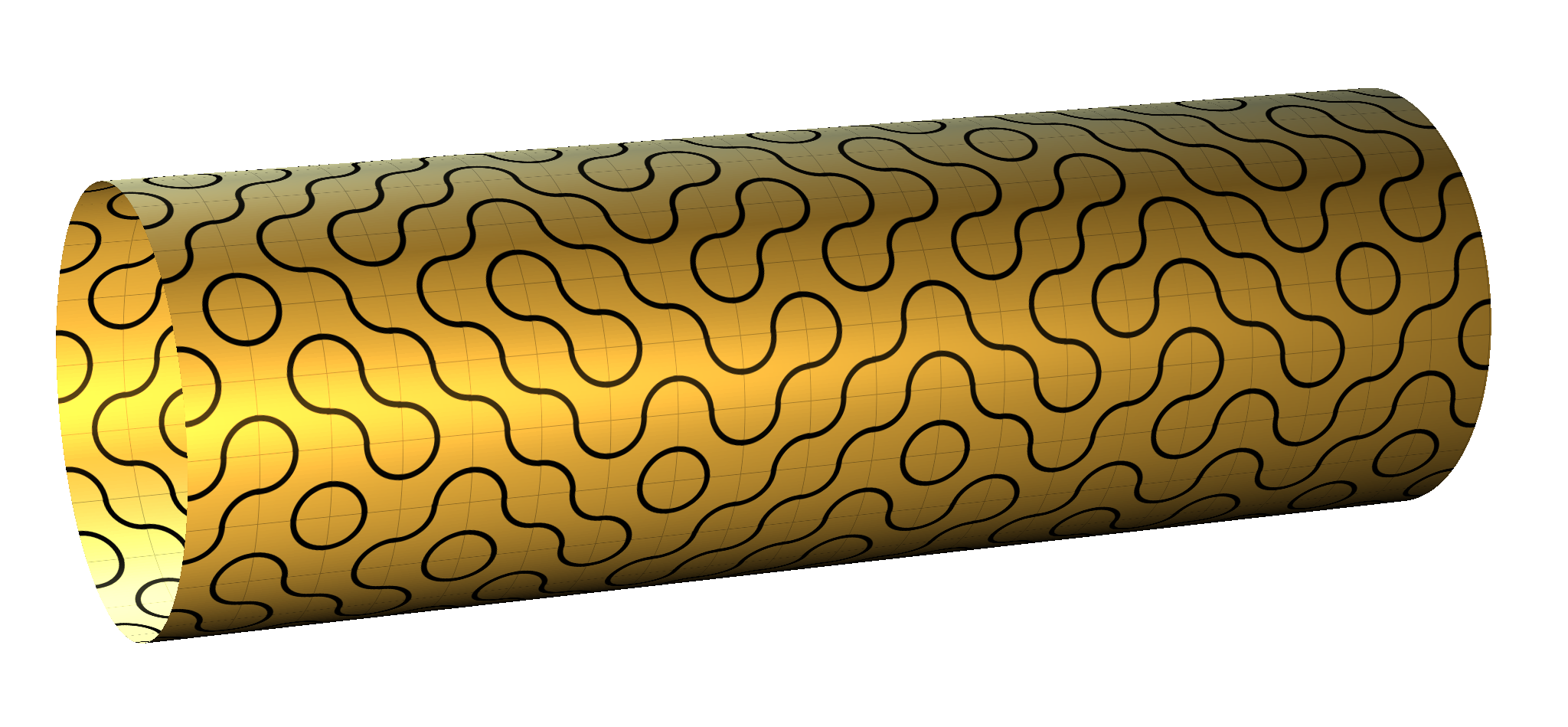}}}
\put(10.3,2.6){\circle*{0.05}}
\put(10.45,2.61){\circle*{0.05}}
\put(10.6,2.62){\circle*{0.05}}
\put(4,4.5){(a)}
\put(9.5,4.5){(b)}
\put(6,0){(c)}
\end{picture}
\caption{(a) Cylindrical loop percolation: the north-west and south-east boundary edges of the infinite diagonal strip are identified along the dashed lines; (b) representation of the same configuration as a tiling of two kinds of plaquettes; (c) wrapping a tiling of plaquettes around a cylinder.
}
\label{fig:cylindrical-loop-perc}
\end{center}
\end{figure}

Let $\pistarn$ denote the connectivity pattern associated with the cylindrical loop percolation graph $\lpgraph_n$, defined analogously to $\pistar$ by following each of the paths originating at the vertices $(-j,j), -n+1\le j\le n$ until it re-emerges in a vertex $(-k,k)$. It is easy to see that $\pistarn$ is a random \emph{finite} noncrossing matching of order $n$; the only justification needed is the following trivial lemma that shows that it is well-defined.

\begin{lem}
In the graph $\lpgraph_n$, almost surely all paths are finite.
\end{lem}

\begin{proof}
Consider the tiling representation of the model as in Fig.~\ref{fig:cylindrical-loop-perc}(b).
It is easy to see that a row in the tiling in which the plaquettes alternate between the two types of plaquettes forces all paths below it to be finite by ``bouncing back'' any path attempting to cross the row. Any row has probability $1/2^{2n-1}$ to have such structure, independently of other rows, so almost surely there will be infinitely many such rows.
\end{proof}

The cylindrical connectivity pattern $\pistarn$ has been the subject of extensive research in recent years, and appears in a surprising number of ways that do not seem immediately related to each other. Before presenting our new results on the behavior of $\pistarn$, let us survey some of the known theory. 

For a matching $\pi\in\noncn{n}$ denote $\mu_\pi = \prob(\pistarn = \pi)$. A natural question is how to compute the probability vector $\boldsymbol{\mu}_n = (\mu_\pi)_{\pi\in\noncn{n}}$. It is easy to see that $\boldsymbol{\mu}_n$ is the stationary distribution of a Markov chain on $\noncn{n}$ where the transitions $\pi\to \pi'$ correspond to the operation of extending the cylinder by one additional row of random plaquettes. This operation leaves the distribution $\boldsymbol{\mu}_n$ invariant since it results in a plaquette tiling equal in distribution to the original one. Formally, there is a Markov transition matrix $T_n^{(1/2)} = (t^{(1/2)}_{\pi,\pi'})_{\pi,\pi'\in\noncn{n}}$ such that
$$ \boldsymbol{\mu}_n T_n^{(1/2)} = \boldsymbol{\mu}_n, $$
where $\boldsymbol{\mu}_n$ is considered as a row vector.
(In the statistical physics literature $T_n^{(1/2)}$ is usually referred to as a transfer matrix.)

The reason for the notation $T_n^{(1/2)}$ is that we can generalize the model and define a matrix $T_n^{(p)}$ for any $0< p<1 $, corresponding to a tiling of independently sampled random plaquettes in which each choice between the two types of plaquette is made by tossing a coin with bias $p$. Remarkably, the value of $p$ does not affect the distribution of the connectivity pattern.

\begin{thm}
\label{thm:commuting-matrices}
The transition matrices $(T_n^{(p)})_{0< p <1}$ are a commuting family of matrices. Consequently, since they are all stochastic, they all share the same row eigenvector $\boldsymbol{\mu}_n$ associated with the eigenvalue $1$.
\end{thm}

Note that using the transition matrix $T_n^{(1/2)}$, or $T_n^{(p)}$ for any fixed $p$, is not necessarily the easiest way to compute $\boldsymbol{\mu}_n$, since to compute the entries of $T_n^{(1/2)}$ one has to count the number of possible rows of $n$ plaquettes (out of the $2^{2n}$ possibilities) that would cause a given state transition $\pi\to\pi'$. It turns out that there is a more convenient (and more theoretically tractable) system of linear equations for computing $\boldsymbol{\mu}_n$, which involves a simpler matrix $H_n$ that arises as a limiting case of the $T_n^{(p)}$ as $p\to 0$, or symmetrically as $p\to 1$.

To define $H_n$, first define for each $k\in[1,2n]$ a mapping $e_k:\noncn{n}\to\noncn{n}$ given by
\begin{equation}
\label{eq:temperley-lieb-gens}
e_k(\pi)(m) = \begin{cases} 
\pi(m) & \textrm{if }m\notin \{k,k+1,\pi(k),\pi(k+1)\}, \\
k+1 & \textrm{if }m=k, \\
k & \textrm{if }m=k+1, \\
\pi(k+1) & \textrm{if }m=\pi(k), \\
\pi(k) & \textrm{if }m=\pi(k+1),
\end{cases}
\end{equation}
where $k+1$ is interpreted as $1$ if $k=2n$.
In words, $e_k(\pi)$ is the matching $\pi'$ obtained from $\pi$ by unmatching the pairs $k\matched{\pi} \pi(k)$ and $k+1\matched{\pi}\pi(k+1)$ and replacing them with the matched pairs $k\matched{\pi'}k+1$ and $\pi(k)\matched{\pi'}\pi(k+1)$. In the case when $k$ and $k+1$ are already matched under $\pi$, nothing happens and $e_k(\pi)=\pi$. It is easy to see that in general $e_k(\pi)$ is a noncrossing matching. We refer to the $e_k$ as the \textbf{Temperley-Lieb operators}. (They generate an algebra known as the \textbf{Temperley-Lieb algebra} \cite{degier, temperley-lieb}, but this has no bearing on the present discussion.)

\begin{thm}
\label{thm:hamiltonian-eigenvector}
Define a square matrix $H_n$ with rows and columns indexed by elements of $\noncn{n}$ by
\begin{equation} \label{eq:def-h-inf-gen}
(H_n)_{\pi,\pi'} =  2n-\#\left\{ 1\le k\le 2n\,:\, e_k(\pi)=\pi' \right\}.
\end{equation}
Then $\boldsymbol{\mu}_n$ satisfies
\begin{equation*} \label{eq:mun-hn-zero}
\boldsymbol{\mu}_n H_n = 0.
\end{equation*}
\end{thm}

Theorems~\ref{thm:commuting-matrices} and \ref{thm:hamiltonian-eigenvector} seem to be well-known to experts in the field, but their precise attribution is unclear to us. As explained in \cite{zinn-justin} (Sections 2.2.2 and 3.3.2), Theorem~\ref{thm:commuting-matrices} follows from the Yang-Baxter equation. A fact equivalent to Theorem~\ref{thm:hamiltonian-eigenvector} is mentioned without proof in \cite[Section 3]{mitra-etal}. We provide a simple proof of this result in Appendix A.

It is easy to see that the matrix $M_n=I-\frac{1}{2n}H_n$ (where $I$ is the identity matrix) is nonnegative and stochastic, i.e., it is a Markov transition matrix, associated with yet another Markov chain that has $\boldsymbol{\mu}_n$ as its stationary vector. That is, $\boldsymbol{\mu}_n$ can be computed by solving the vector equation $\boldsymbol{\mu}_n M_n = \boldsymbol{\mu}_n$, which can be written more explicitly as the linear system 
\begin{equation} \label{eq:razumov-stroganov-linear-system}
\mu_\pi = \frac{1}{2n} \sum_{k=1}^{2n} \sum_{\begin{array}{c} \scriptstyle \pi'\in\noncn{n}, \\[-5pt] \scriptstyle e_k(\pi')=\pi \end{array}}  \mu_{\pi'} \qquad (\pi \in\noncn{n}).
\end{equation}
The Markov chain $(\pi_m)_{m\ge 0}$ associated with the transition matrix $M_n$ has the following simple description as a random walk on $\noncn{n}$, known as the \textbf{Temperley-Lieb random walk} or \textbf{Temperley-Lieb stochastic process} \cite{pearce-etal}: start with some initial matching $\pi_0$; at each step, to obtain $\pi_{m+1}$ from $\pi_m$, choose a uniformly random integer $k \in \{1,2,\ldots,2n\}$ (independently of all other random choices), and set $\pi_{m+1}=e_k(\pi_m)$.

The application of the maps $e_k$ can be represented graphically by associating with each $e_k$ a ``connection diagram'' of the form
\begin{center}
\scalebox{0.8}{\includegraphics{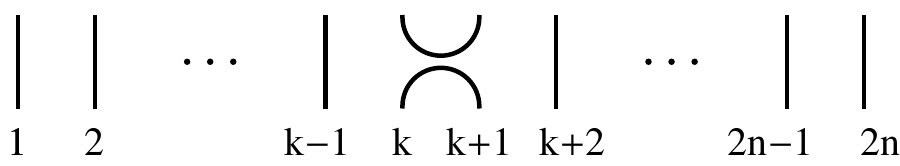}}
\end{center}
which will be ``composed'' with the diagram of the matching $\pi$ to which $e_k$ is applied by drawing one diagram below the other. (Note that to get the correct picture in the case $k=2n$ one should interpret the diagram as being drawn around a cylinder.) By composing a sequence of such diagrams with a noncrossing matching diagram one can compute the result of the application of the corresponding maps to the matching; see Fig.~\ref{fig:pipe-diagram}. Using this graphical interpretation, it can be seen easily that the stationary distribution $\boldsymbol{\mu}_n$ is realized as the distribution of the connectivity pattern of endpoints in an infinite composition of $e_k$ connection diagrams (where the values of $k$ are i.i.d.\ discrete uniform random variables in $\{1,\ldots,2n\}$), drawn on a semi-infinite cylinder. This is illustrated in Fig.~\ref{fig:tl-randomwalk-graphical}.

\begin{figure}
\begin{center}
\setlength{\unitlength}{0.3in}
\begin{picture}(23,7.5)(2,-1)
\put(3.2,0){\scalebox{0.7}{\includegraphics{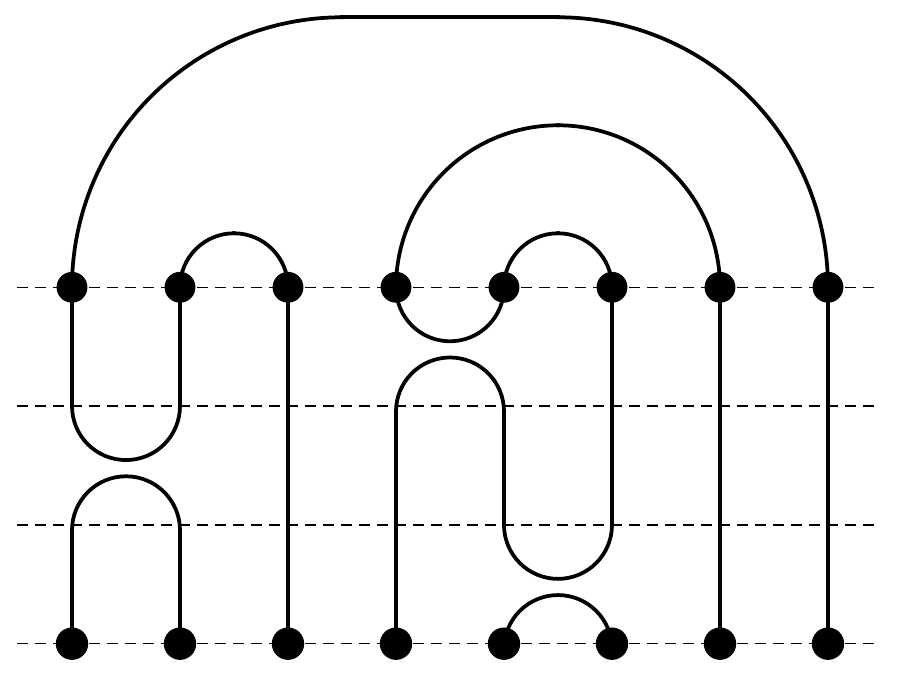}}}
\put(2.4,4.8){$\left\{ \vphantom{ \begin{array}{c}\ \\[24pt] \  \end{array} }\right.$}
\put(0.2,5.0){\footnotesize Original}
\put(0,4.5){\footnotesize matching $\pi$}
\put(2.4,2.98){$\left\{ \vphantom{ \begin{array}{c}\ \\[-10pt] \  \end{array} }\right.$}
\put(2.4,1.8){$\left\{ \vphantom{ \begin{array}{c}\ \\[-10pt] \  \end{array} }\right.$}
\put(2.4,0.63){$\left\{ \vphantom{ \begin{array}{c}\ \\[-10pt] \  \end{array} }\right.$}
\put(1.7,2.98){\footnotesize $e_4$}
\put(1.7,1.8){\footnotesize $e_1$}
\put(1.7,0.63){\footnotesize $e_5$}
\put(13.3,1){\scalebox{0.7}{\includegraphics{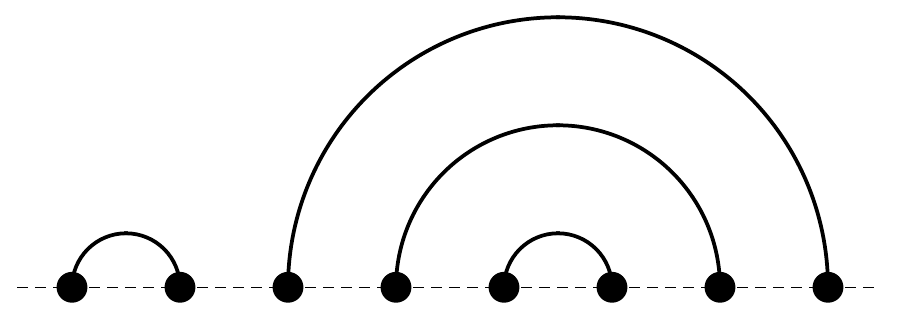}}}
\put(11.8,2){\vector(1,0){1.3}}
\put(6.5,-1){(a)}
\put(17,-1){(b)}
\put(16,0.3){$\pi'=e_5 e_1 e_4 \pi$}
\end{picture}
\caption{Graphical representation of the application of a sequence of operators $e_k$, $1\le k\le 2n$ as a ``composition of diagrams'': (a) the diagrams associated with operators $e_4,e_1,e_5$ are attached to the diagram of the original matching; (b) the lines are ``pulled'' (and any loops are discarded) to arrive at the diagram for the transformed matching. }
\label{fig:pipe-diagram}
\end{center}
\end{figure}
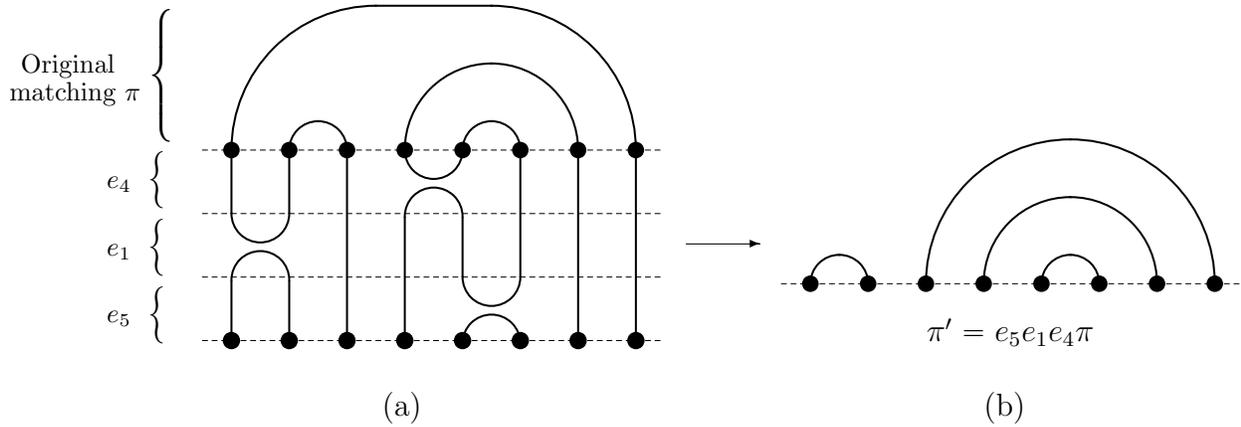

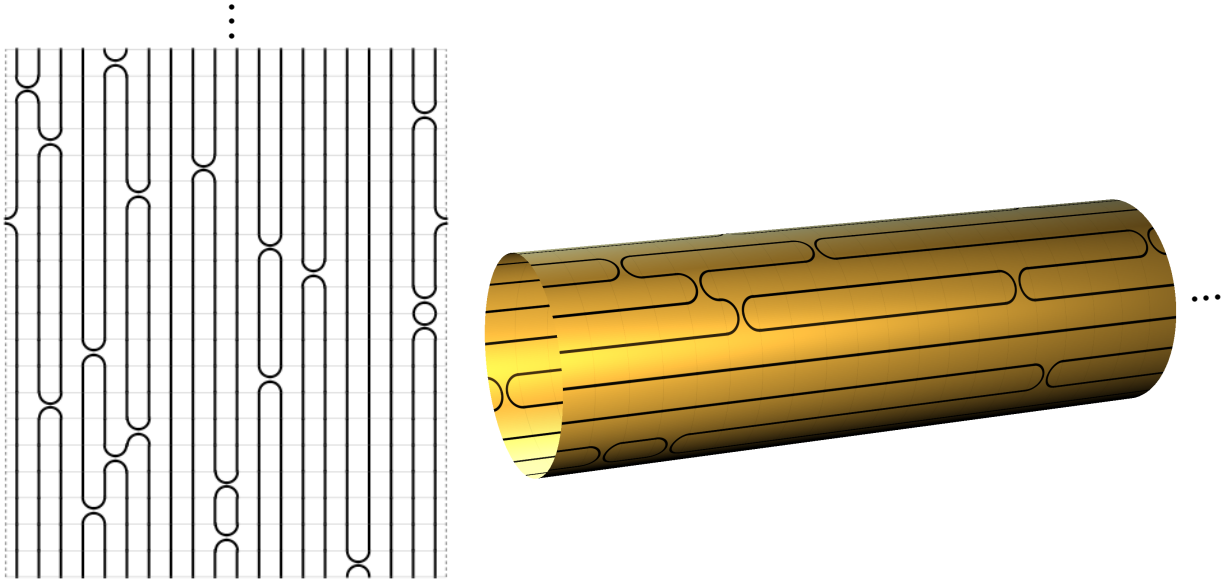
\begin{figure}
\begin{center}
\setlength{\unitlength}{0.3in}
\begin{picture}(20,9.4)
\put(-1,0){\scalebox{0.48}{\includegraphics{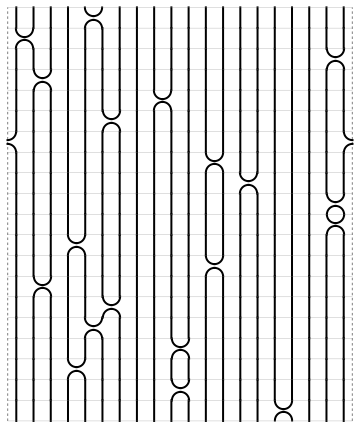}}}
\put(7,1){\scalebox{0.14}{\includegraphics{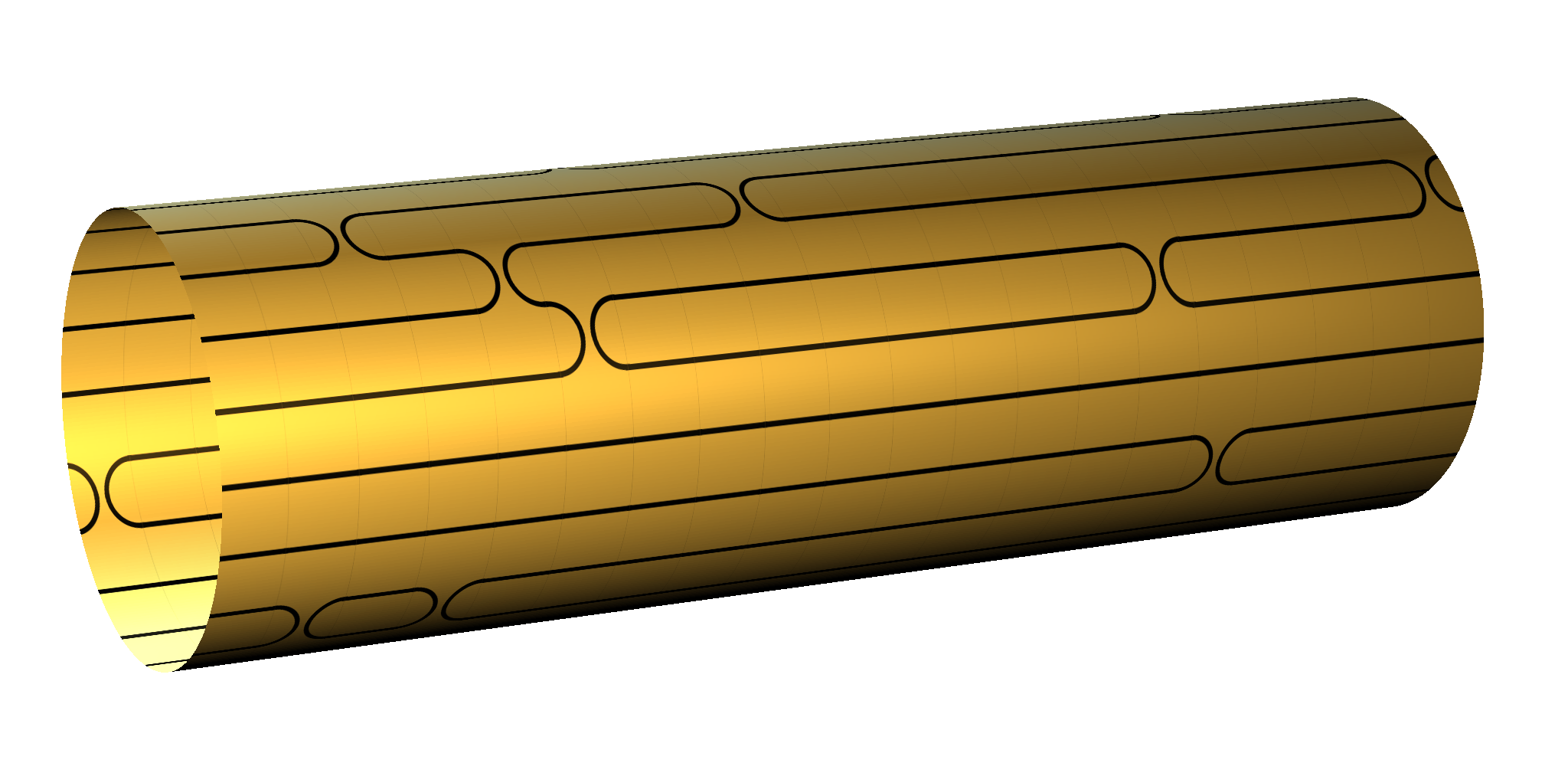}}}
\put(3.1,9.6){\circle*{0.1}}
\put(3.1,9.85){\circle*{0.1}}
\put(3.1,10.1){\circle*{0.1}}
\put(19.9,5){\circle*{0.1}}
\put(20.1,5.02){\circle*{0.1}}
\put(20.3,5.04){\circle*{0.1}}
\end{picture}
\caption{The connectivity pattern of endpoints in a semi-infinite arrangement of uniformly random i.i.d.\ Temperley-Lieb operator diagrams is invariant under the addition of another operator, hence has $\boldsymbol{\mu}_n$ as its distribution.}
\label{fig:tl-randomwalk-graphical}
\end{center}
\end{figure}

Yet another interpretation of $\boldsymbol{\mu}_n$ is as the ground state eigenvector associated with a certain quantum many-body system, the XXZ spin chain. The Hamiltonian of this spin chain is an operator acting on the space $V=(\C^2)^{\otimes 2n}$, and it has been shown that in the case of ``twisted'' periodic boundary conditions and when a parameter $\Delta$ of the chain, known as the anisotropy parameter, is set to the value $\Delta=-1/2$, the space $V$ will possess a subspace $U$ invariant under the action of the Hamiltonian, such that the restriction of the Hamiltonian to $U$ coincides (under an appropriate choice of basis) with the operator $H_n$ defined in \eqref{eq:def-h-inf-gen}; see \cite[Section 8]{mitra-etal} and \cite[Section 3.2.4]{zinn-justin}.

One additional way in which the probability vector $\boldsymbol{\mu}_n$ makes an appearance is in the study of  connectivity patterns associated with a different type of loop model known as the \textbf{fully packed loops} (FPLs). An FPL configuration of order $n$ is a subset of the edges of an $(n-1)\times (n-1)$ square lattice $[0,n-1]\times[0,n-1]$, to which are added  $2n$ of the $4n$ ``boundary'' edges connecting the square to the rest of the lattice $\Z^2$, by starting with the edge from $(0,0)$ to $(0,-1)$ and then taking alternating boundary edges as one goes around the boundary in a counter-clockwise direction (e.g., one would take edges incident to $(2,0), (4,0)$, etc.); these extra edges are referred to as \textbf{stubs}. The configuration is subject to the condition that any lattice vertex in the square is incident to exactly two of the configuration edges; see Fig.~\ref{fig:fpl}.

\begin{figure}
\begin{center}
\begin{tabular}{ccc}
\scalebox{0.6}{\includegraphics{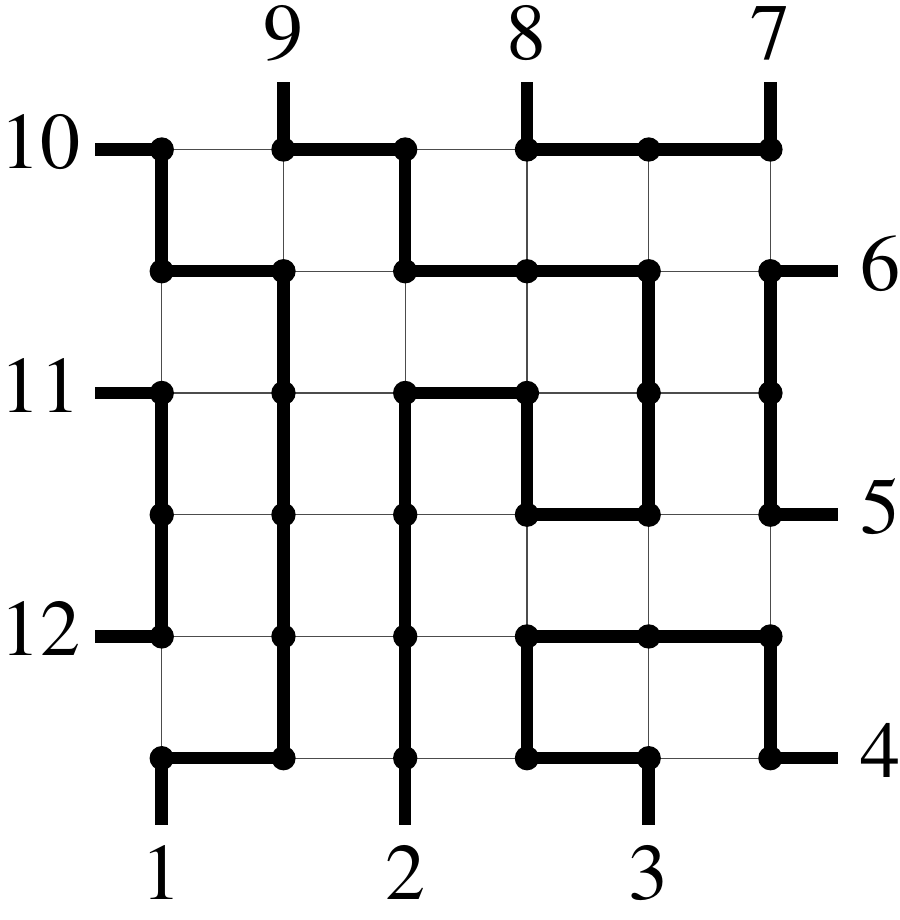}}
& &
\scalebox{0.65}{\includegraphics{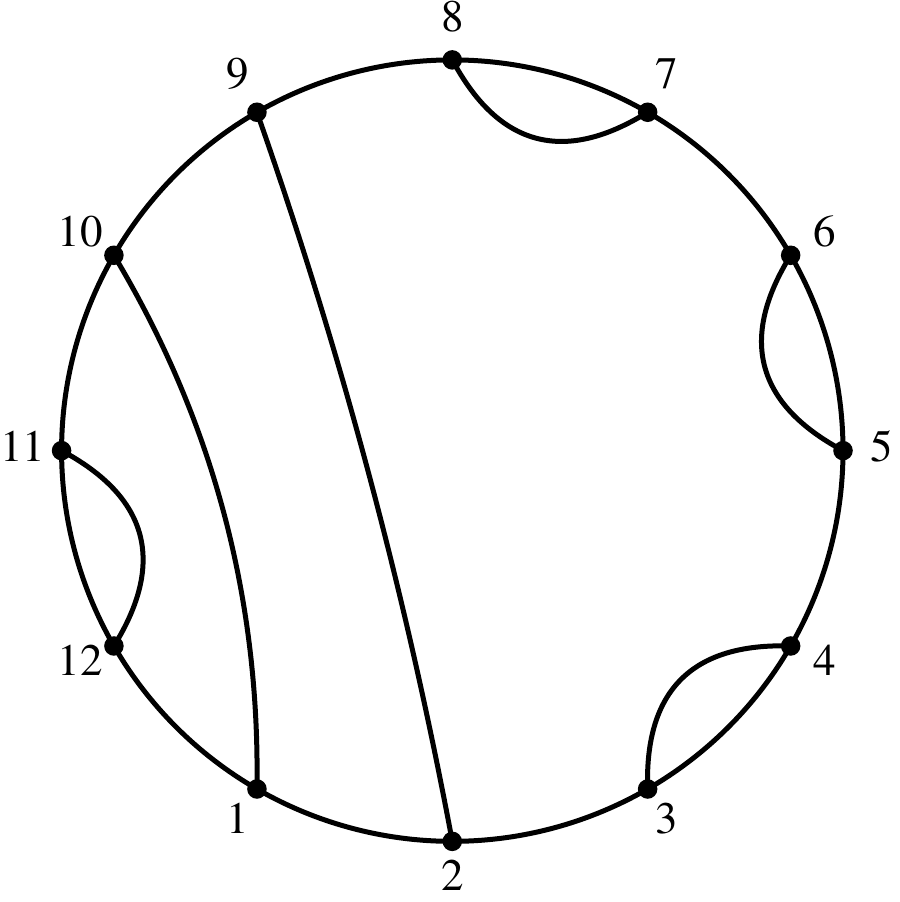}}
\end{tabular}
\caption{A fully packed loop configuration of order $6$ and its associated noncrossing matching.}
\label{fig:fpl}
\end{center}
\end{figure}

Let $\fpl_n$ denote the set of fully packed loop configurations of order $n$. A well-known bijection \cite{propp} shows that $\fpl_n$ is in correspondence with the set of \textbf{alternating sign matrices} (ASMs) of order $n$. These important combinatorial objects also have an interpretation as configurations of the \textbf{six-vertex model} (a.k.a.\ \textbf{square ice}) on an $n\times n$ lattice with prescribed boundary behavior known as the \textbf{domain wall boundary condition}. It was conjectured by Mills, Robbins and Rumsey \cite{mills-robbins-rumsey} and proved by Zeilberger \cite{zeilberger1} (see also \cite{bressoud}, \cite{kuperberg}) that the number $|\fpl_n|$ of ASMs of order $n$ is given by the famous sequence of numbers $1,2,7,42,429,\ldots$, defined by
$$ \asm(n) = \frac{1!4!7!\ldots (3n-2)!}{n!(n+1)!\ldots(2n-1)!}. $$
As the reader will see below, the function $\asm(n)$ will play a central role in our current investigation of connectivity patterns in loop percolation on a cylinder.

Label the $2n$ stubs around the square $[0,n-1]\times[0,n-1]$ by the numbers $1$ through $2n$, starting with the edge pointing down from $(0,0)$. From the definition of fully packed loop configurations we see that any such configuration induces a connectivity pattern on the stubs, which is a noncrossing matching in $\noncn{n}$, in an analogous manner to the way a loop percolation configuration on the semi-infinite cylinder does. For $\pi\in\noncn{n}$, denote by $A_n(\pi)$ the number of configurations in $\fpl_n$ whose connectivity pattern is equal to~$\pi$.

It was first observed by Batchelor, de Gier and Nienhuis \cite{batchelor-etal} that the coordinates of the vector $\boldsymbol{\mu}_n$ are related to the enumeration of alternating sign matrices. They conjectured the following result, which was later proved by Zinn-Justin and Di Francesco \cite{zinn-justin-di-francesco1, zinn-justin-di-francesco2}.

\begin{thm}[Di Francesco-Zinn-Justin]
\label{thm:sum-rules}
\begin{enumerate}
\item The numbers $\mu_\pi$, $(\pi\in\noncn{n})$ are all fractions of the form $\alpha_n(\pi)/\asm(n)$ where $\alpha_n(\pi)$ is an integer.
\item The minimal value of $\alpha_n(\pi)$ is $1$ and is attained for $\pi=\pi^n_{\textrm{min}}$, the ``minimal''\footnote{according to a certain partial ordering of noncrossing matchings that is discussed in Section~\ref{sec:wheel-polynomials}.} noncrossing matching consisting of $n$ nested arcs (Fig.~\ref{fig:minmax-matchings}(a)).
\item The maximal value of $\alpha_n(\pi)$ is $\asm(n-1)$ and is attained for $\pi=\pi^n_{\textrm{max}}$, the ``maximal'' noncrossing matching consisting of $n$ nearest-neighbor arcs (Fig.~\ref{fig:minmax-matchings}(b)).
\end{enumerate}
\end{thm}

\begin{figure}
\begin{center}
\begin{tabular}{ccc}
\scalebox{0.6}{\includegraphics{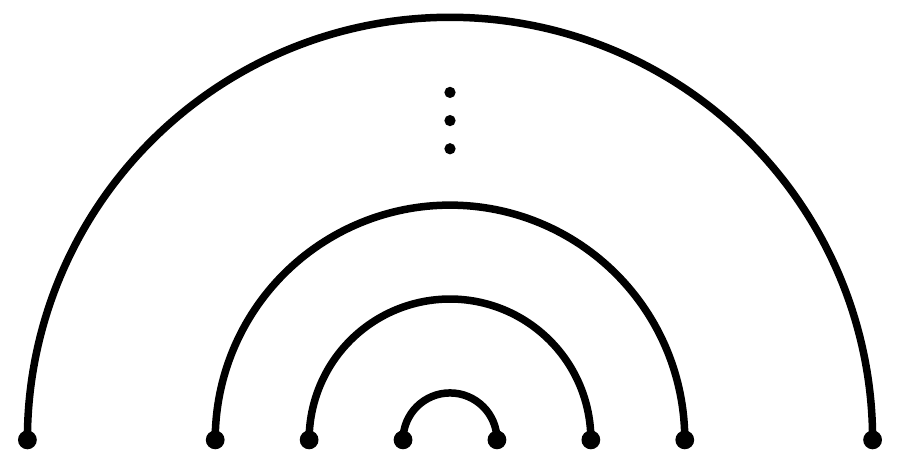}}
& &
\scalebox{0.6}{\includegraphics{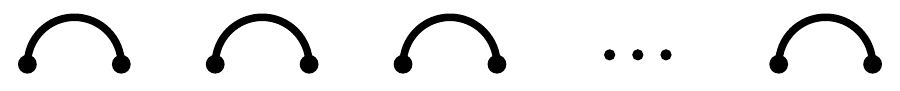}}
\\ (a) & & (b)
\end{tabular}
\caption{(a) The minimal matching $\pi^n_\textrm{min}$; (b) the maximal matching $\pi^n_\textrm{max}$.}
\label{fig:minmax-matchings}
\end{center}
\end{figure}

Shortly after this discovery, Razumov and Stroganov discovered a much more precise conjecture \cite{razumov-stroganov} about the connection between alternating sign matrices and the vector $\boldsymbol{\mu}_n$; it turns out that the correct thing to do is to look at the ASMs as fully packed loop configurations. Their conjecture, known for several years as the Razumov-Stroganov conjecture, was proved in 2010 by Cantini and Sportiello \cite{cantini-sportiello}.

\begin{thm}[Cantini-Sportiello-Razumov-Stroganov theorem]
For any $\pi\in\noncn{n}$ we have $\alpha_n(\pi)=A_n(\pi)$. That is,
$$
\mu_\pi = \frac{A_n(\pi)}{\asm(n)} \qquad (\pi\in\noncn{n}).
$$
\end{thm}

Probabilistically, this means that the vector $\boldsymbol{\mu}_n$ is realized as the distribution of the connectivity pattern of a \emph{uniformly random} FPL configuration of order $n$. Note that this is a finite probability space, whereas the Markov chain realization would require a potentially unbounded amount of randomness to generate a sample from $\boldsymbol{\mu}_n$.

Cantini and Sportiello's proof of the conjecture of Razumov and Stroganov is highly nontrivial and involves subtle combinatorial and linear algebraic arguments. Note that while the two (\emph{a posteriori} equivalent) definitions of $\boldsymbol{\mu}_n$ as the distribution of connectivity patterns in cylindrical loop percolation and uniformly random fully packed loop configurations are superficially similar, the two models are quite different. In particular, the cylindrical model has an obvious symmetry under rotations of the cylinder by an angle $2\pi/2n$, which induces a symmetry under rotation on $\boldsymbol{\mu}_n$. The fact that the distribution of the connectivity pattern of random fully packed loops---which are defined on a square, not circular, geometry---also has the same symmetry, is far from obvious, and its earlier proof by Wieland \cite{wieland} played an important role in Cantini and Sportiello's analysis.

In this subsection we surveyed several settings in which the probability vector $\boldsymbol{\mu}_n$ appears: as the distribution of the connectivity pattern of cylindrical loop percolation; as the stationary distribution of the Temperley-Lieb random walk (or, equivalently, the distribution of the connectivity pattern of a semi-infinite arrangement of connection diagrams associated with independent, uniformly random Temperley-Lieb operators); as the ground state of the XXZ spin chain; and as the distribution of the connectivity pattern of uniformly random fully packed loop configurations of order $n$. To conclude this discussion, we note that one of the minor results of this paper is an additional characterization of the connectivity probabilities $(\mu_\pi)_{\pi\in\noncn{n}}$ as coefficients in a certain natural linear-algebraic expansion; see Theorem~\ref{thm:lin-alg-expansion} in Subsection~\ref{sec:wheel-polynomials-background}.

\subsection{Loop percolation on a cylinder: connectivity events}

\label{sec:loop-perc-cylinder}

Having described the origins of the investigations into the random noncrossing matching $\pistarn$ and its distribution $\boldsymbol{\mu}_n$, we are ready to discuss the problem of computing explicitly the probabilities of various events. Our main interest will be with submatching events of the type $\{ \pi_0 \submatching \pistarn \}$, where for $\pi_0\in\noncn{k}$ and $\pi\in\noncn{n}$ ($n\ge k$) the notation $\pi_0 \submatching \pi$ means (as defined earlier when $\pi$ is an infinite matching) that $\pi_0$ is a submatching of $\pi$, i.e., that $\pi_{\raisebox{2pt}{\big|} [1,2k]} \equiv \pi_0$.

It was observed starting with numerical work of Mitra et al.\ \cite{mitra-etal}, Zuber \cite{zuber} and Wilson (unpublished work, cited in \cite{zuber}) that the probabilities of certain events had nice formulas as rational functions in $n$. For example, one has empirically the relations
\begin{align}
\prob\left(
\raisebox{-8pt}{\scalebox{0.1}{\includegraphics{event12}}}
\submatching \pistarn
\right) &= \frac{3}{2} \cdot \frac{n^2+1}{4n^2-1}, \hspace{115.0pt} (n\ge 1),
\label{eq:three-eights-finite-n} \\[5pt]
\prob\left(
\raisebox{-6pt}{\scalebox{0.2}{\includegraphics{event12-34}}}
\submatching \pistarn
\right) &= \frac18 \cdot \frac{97 n^6 + 82 n^4 - 107 n^2 - 792}{(4n^2-1)^2(4n^2-9)}
 \hspace{20.0pt} (n\ge 2),
 \label{eq:identity97}
\\
\prob\left(
\raisebox{-6pt}{\scalebox{0.2}{\includegraphics{event14-23}}}
\submatching \pistarn
\right) &= \frac1{16} \cdot \frac{59n^6 + 299n^4 +866 n^2 + 576}{(4n^2-1)^2(4n^2-9)}
 \hspace{10.0pt} (n\ge 2),
\end{align}

\vspace{-12.0pt}
\begin{align} 
& \prob  \left(
\raisebox{-6pt}{\scalebox{0.25}{\includegraphics{event12-34-56}}}
\submatching \pistarn
\right) \nonumber \\ &= 
\frac{\scriptstyle 1}{\scriptstyle 512}\cdot
\frac{\scriptscriptstyle 214093n^{12}-980692n^{10}-584436n^8-1887916n^6+1361443n^4 -17432892n^2-316353600}{\scriptstyle  (4n^2-1)^3(4n^2-9)^2(4n^2-25)}
\  (n\ge 3),
\label{eq:event12-34-56-prob}
\end{align}
and several other such formulas, which can be discovered by a bit of experimentation after programming the linear equations \eqref{eq:razumov-stroganov-linear-system} into a computer algebra system such as \texttt{Maple} or \texttt{Mathematica}.\footnote{Unfortunately the size of the system is the Catalan number $\cat(n)$ which grows exponentially with $n$, making it impractical to compute $\boldsymbol{\mu}_n$ for values of $n$ much greater than $n=9$. Zuber \cite{zuber} managed the computation up to $n=11$ by using rotational symmetry to reduce the order of the system.} The discovery of additional conjectural relations of this type is only difficult insofar as it taxes one's patience, programming skill and computational resources, but,  as Zuber remarks at the end of \cite{zuber}, ``\textit{More conjectural expressions have been collected for other types of configurations \ldots\ but this seems a gratuitous game in the absence of a guiding principle.}'' It therefore appeared sensible to wait for more theoretical developments before proceeding with the ``gratuitous game.'' Indeed, the first progress (and so far, to our knowledge, the only progress) on this front was made by Fonseca and Zinn-Justin \cite{fonseca-zinn-justin}, who, building on a theoretical framework developed earlier by Zinn-Justin and Di Francesco \cite{zinn-justin-di-francesco1, zinn-justin-di-francesco2} (see also \cite{difran-zj-zuber}) managed to prove explicit formulas for two classes of events.

\begin{thm}[Fonseca-Zinn-Justin \cite{fonseca-zinn-justin}]
\label{thm:fonseca-zinn-justin}
For each $k\ge 1$, Denote by $\anticluster_k^{(n)}$ the ``anti-cluster'' event that no two of the numbers $1,\ldots,k$ are matched under $\pistarn$. Denote by $B_k^{(n)}$ the event $\left\{ \pi^k_\textrm{min} \submatching \pistarn\right\}$ (the submatching event associated with the minimal matching $\pi^k_\textrm{min}$; see Fig.~\ref{fig:minmax-matchings}). Define a function
\begin{equation}
\label{eq:complicated-product-def}
R_k(n)
= \begin{cases}
\displaystyle \frac{\prod_{j=1}^{(k+1)/2} \prod_{m=j}^{2j-2} (n^2-m^2)}{\prod_{j=0}^{(k-3)/2} (4n^2-(2j+1)^2)^{(k-1)/2-j}}
& \textrm{$k$ odd}, \\[15pt]
\displaystyle \frac{\prod_{j=1}^{k/2} \prod_{m=j}^{2j-1} (n^2-m^2)}{\prod_{j=0}^{k/2-1} (4n^2-(2j+1)^2)^{k/2-j}}
& \textrm{$k$ even}.
\end{cases}
\end{equation}
Then we have
\begin{align}
\prob( \anticluster_k^{(n)} ) &= \frac{1}{\asm(k)} \frac{R_k(n)}{R_k(k)},
\label{eq:anti-cluster} \\
\prob( B_k^{(n)} ) &= \frac{1}{\asm(n)} \sum_{1\le a_1 < a_2 < \ldots < a_{n-k}} \det \left( \binom{j+k}{a_i-j+ k } \right)_{i,j=1}^{n-k}.
\label{eq:k-nested-arcs}
\end{align}
In particular, we have $\prob(\anticluster_2^{(n)})=\frac52 \frac{n^2-1}{4n^2-1} \vphantom{\begin{array}{c}a\\a\end{array}}$, which implies equation \eqref{eq:three-eights-finite-n} since the event in that identity is the complement of $\anticluster_2^{(n)}$.
\end{thm}

Note that the approach of Fonseca and Zinn-Justin to the relation \eqref{eq:anti-cluster} started out, similarly to \eqref{eq:k-nested-arcs}, by expressing the probabilities of these events as sums of determinants over a certain family of matrices with binomial coefficients. However, in the case of \eqref{eq:anti-cluster} a family of Pfaffian evaluations due to Krattenthaler \cite{krattenthaler} made it possible to evaluate the sum in closed form. As the authors of \cite{fonseca-zinn-justin} point out, both the formulas \eqref{eq:anti-cluster} and \eqref{eq:k-nested-arcs} have an interesting interpretation in terms of the enumeration of certain families of totally symmetric self-complementary plane partitions.

\subsection{Loop percolation on a cylinder: new results and the rationality phenomenon}

\label{sec:cylinder-new-results}

The identities \eqref{eq:three-eights-finite-n}--\eqref{eq:event12-34-56-prob} generalize in a straightforward way to a conjecture on the form of the dependence on $n$ of probabilities of submatching events. 

\begin{conj}[The rationality phenomenon for submatching events; finite $n$ case]
\label{conj:rationality-finite-n}
If $\pi_0\in\noncn{k}$, the probability of the submatching event $\left\{ \pi_0 \submatching \pistarn \right\}$ has the form of a rational function in $n$, specifically
\begin{equation} \label{eq:rationality-finite-n}
\prob\left( \pi_0 \submatching \pistarn \right) = \frac{Q_{\pi_0}(n)}{\prod_{j=1}^k (4n^2-j^2)^{k+1-j}} \qquad (n\ge k), 
\end{equation}
where $Q_{\pi_0}(n)$ is an even polynomial of degree $k(k+1)$ with dyadic rational coefficients. $Q_{\pi_0}$ can be computed by polynomial interpolation---see Algorithm~B in Appendix~A.
\end{conj}

In pursuit of an approach that would lead to a proof of Conjecture~\ref{conj:rationality-finite-n}, one of our goals has been to find a ``guiding principle'' of the type alluded to by Zuber, that would reveal the underlying structure behind the empirical phenomena described above. In Section~\ref{sec:wheel-polynomials} we will prove the following result, which can be thought of as one version of such a principle, and is our main result.

\begin{thm}[Explicit formula for submatching event probabilities]
\label{thm:explicit-formulas-finite-n}
For any $\pi_0 \in \noncn{k}$, there exists a multivariate polynomial $F_{\pi_0}(w_1,\ldots,w_k)$, computable by an explicit algorithm (Algorithm~E in Appendix~A), with the following properties:
\begin{enumerate}
\item $F_{\pi_0}$ has integer coefficients.
\item All the monomials in $F_{\pi_0}(w_1,\ldots,w_k)$ are of the form $\prod_{j=1}^k w_j^{2j-a_j}$ 
where $1\le a_1<\ldots<a_k$ are integers satisfying $a_j\le 2j-1$ for $1\le j\le k$.\footnote{Note that the number of different sequences $(a_1,\ldots,a_k)$ satisfying these conditions is $\cat(k)$, the $k$th Catalan number, and indeed the sequence $(a_1,\ldots, a_k)$ can be thought of as encoding a noncrossing matching in $\noncn{k}$, a fact that will have a role to play later on---see Section~\ref{sec:wheel-polynomials}.}
\item The probability $ \prob\left(\pi_0 \submatching \pistarn \right)$ is given for any $n\ge k+1$ by
\begin{align}
\nonumber
& \prob\left( \pi_0 \submatching \pistarn \right) = \frac{1}{\asm(n)} [z_1^0 z_2^2 z_3^4 \ldots z_n^{2n-2}] \Bigg( F_{\pi_0}(z_2,\ldots,z_{k+1})
\\ & \hspace{110pt} \times 
 \prod_{1\le i<j\le n} (z_j-z_i)(1+z_j+ z_i z_j) \prod_{j=k+2}^n (1+z_j) \Bigg),
 \label{eq:submatching-polycoeff-formula}
\end{align}
where $[z_1^{m_1} \ldots z_n^{m_n}]g(z_1,\ldots,z_n)$ denotes the coefficient of the monomial $z_1^{m_1} \ldots z_n^{m_n}$ in a polynomial $g(z_1,\ldots,z_n)$.
\end{enumerate}

\end{thm}

The formula \eqref{eq:submatching-polycoeff-formula} can be recast in two equivalent forms which some readers may find more helpful: first, as a constant term identity
\begin{align*}
\nonumber
& \prob\left( \pi_0 \submatching \pistarn \right) = \frac{1}{\asm(n)} \CT_{z_1,\ldots,z_n} \Bigg( F_{\pi_0}(z_2,\ldots,z_{k+1})
\!\!\!\!  
 \\ &  \hspace{130.0pt} \left. \times \prod_{1\le i<j\le n} \!\! \!(z_j-z_i)(1+z_j+ z_i z_j)\frac{\prod_{j=k+2}^n (1+z_j)}{
\prod_{j=1}^n z_j^{2j-2}} \right),
\end{align*}
where $\CT_{z_1,\ldots,z_n} g(z_1,\ldots,z_n)$ denotes the constant term of a Laurent polynomial $g(z_1,\ldots,z_n)$; and second, as a multi-dimensional complex contour integral
\begin{align*}
\nonumber
&\prob\left( \pi_0 \submatching \pistarn \right) = \frac{1}{\asm(n)} \oint\hspace{-2.0pt}\ldots\hspace{-2.0pt}\oint \Bigg( F_{\pi_0}(z_2,\ldots,z_{k+1})
 \\ & 
 \hspace{60.0pt}
 \times
 \!\!\!
\! \prod_{1\le i<j\le n} \!\!(z_j-z_i)(1+z_j+ z_i z_j)
 \frac{\prod_{j=k+2}^n (1+z_j)}{ \prod_{j=1}^n z_j^{2j-1}}
 \Bigg) \prod_{j=1}^n \frac{dz_j}{2\pi i},
\end{align*}
where the contour for each of the $z_j$'s is a circle of arbitrary radius around~$0$.

Table~\ref{table:submatching-polynomials} lists 
a few of the simplest submatching events and the multivariate polynomials associated to them. Note that the first case of the empty matching listed in the table (for which the probability of the submatching event is $1$) corresponds to the known identity
\begin{equation}
\label{eq:asm-const-term}
\asm(n) = [z_1^0 z_2^2 \ldots z_n^{2n-2}] \left( \prod_{1\le i<j\le n} (z_j-z_i)(1+z_j+z_i z_j) \prod_{j=2}^n (1+z_j) \right).
\end{equation}
This identity is already a difficult result. It was proved in \cite{zinn-justin-di-francesco2} by Zinn-Justin and Di Francesco, relying conditionally on a conjectural anti-symmetrization identity that they discovered which itself was proved shortly afterwards by Zeilberger \cite{zeilberger2}. Their proof also relies on a highly nontrivial Pfaffian evaluation due to Andrews \cite{andrews} (see also \cite{andrews-burge, krattenthaler}).  The following conjecture can be thought of as a natural generalization or ``deformation'' of \eqref{eq:asm-const-term}.

\begin{table}[h]
\begin{center}
\begin{tabular}{ccc}
\raisebox{40pt}{
\begin{tabular}{c|c|c}
k & $\pi_0$ & $\frac{F_{\pi_0}(w_1,\ldots,w_k)}{w_1\ldots w_{2k}}$ \\[3pt]
\hline & & \\[-1ex]
0 & empty
& 1 \\[8pt] 
\raisebox{8pt}{1} & \scalebox{0.14}{\includegraphics{event12}} & 
\raisebox{8pt}{1}
\\[10pt]
\raisebox{11pt}{2} & \scalebox{0.3}{\includegraphics{event12-34}} & 
\raisebox{11pt}{1}
\\[10pt]
\raisebox{11pt}{2} & \scalebox{0.3}{\includegraphics{event14-23}} & 
\raisebox{12pt}{$w_1$}
\end{tabular}
}
& &
\begin{tabular}{c|c|c}
k & $\pi_0$ & $\frac{F_{\pi_0}(w_1,\ldots,w_k)}{w_1\ldots w_{2k}}$ \\[3pt]
\hline & & \\
\raisebox{8pt}{3} &
\scalebox{0.3}{\includegraphics{event12-34-56}} & 
\raisebox{8pt}{1}
\\[10pt]
\raisebox{8pt}{3} &
\scalebox{0.3}{\includegraphics{event12-36-45}} & 
\raisebox{8pt}{$w_2-w_1 w_2$}
\\[10pt]
\raisebox{8pt}{3} &
\scalebox{0.3}{\includegraphics{event14-23-56}} & 
\raisebox{8pt}{$w_1$}
\\[10pt]
\raisebox{8pt}{3} &
\scalebox{0.3}{\includegraphics{event16-23-45}} & 
\raisebox{8pt}{$w_1 w_2$}
\\[10pt]
\raisebox{8pt}{3} &
\scalebox{0.3}{\includegraphics{event16-25-34}} & 
\raisebox{14pt}{$w_1 w_2^2$}
\end{tabular}
\end{tabular}

\vspace{10pt}
\caption{The polynomials $F_{\pi_0}$ associated with some submatching events, after factoring out the product $w_1\ldots w_{2k}$ which always divides $F_{\pi_0}$. (The first case of the ``empty'' matching refers to the trivial matching of order $k=0$, in which case the submatching event has probability $1$.)
}
\label{table:submatching-polynomials}
\end{center}
\end{table}

\vbox{
\begin{conj}[The algebraic rationality phenomenon] \label{conj:perturbed-asm-iden}
Let $k\ge1$, and let $1\le a_1<\ldots<a_k$ be integers satisfying $a_j\le 2j-1$ for all $j$. There exists a rational function of the form $R(n)=P(n)/\prod_{j=1}^k (4n^2-j^2)^{k+1-j}$, where $P(n)$ is an even polynomial of degree at most $k(k+1)/2$ with dyadic rational coefficients, such that for all $n\ge k+1$ we have that
\begin{align}
\nonumber
[ z_1^0 z_2^2 \ldots z_n^{2n-2} ]& \left(\prod_{j=1}^k z_{j+1}^{2j-a_j} \prod_{1\le i<j\le n} (z_j-z_i)(1+z_j+z_i z_j)
\prod_{j=k+2}^n (1+z_j) \right)
\\ & \ \ \ = \asm(n) R(n). \label{eq:perturbed-asm-iden}
\end{align}
\end{conj}
}

It is unclear whether the specific assumption about the sequence $a_1,\ldots,a_k$ that enters the form of the monomial $\prod_{j=1}^k z_{j+1}^{2j-a_j}$ is necessary to imply the rationality of $R(n)$ in \eqref{eq:perturbed-asm-iden}. The main assumption, which may be sufficient to imply the result, is that
we are looking at a Taylor coefficient that is ``a fixed distance away'' from the coefficient in \eqref{eq:asm-const-term}.

An immediate consequence of Theorem~\ref{thm:explicit-formulas-finite-n} is that Conjecture~\ref{conj:perturbed-asm-iden} essentially implies Conjecture~\ref{conj:rationality-finite-n}, although a small gap remains regarding the validity of \eqref{eq:rationality-finite-n} in the case $n=k$.

\begin{thm}
\label{thm:conj-implies-weak-conj}
Conjecture~\ref{conj:perturbed-asm-iden} implies a weaker version of Conjecture~\ref{conj:rationality-finite-n}
in which \eqref{eq:rationality-finite-n} is only claimed to hold for $n\ge k+1$.
\end{thm}

From the above discussion we see that, while we have been unable to get a complete understanding of the probabilities of submatching events, we have reduced the problem to the essentially algebraic question of understanding the form of the dependence of $n$ of the polynomial coefficients appearing on the left-hand side of \eqref{eq:perturbed-asm-iden}. Moreover, the algorithm for finding the polynomial $F_{\pi_0}$ associated with a submatching event, which will be explained in Section~\ref{sec:wheel-polynomials}, stems from a fairly detailed theoretical understanding of the model, and therefore already helps eliminate much of the mystery surrounding identities such as \eqref{eq:three-eights-finite-n}--\eqref{eq:event12-34-56-prob}.

It should be noted as well that the relations \eqref{eq:asm-const-term} and \eqref{eq:perturbed-asm-iden} belong to a large family of identities known as \textbf{constant term identities}. The study of such identities became popular following the discovery by Dyson of 
the identity
\begin{equation}
\label{eq:dyson}
\CT_{z_1,\ldots,z_n} \left( \prod_{1\le i\neq j \le n} \left(1-\frac{z_j}{z_i}\right)^{a_j} \right)
= \frac{(a_1+\ldots+a_n)!}{a_1! \ldots a_n!} \qquad (a_1,\ldots,a_n\ge 0),
\end{equation}
which became known as the Dyson conjecture \cite{dyson}. (Dyson's conjecture was proved by Gunson \cite{gunson} and Wilson \cite{wilson}, and a particularly simple proof was later found by Good \cite{good}.) Research on such identities has been an active area that involves a mixture of techniques from algebraic combinatorics and the theory of special functions. In particular, Sills and Zeilberger \cite{sills-zeilberger} proved a deformation of Dyson's identity in which a coefficient near the constant term is shown to equal the multinomial coefficient on the right-hand side of \eqref{eq:dyson} times a rational function in the exponents $a_1,\ldots,a_n$; this is quite similar in spirit to the claim of Conjecture~\ref{conj:perturbed-asm-iden}.

To conclude this section, we show how the relation \eqref{eq:identity97} can be derived in a simple manner from the results of Fonseca and Zinn-Justin, and in addition derive another identity concerning a finite connectivity event which is not a submatching event.

\begin{thm}
\label{thm:ninety-seven-finiten}
The identity \eqref{eq:identity97} holds for $n\ge 2$,\footnote{This proves part of Conjecture 9 in Zuber's paper \cite{zuber}.} and we have the additional identity
\begin{equation} \label{eq:prov-event12-45}
\prob\left(
\raisebox{-6pt}{\scalebox{0.2}{\includegraphics{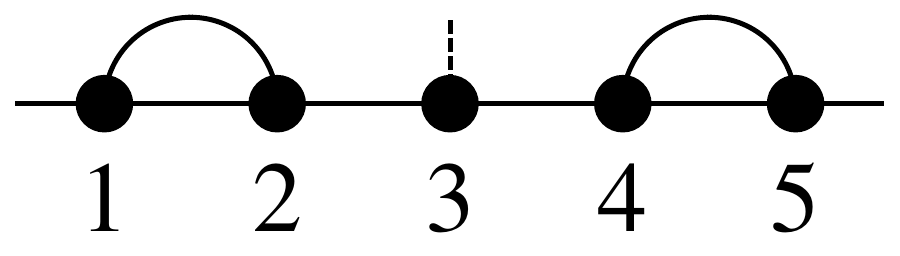}}}
\submatching \pistarn
\right) = \frac{15}{16} \cdot \frac{(n^2-4)(9n^4+38n^2-63)}{(4n^2-1)^2(4n^2-9)}
 \hspace{20.0pt} (n\ge 3)
\end{equation}
(we use the notation for submatching events for convenience, but note that this refers to the event that $\pistarn$ matches the two pairs $1\matched{\pistarn}2$ and $4\matched{\pistarn}5$).
\end{thm}

\begin{proof}
By \eqref{eq:anti-cluster}, we know that
\begin{equation} \label{eq:anti-cluster4}
\prob(\anticluster^{(n)}_4) = \frac{33}{8} \cdot \frac{(n^2-1)(n^2-4)(n^2-9)}{(4n^2-1)^2 (4n^2-9)}.
\end{equation}
On the other hand, by the inclusion-exclusion principle, we have that
\begin{align*}
\prob(\anticluster^{(n)}_4) &= 
1-
\prob\left(\pistarn \in 
\raisebox{-8pt}{\scalebox{0.1}{\includegraphics{event12}}}
\right)
-
\prob\left(\pistarn \in 
\raisebox{-8pt}{\scalebox{0.1}{\includegraphics{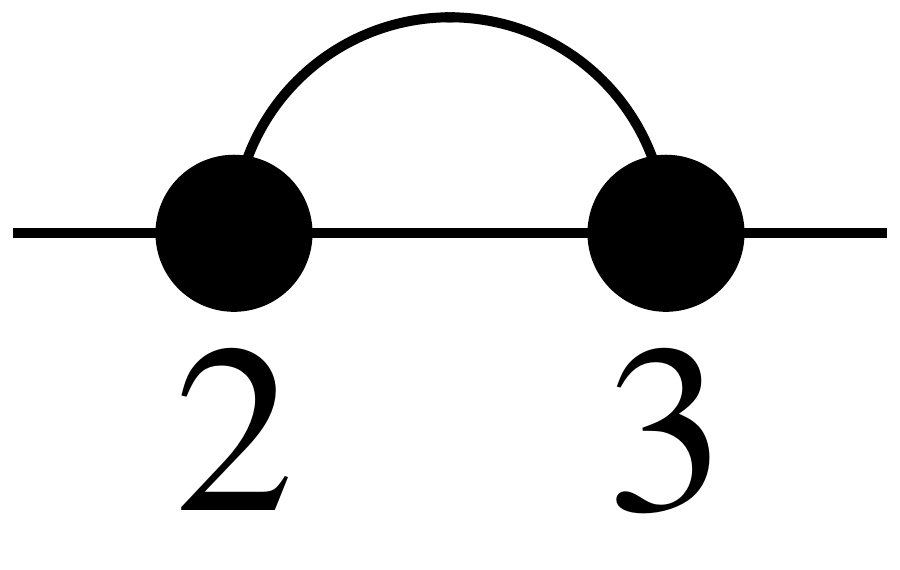}}}
\right)
-
\prob\left(\pistarn \in 
\raisebox{-8pt}{\scalebox{0.1}{\includegraphics{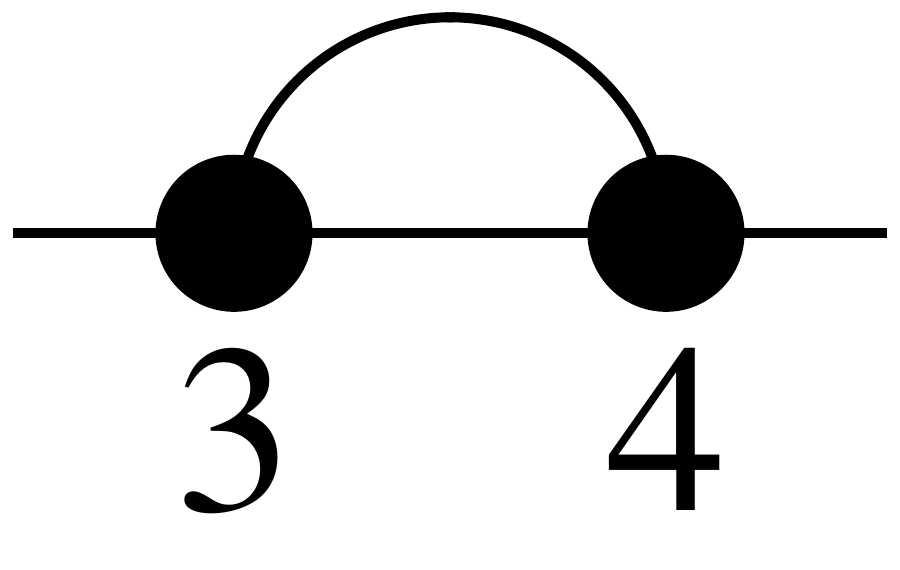}}}
\right)
\\ & \ \ \ + 
\prob\left(
\pistarn \in
\raisebox{-6pt}{\scalebox{0.2}{\includegraphics{event12-34}}}
\right)
\\ &
= 1 - 3\cdot \frac{3}{2}\cdot\frac{n^2+1}{4n^2-1}
+ 
\prob\left(
\pistarn \in
\raisebox{-6pt}{\scalebox{0.2}{\includegraphics{event12-34}}}
\right)
\end{align*}
(using \eqref{eq:three-eights-finite-n} and the rotation-invariance of $\pistarn$).
Substituting the probability from \eqref{eq:anti-cluster4} and solving for $\prob\left(
\pistarn \in
\raisebox{-6pt}{\scalebox{0.2}{\includegraphics{event12-34}}}
\right)$ gives \eqref{eq:identity97}. For the second identity \eqref{eq:prov-event12-45}, perform a similar inclusion-exclusion computation for the event $\anticluster^{(n)}_5$, whose probability is given according to \eqref{eq:anti-cluster} by
$$
\prob(\anticluster^{(n)}_5) = \frac{11}{16} \cdot \frac{(n^2-1)(n^2-4)(n^2-9)}{(4n^2-1)^2 (4n^2-9)}.
$$
The details of the computation are easy and left to the reader.
\end{proof}

\subsection{Consequences for loop percolation on a half-plane}

\label{sec:consequences-halfplane}

Let us return to the original setting of loop percolation on a half-plane discussed earlier. The following result allows us to deduce exact results on probabilities of local connectivity events in the half-plane from corresponding results in the cylindrical model.

\begin{lem}
Let $A \subset \Z\times \Z$ be a finite set, and let $E_n$ denote the finite connectivity events
$$ E_n = \bigcap_{(j,k)\in A} \left\{ j\matched{\pistarn} k \right\} $$
associated with $A$ (which are defined for large enough $n$). Then we have $\prob(E_n)\to \prob(E)$ as $n\to\infty$, where $E = \bigcap_{(j,k)\in A} \left\{ j\matched{\pistar} k \right\}$.
\end{lem}

\begin{proof} Couple the cylindrical and half-planar loop percolation by using the same random bits to select the edge configuration incident to the vertices in the region $V(\lpgraph_n)$ (defined in \eqref{eq:v-lpgraph-n}).
With this coupling, it is easy to see that the symmetric difference $E_n \triangle E$ of the events $E_n$ and $E$ is contained in the event $B_n$ that one of the half-planar loop percolation paths starting at the point $(j,-j)$ for some $j$ belonging to one of the pairs $(j,k)\in A$ reaches the complement $\Z^2\setminus V(\lpgraph_n)$. Since the loop percolation paths are almost surely finite, we have $\lim_{n\to\infty} \prob(B_n) = \prob\left(\bigcap_{n=1}^\infty B_n\right) = 0$, and therefore we get that
$|\prob(E_n)-\prob(E)|\le \prob(E_n \triangle E) \le \prob(B_n) \to 0$ as $n\to\infty$.
\end{proof}

As immediate corollaries from the lemma we get the following facts: first, Conjecture~\ref{conj:rationality-finite-n} (or the weaker version of it mentioned in Theorem~\ref{thm:conj-implies-weak-conj}) implies Conjecture~\ref{conj:rationality}. Second, Theorem~\ref{thm:two-explicit-cases} follows as a limiting case of  \eqref{eq:three-eights-finite-n} and \eqref{eq:identity97}. Third, we have the following explicit formulas for events in the half-plane model corresponding to \eqref{eq:anti-cluster} and \eqref{eq:prov-event12-45}.

\begin{thm}
\label{thm:event135}
For $k\ge 1$, let $\anticluster_k$ denote the anti-cluster event that no two of the numbers $1,\ldots,k$ are matched under $\pistar$. We have the formulas
$$
\prob\left(\anticluster_k \right) = \frac{1}{2^{\lfloor k/2\rfloor\cdot \lfloor k/2+1\rfloor}\asm(k) R_k(k)},
$$
where $R_k(\cdot)$ is defined in \eqref{eq:complicated-product-def} (see Table~\ref{table:anti-cluster}), and 
$$
\prob\left(
\raisebox{-6pt}{\scalebox{0.2}{\includegraphics{event12-45}}}
\submatching \pistar \right) = \frac{135}{1024}.
$$
\end{thm}

\begin{table}[h]
$$
\begin{array}{c|c|c|c|c|c|c|c}
k & 1&2&3&4&5&6&7 \\ \hline &&&&&&&\\[-1.8ex]
\prob\left(\pistar\in\anticluster_k \right) &
\displaystyle \frac58 & 
\displaystyle \frac14 & 
\displaystyle \frac{33}{512} & 
\displaystyle \frac{11}{1024} & 
\displaystyle \frac{2431}{2^{21}} & 
\displaystyle \frac{85}{2^{20}} & 
\displaystyle \frac{126293}{2^{35}}
\displaystyle \end{array}
$$
\caption{Probabilities of the anti-cluster event $\anticluster_k$ for $k=1,\ldots,7$.}
\label{table:anti-cluster}
\end{table}

\section{The theory of wheel polynomials and the qKZ equation}

\label{sec:wheel-polynomials}

Our goal in this section is to prove Theorem~\ref{thm:explicit-formulas-finite-n}. The proof will be based on an extension of an algebraic theory that was developed in a recent series of papers \cite{fonseca-zinn-justin, zinn-justin, zinn-justin-di-francesco1, zinn-justin-di-francesco2}, where it is shown that a tool from the statistical physics literature known as the \textbf{quantum Knizhnik-Zamolodchikov equation} (or \textbf{qKZ equation}) can be applied to the study of the connectivity pattern of cylindrical loop percolation. 

Mathematically, the qKZ equation as applied to the present setting reduces to the analysis of a vector space of multivariate polynomials satisfying a condition known as the \textbf{wheel condition}. (Similar and more general wheel conditions are also discussed in the papers \cite{feigin-etal, kasatani, pasquier}.)
We call such polynomials \textbf{wheel polynomials}. 
In Subsections~\ref{sec:wheel-polynomials-background}--\ref{sec:p-nested-matchings} we will survey the known results regarding the theory of wheel polynomials that are needed for our purposes. In Subsection~\ref{sec:gen-submatching} we present a new result (Theorem~\ref{thm:submatching-event-expansion}) on expansions of families of wheel polynomials associated with submatching events, that is another main result of the paper. In Subsection~\ref{sec:proof-main-thm} we show how to derive Theorem~\ref{thm:explicit-formulas-finite-n} from 
Theorem~\ref{thm:submatching-event-expansion}.

\subsection{Background}

\label{sec:wheel-polynomials-background}

Let $\noncn{n}$ denote as before the set of noncrossing matchings of $1,\ldots,2n$. It is well-known that elements of $\noncn{n}$ are in canonical bijection with the set of \textbf{Dyck paths} of length $2n$, and with the set of Young diagrams contained in the staircase shape $(n-1,n-2,\ldots,1)$. These bijections will play an important role, and are illustrated in Fig.~\ref{fig:noncn-bijections}; see \cite[Section 2]{fonseca-zinn-justin} for a detailed explanation.

\begin{figure}
\begin{center}
\begin{tabular}{cc}
\scalebox{0.7}{\includegraphics{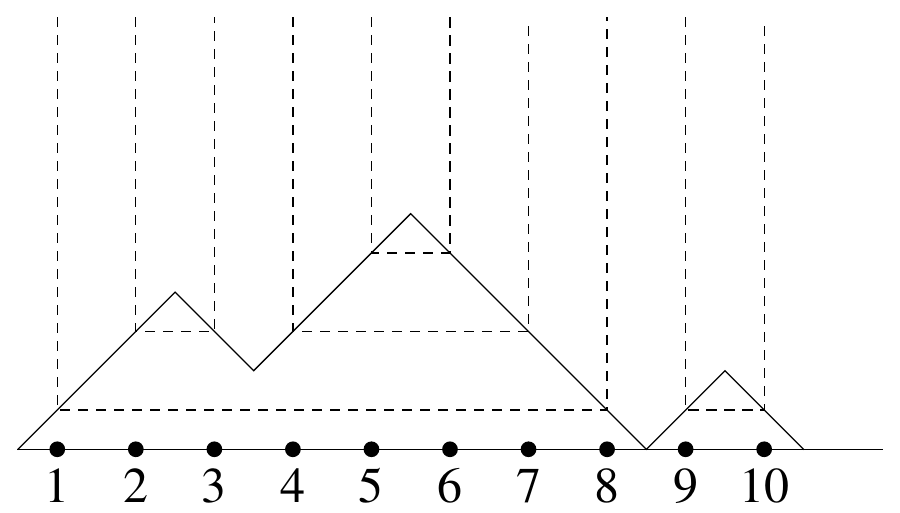}} &
\scalebox{0.7}{\includegraphics{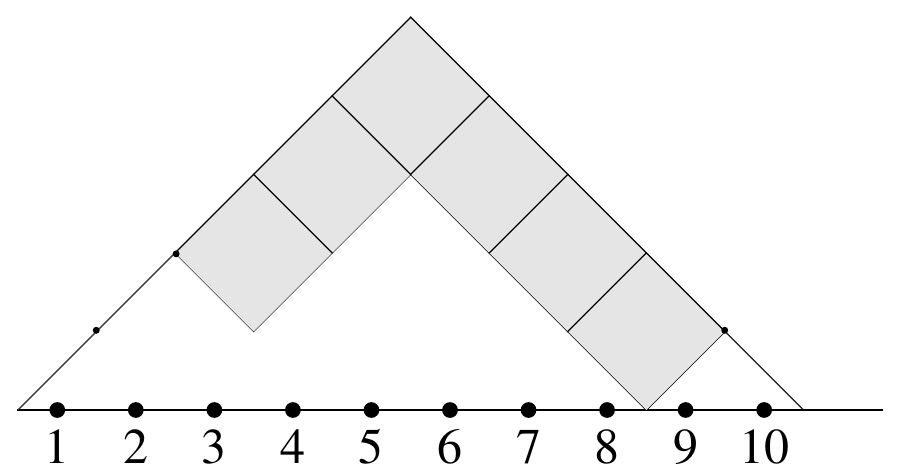}}

\end{tabular}
\caption{The noncrossing matching from Fig.~\ref{fig:example-noncrossing-matching}(a) represented as a Dyck path and as a Young diagram.}
\label{fig:noncn-bijections}
\end{center}
\end{figure}

We adopt the following notation related to these bijections. If $\pi\in\noncn{n}$, let $\pi_k$ be $1$ if $\pi(k)>k$ (``$k$ is matched to the right'') or $-1$ if $\pi(k)<k$ (``$k$ is matched to the left''). The vector $(\pi_1,\ldots,\pi_{2n})$ is the encoding of $\pi$ as the sequence of steps in the associated Dyck path. Let $\pi^+$ denote the sequence $(a_1,\ldots,a_n)$ of positions where $\pi_k=1$ (the positions of increase of the Dyck path). Encoding $\pi$ in this way maps $\noncn{n}$ bijectively onto the set
\begin{align}
\noncn{n}^+ &= \big\{ \mathbf{a}=(a_1,\ldots,a_n)\,:\, 1\le a_1<\ldots< a_n \le 2n-1
\nonumber\\ & \hspace{160pt} \textrm{and }a_j\le 2j-1 \textrm{ for all }j \big\}.
\label{eq:noncn-plus-def}
\end{align}

The Young diagram associated with a noncrossing matching $\pi\in\noncn{n}$ is denoted~$\lambda_\pi$. For $\pi,\sigma\in\noncn{n}$, denote $\sigma\nearrow \pi$ if $\lambda_\sigma$ is obtained from $\lambda_\pi$ by the addition of a single box, and $\pi \preceq \sigma$ if $\lambda_\pi$ is contained in $\lambda_{\sigma}$. (This partial order on $\noncn{n}$ is precisely the order with respect to which the matchings $\pi^n_{\textrm{min}}$ and $\pi^n_{\textrm{max}}$ defined in Subsection~\ref{sec:cylinder-background} are minimal and maximal, respectively.)
Denote $\pi\nearrow_j \sigma$ if $\lambda_\sigma$ is obtained from $\lambda_\pi$ by adding a box in a position that, in the coordinate system of Fig.~\ref{fig:noncn-bijections}, lies vertically above the positions $j$ and $j+1$ on the horizontal axis.

\begin{defi}[Wheel polynomials]
Let $q=e^{2\pi i/3}$.
A \textbf{wheel polynomial of order $n$} is a polynomial $p(\mathbf{z}) = p(z_1,\ldots,z_{2n})$ having the following properties:
\begin{enumerate}
\item $p$ is a homogeneous polynomial of total degree $n(n-1)$;
\item $p$ satisfies the \textbf{wheel condition}
\begin{equation}
\label{eq:wheel-condition}
p(z_1,\ldots,z_{2n})_{\big| z_k = q^2 z_j = q^4 z_i} = 0 \qquad (1\le i<j<k\le 2n).
\end{equation}
\end{enumerate}
Denote by $\wheelpoly{n}$ the vector space of wheel polynomials of order $n$.
\end{defi}

It is easy to see that the condition \eqref{eq:wheel-condition} depends only on the cyclical ordering of the variables $z_1,\ldots,z_{2n}$ (in other words, the space $\wheelpoly{n}$ is invariant under the rotation action $p(z_1,\ldots,z_{2n})\mapsto p(z_2,\ldots,z_{2n},z_1)$); hence the name ``wheel polynomials.''

The algebraic properties of wheel polynomials are quite elegant and of direct relevance to our study of connectivity patterns in loop percolation. Below we survey some of the known theory. 

\begin{example}
\label{example:minimal-wheelpoly}
Let $p(\mathbf{z})=\prod_{1\le i<j\le n} \left(qz_i - q^{-1} z_j \right) \prod_{n+1 \le i<j \le 2n} \left(qz_i - q^{-1} z_j\right)$. It is easy to check that $p$ is a wheel polynomial. Indeed, it is homogeneous of the correct degree, and if $1\le i<j<k\le 2n$ and we make the substitution $z_k=q^2 z_j = q^4 z_i$, then, if $j\le n$ there will be a zero factor in the first product $\prod_{1\le i<j\le n} \left(qz_i - q^{-1} z_j \right)$, and similarly if $j>n$ then there will be a zero factor in the second product. In either case the wheel condition \eqref{eq:wheel-condition} is satisfied.
\end{example}

Two types of pointwise evaluations of a wheel polynomial will be of particular interest: at the point $\mathbf{z}=(1,1,\ldots,1)$ and at a point $\mathbf{z}=(q^{-\pi_1},\ldots,q^{-\pi_{2n}})$ associated with the steps of a Dyck path, where $\pi\in\noncn{n}$. As a shorthand, we denote $p(\mathbf{1})=p(1,\ldots,1)$ and $p(\pi)=p(q^{-\pi_1},\ldots,q^{-\pi_{2n}})$ and refer to these values as the \textbf{$\mathbf{1}$-evaluation} and \textbf{$\pi$-evaluation} of $p$, respectively.

\begin{thm}
\label{thm:sufficient-evaluations}
A polynomial $p\in \wheelpoly{n}$ is determined uniquely by its values $p(\pi)$ as $\pi$ ranges over the noncrossing matchings in $\noncn{n}$. Equivalently, the $\pi$-evaluation linear functionals $(\operatorname{ev}_\pi)_{\pi\in \noncn{n}}$ defined by $\operatorname{ev}_\pi(p)= p(\pi)$ span the dual vector space $(\wheelpoly{n})^*$.
\end{thm}

\begin{proof} See \cite[Appendix C]{fonseca-zinn-justin-doubly}.
\end{proof}

In particular, it follows that $\dim \wheelpoly{n}\le |\noncn{n}|=\cat(n)$. We shall soon see that this is an equality, that is, the evaluation functionals $(\operatorname{ev}_\pi)_{\pi\in \noncn{n}}$ are in fact a basis of $(\wheelpoly{n})^*$. The proof of this fact involves a remarkable explicit construction of a family of wheel polynomials that will turn out to be a basis of $\wheelpoly{n}$ dual to $(\operatorname{ev}_\pi)_{\pi\in \noncn{n}}$, and will play a central role in our analysis.

For a multivariate polynomial $p\in \C[z_1,\ldots,z_m]$ and $1\le j< m$, define the divided difference operator $\partial_j:\C[z_1,\ldots,z_m]\to \C[z_1,\ldots,z_m]$ by
$$
(\partial_j p)(z_1,\ldots,z_m) =
\frac{p(z_1,\ldots,z_{j+1},z_j,\ldots,z_{2n})-
p(z_1,\ldots,z_j,z_{j+1},\ldots,z_{2n})}{z_{j+1}-z_j}.
$$
The next two lemmas are easy to check; see \cite[Sec.~4.2.1]{zinn-justin} for a related discussion.

\begin{lem}
\label{lem:divided-difference-wheelpoly}
If $p\in \wheelpoly{n}$ and $1\le j\le 2n-1$, then $(q z_j - q^{-1} z_{j+1}) \partial_j p$ is also in $\wheelpoly{n}$.
\end{lem}

\begin{lem} 
If for some $1\le j\le 2n-1$, $\pi\in\noncn{n}$ satisfies $j\matched{\pi}j+1$, then the set of preimages $e_j^{-1}(\pi)$ of $\pi$ under $e_j$ (the Temperley-Lieb operator defined in \eqref{eq:temperley-lieb-gens}) consists of: 
\begin{enumerate}
\item a noncrossing matching $\hat{\pi}\in \noncn{n}$ satisfying $\pi\nearrow_j \hat{\pi}$, if such a matching exists (i.e., if the associated Young diagram is still contained in the staircase shape $(n-1,\ldots,1)$); and
\item a set of noncrossing matchings $\sigma\in \noncn{n}$ satisfying $\sigma\preceq \pi$.
\end{enumerate}
\end{lem}

The last two lemmas make it possible to construct a family of wheel polynomials indexed by noncrossing matchings in $\noncn{n}$ (or, equivalently, by Young diagrams contained in the staircase shape $(n-1,\ldots,1)$) recursively. We start with an explicit polynomial known to be in $\wheelpoly{n}$---a scalar multiple of the polynomial from Example~\ref{example:minimal-wheelpoly} above---which we associate with the minimal matching $\pi^n_{\textrm{min}}$ (which corresponds to the empty Young diagram). We then define for each noncrossing matching $\pi$ a polynomial obtained from the polynomials of matchings preceding $\pi$ in the order $\preceq$ using linear combinations and the operation of Lemma~\ref{lem:divided-difference-wheelpoly}. The precise definition is as follows.

\begin{defi}[qKZ basis]
The \textbf{qKZ polynomials} are a family of polynomials $(\Psi_\pi)_{\pi\in\noncn{n}}$ defined using the following recursion on Young diagrams:
\begin{align}
\Psi_{\pi^n_{\textrm{min}}}(\mathbf{z}) &= (-3)^{-\binom{n}{2}} \prod_{1\le i<j\le n} \left(qz_i - q^{-1} z_j \right) \prod_{n+1 \le i<j \le 2n} \left(qz_i - q^{-1} z_j\right), \label{eq:qkz1} \\
\Psi_\pi(\mathbf{z}) &= \left(qz_j - q^{-1} z_{j+1}\right) \partial_j \Psi_{\sigma} -
\sum_{\nu \in e_j^{-1}(\sigma) \setminus \{\pi,\sigma\}} \Psi_\nu
\ \ \ \textnormal{ if }\sigma \nearrow_j \pi.
\label{eq:qkz2}
\end{align}
\end{defi}

Since for a given noncrossing matching $\pi$, the choices of $\sigma$ and $j$ such that $\sigma\nearrow_j \pi$ are in general not unique, it is not clear that the above definition makes sense. The fact that it does is a nontrivial statement, which we include as part of the next result.

\begin{thm} 
\label{thm:qkz-poly-properties}
The qKZ polynomials satisfy:
\begin{enumerate}
\item $\Psi_\pi$ is well-defined, i.e., the result of the recursive computation does not depend on the order in which boxes are added to the Young diagram.
\item $\Psi_\pi \in \wheelpoly{n}$.
\item Let $\rho(\pi)$ denote the rotation operator acting on $\pi$, defined by $(\rho (\pi))(k)=\pi(k+1)$ for $1\le k<2n$ and $(\rho (\pi))(2n)=\pi(1)$. Then we have
\begin{equation}
\Psi_{\rho(\pi)}(z_1,\ldots,z_{2n}) = \Psi_\pi(z_2,\ldots,z_{2n},z_1).
\label{eq:qkz3}
\end{equation}
\end{enumerate}
\end{thm}

\begin{proof}
See \cite[Section 4.2]{zinn-justin}.
\end{proof}

As explained in \cite[Sections 4.1--4.2]{zinn-justin}, equations \eqref{eq:qkz1}--\eqref{eq:qkz3} together are equivalent to a different system of equations which forms (a special case of) the qKZ equation. A different way of solving the same system, which we will not use here, is described in \cite{de-gier-lascoux-sorrell}.

If $\pi\in\noncn{n}$ and $1\le j\le 2n-1$ is a number such that $j\matched{\pi}j+1$, denote by $\hat{\pi}_j$ the noncrossing matching in $\noncn{n-1}$ obtained by deleting the arc connecting $j$ and $j+1$ from the diagram of the matching and relabelling the remaining elements. This operation makes it possible to perform many computations recursively. The next result provides an important example.

\begin{thm}
\label{thm:qkz-eval-recursion}
If $j\matched{\pi}j+1\in \noncn{n}$ then for any $\sigma\in \noncn{n}$ we have
\begin{equation}
\label{eq:qkz-eval-recursion}
\Psi_\pi(\sigma) = \begin{cases} 3^{n-1} \Psi_{\hat{\pi}_j}(\hat{\sigma}_j)
& \textnormal{if }j\matched{\sigma}j+1, \\ 0 & \textnormal{otherwise}.
\end{cases}
\end{equation}
\end{thm}

\begin{proof}
See \cite[Section 4.2]{zinn-justin}.
\end{proof}

\begin{cor}
\begin{enumerate}
\item For all $\pi,\sigma\in\noncn{n}$, we have
\begin{equation} \label{eq:psi-sigma-evaluations}
\Psi_\pi(\sigma) = \delta_{\pi,\sigma} = \begin{cases}1 & \textrm{if }\pi=\sigma, \\ 0 & \textrm{otherwise}. \end{cases}
\end{equation}
\item The qKZ polynomials $(\Psi_\pi)_{\pi\in\noncn{n}}$ are a basis for $\wheelpoly{n}$.
\item The $\pi$-evaluation linear functionals $(\operatorname{ev}_\pi)_{\pi\in\noncn{n}}$ are the basis of $(\wheelpoly{n})^*$ dual to $(\Psi_\pi)_{\pi\in\noncn{n}}$.
\item $\dim \wheelpoly{n} = \cat(n)$.
\end{enumerate}
\end{cor}

\begin{proof}
Part 1 follows by induction from \eqref{eq:qkz-eval-recursion}, and parts 2--4 are immediate from part 1.
\end{proof}

\begin{example}[{\cite[Section 4.3.1]{zinn-justin}}]
Let $p(\mathbf{z})=s_\lambda(z_1,\ldots,z_{2n})$, the Schur polynomial associated with the Young diagram $\lambda=(n-1,n-1,\ldots,2,2,1,1)$. That is, we have explicitly
\begin{equation} \label{eq:schur-double-staircase}
p(\mathbf{z}) = \left( \prod_{1\le i<j\le 2n} (z_j-z_i) \right)^{-1} \det \left( z_i^j \right)_{
\!\!\!\begin{array}{l} \scriptstyle 1\le i\le 2n, \\[-1ex]
\scriptstyle 0\le j\le 3n-2, \ j\equiv 0,1\textrm{ (mod $3$)}
\end{array}
}
\end{equation}
Then $p$ is homogeneous of degree $n(n-1)$, and the matrix whose determinant appears in \eqref{eq:schur-double-staircase} has the property that if we make the substitution $z_k=q^2 z_j = q^4 z_i$ for some $1\le i<j<k\le 2n$, the columns with index $i,j,k$ of the matrix become linearly dependent (take the linear combination with coefficients $1,q^2,q^4$), causing $p$ to be $0$. Thus, $p$ is a wheel polynomial.

Since $p$ is also a symmetric polynomial, it satisfies $p(\pi)=p(\pi')$ for any $\pi,\pi'\in \noncn{n}$. By Theorem~\ref{thm:sufficient-evaluations} it follows that $p$ is the unique symmetric wheel polynomial up to scalar multiplication. Note that this Schur function is also known to be the partition function of the six-vertex model with domain wall boundary condition at the ``combinatorial point,'' and played an important role in the study of the enumeration of alternating sign matrices; see \cite{okada, stroganov} and \cite[Section 2.5.6]{zinn-justin}.
\end{example}

\bigskip
Denote $\psi_\pi = \Psi_\pi(\mathbf{1})$ and $\boldsymbol{\psi}_n = (\psi_\pi)_{\pi\in\noncn{n}}$. The vector $\boldsymbol{\psi}_n$ is related to the connectivity pattern $\pistarn$ in cylindrical loop percolation, as explained in the following result.

\begin{thm} 
\label{thm:psi-mu-related}
The vector $\boldsymbol{\psi}_n$ is a solution of the linear system \eqref{eq:razumov-stroganov-linear-system}.
Consequently it is a scalar multiple of the probability distribution $\boldsymbol{\mu}_n$ of $\pistarn$. That is, for any $\pi\in\noncn{n}$ we have $ \psi_\pi = a_n \mu_\pi $, where $a_n = \sum_{\pi\in\noncn{n}} \psi_\pi$.
\end{thm}

\begin{proof}
See \cite[Sec.~4.3]{zinn-justin}.
\end{proof}

We will see later (see Theorem~\ref{thm:sum-rule} below) that the normalization constant $a_n$ is equal to $\asm(n)$. Assuming this fact temporarily, Theorem~\ref{thm:psi-mu-related} implies the following interesting and purely algebraic characterization of the numbers $\mu_\pi = \prob(\pistarn = \pi)$, which does not seem to have been noted before.

\begin{thm} 
\label{thm:lin-alg-expansion}
For any wheel polynomial $p\in\wheelpoly{n}$, we have that $p(\mathbf{1}) = \asm(n) \sum_{\pi\in\noncn{n}} \mu_\pi p(\pi)$. In other words, the $\mathbf{1}$-evaluation functional $\operatorname{ev}_\mathbf{1}$ has the expansion
$$ \operatorname{ev}_\mathbf{1} = \asm(n) \sum_{\pi\in\noncn{n}} \mu_\pi \operatorname{ev}_\pi $$
as a linear combination of the $\pi$-evaluation functionals.
\end{thm}

\begin{proof}
If $p=\sum_{\pi\in\noncn{n}} b_\pi \Psi_\pi$ gives the expansion of $p$ in terms of the qKZ basis polynomials, then
$$ p(\mathbf{1}) = \sum_{\pi\in\noncn{n}} b_\pi \Psi_\pi(\mathbf{1}) = 
\sum_{\pi\in\noncn{n}} b_\pi \psi_\pi =
\asm(n) \sum_{\pi\in\noncn{n}} \mu_\pi b_\pi.
$$
On the other hand, taking the $\pi$-evaluation of the expansion $p=\sum_{\sigma\in\noncn{n}} b_\sigma \Psi_\sigma$ easily implies using \eqref{eq:psi-sigma-evaluations} that $b_\pi = p(\pi)$.
\end{proof}

A referee has pointed out to us that Theorem~\ref{thm:lin-alg-expansion} can also be deduced from some of the results (specifically, Theorem~2 and equations (7) and (8)) of \cite{de-gier-lascoux-sorrell}.

Next, we construct a second important family of wheel polynomials defined using contour integration.

\begin{defi}[Integral wheel polynomials]
Denote 
\begin{align*}
\mathcal{A}_n &= \big\{ \mathbf{a}=(a_1,\ldots,a_n)\,:\, 1\le a_1\le \ldots \le a_n \le 2n-1
\\ & \hspace{160pt} \textrm{and }a_j\le 2j-1 \textrm{ for all }j \big\}.
\end{align*}
(These sequences generalize the steps of increase of Dyck paths---compare with \eqref{eq:noncn-plus-def}.)
For a sequence $\mathbf{a}=(a_1,\ldots,a_n)\in\mathcal{A}_n$, denote
\begin{align}
&\Phi_{\mathbf{a}}(z_1,\ldots,z_{2n}) =
(-1)^{\binom{n}{2}} \prod_{1\le j<k\le 2n} (qz_j-q^{-1} z_k) 
\nonumber \\ & \hspace{30.0pt} \times \overbrace{\mathlarger{\oint}\ldots\mathlarger{\oint}}^{n}
\frac{\displaystyle \prod_{1\le j<k\le n} (w_k-w_j)(qw_j-q^{-1} w_k)}{\displaystyle \prod_{j=1}^n \left(
  \prod_{k=1}^{a_j} (w_j-z_k) \prod_{k=a_j+1}^{2n} (qw_j-q^{-1} z_k)
\right)} \prod_{j=1}^n \frac{dw_j}{2\pi i},
\label{eq:phi-a-def}
\end{align}
where the contour of integration for each variable $w_j$ surrounds in a counterclockwise direction the singularities at $z_k\ (1\le k\le a_j)$, but does not surround the singularities at $q^{-2} z_k\ (a_j+1\le k\le 2n)$. 
\end{defi}

We will be mostly interested in $\Phi_\mathbf{a}$ for $\mathbf{a}\in \noncn{n}^+$---the set of such polynomials will turn out to be a basis for $\wheelpoly{n}$---but the case when $\mathbf{a}$ is in the larger set $\mathcal{A}_n$ will play a useful intermediate role.

The above definition of $\Phi_\mathbf{a}$ is somewhat difficult to work with. In particular, it is not even obvious that the $\Phi_\mathbf{a}$'s are polynomials. However, the contour integrals can be evaluated using the residue formula, which gives a more concrete representation of the $\Phi_\mathbf{a}$'s that makes it possible to show that they are in fact wheel polynomials and prove additional properties.

\begin{lem}
\begin{enumerate}

\item
The functions $\Phi_\mathbf{a}$ can be written explicitly as follows.
\begin{align*}
& \hspace{-20.0pt}  \Phi_{\mathbf{a}}(\mathbf{z}) =
(-1)^{\binom{n}{2}} \prod_{1\le j<k\le 2n} (qz_j-q^{-1} z_k) \\ & 
\times \sum_{\tiny \begin{array}{c} (m_1,\ldots,m_n) \\ 1\le m_j\le a_j, m_j\neq m_k \end{array}} 
\frac{\displaystyle \prod_{1\le j<k\le n} (z_{m_k}-z_{m_j})(qz_{m_j}-q^{-1} z_{m_k})}{\displaystyle  \prod_{j=1}^n \left(
\prod_{\scriptsize \begin{array}{c} k=1 \\ k\neq a_j \end{array}}^{a_j} (z_{m_j}-z_k) \prod_{k=a_j+1}^{2n} (qz_{m_j}-q^{-1} z_k)
\right) }
\\[3pt] & \hspace{-20.0pt}
\ \ \ \ \ =
(-1)^{\binom{n}{2}}
\\[3pt] &\hspace{-10.0pt} \times \hspace{-15.0pt} \sum_{\tiny \begin{array}{c} M=(m_1,\ldots,m_n) \\ 1\le m_j\le a_j, m_j\neq m_k \end{array}}
\hspace{-15.0pt} (-1)^{\operatorname{inv}(M)}
\frac{\displaystyle \prod_{1\le j<k\le n} (q z_{m_j} - q^{-1} z_{m_k}) \hspace{-20pt} \prod_{\tiny \begin{array}{c}1\le j <k\le 2n\\ j\notin M\textnormal{ or }j=m_p, k\le a_p\end{array}} \hspace{-15.0pt} (qz_j-q^{-1} z_k) }
{\displaystyle \prod_{j=1}^n \prod_{\tiny \begin{array}{c} 1\le k\le a_j \\ k\notin M \textnormal{ or }k>m_j\end{array}} (z_{m_j}-z_k)},
\end{align*}
where for a sequence $M=(m_1,\ldots,m_n)$, $\operatorname{inv}(M)$ denotes the number of inversions of $M$, given by
$\operatorname{inv}(M) =\#\{ 1\le i<j\le n\,:\, m_i>m_j \}$.

\item
$\Phi_\mathbf{a}$ is a wheel polynomial of order $n$.

\end{enumerate}
\end{lem}

\begin{proof} See \cite[Section 3]{zinn-justin-di-francesco2}.
\end{proof}

For each sequence $\mathbf{a}=(a_1,\ldots,a_n)\in\mathcal{A}_n$, let $(C_{\mathbf{a},\sigma})_{\sigma\in\noncn{n}}$ be coefficients such that
\begin{equation} \label{eq:def-c-a-sigma}
\Phi_\mathbf{a} = \sum_{\sigma\in\noncn{n}} C_{\mathbf{a},\sigma} \Psi_{\sigma}.
\end{equation}
For $\pi\in\noncn{n}$ denote $C_{\pi,\sigma} = C_{\mathbf{a},\sigma}$ with $\mathbf{a}=\pi^+$. The coefficients $(C_{\pi,\sigma})_{\sigma\in\noncn{n}}$ have the property that
\begin{equation} \label{eq:def-c-pi-sigma}
\Phi_\pi = \sum_{\sigma\in\noncn{n}} C_{\pi,\sigma} \Psi_{\sigma}.
\end{equation}

\begin{lem} We have the relations
\begin{align*} 
\Phi_\mathbf{a}(\sigma) &= C_{\mathbf{a},\sigma} \qquad (\mathbf{a}\in\mathcal{A}_n), \\
\Phi_\pi(\sigma) &= C_{\pi,\sigma} \qquad (\pi\in\noncn{n}).
\end{align*}
\end{lem}

\begin{proof}
This follows immediately from \eqref{eq:psi-sigma-evaluations} by considering the $\pi$-evaluation of both sides of \eqref{eq:def-c-a-sigma}--\eqref{eq:def-c-pi-sigma}.
\end{proof}

The next result is a recurrence relation that makes it possible to recursively compute the quantities $\Phi_\mathbf{a}(\sigma) = C_{\mathbf{a},\sigma}$, analogously to Theorem~\ref{thm:qkz-eval-recursion}.

\begin{thm}
If $\sigma\in\noncn{n}$ has a ``little arc'' $j\matched{\sigma}j+1$ and $\mathbf{a} \in\mathcal{A}_n$, let $p$ be the number of times $j$ appears in $\mathbf{a}$, so that we can write
$$ \mathbf{a} = (a_1,\ldots,a_k,\overbrace{j,\ldots,j}^{p\textrm{ times}},a_{k+p+1},\ldots,a_n)
$$
where $a_k<j<a_{k+p+1}$. In the case when $p>0$, denote
\begin{equation}
\label{eq:hat-a-def}
\hat{\mathbf{a}}_j = (a_1,\ldots,a_k,\overbrace{j,\ldots,j}^{p-1\textrm{ times}},a_{k+p+1}-2,\ldots,a_n-2).
\end{equation}
Then we have
\begin{equation}
\label{eq:phi-a-recurrence}
\Phi_{\mathbf{a}}(\sigma) = \begin{cases}
3^{n-1} \chi(p) \Phi_{\hat{\mathbf{a}}_j}(\hat{\sigma}_j)
& \textrm{if }p>0 \ \ (\textrm{i.e., if }j\in\{a_1,\ldots,a_n\}), \\
0 & \textrm{if }p=0,
\end{cases}
\end{equation}
where
\begin{equation} \label{eq:character-mod-three}
\chi(p) = \frac{\displaystyle q^p-q^{-p}}{\displaystyle  q-q^{-1}} = \begin{cases}
0 & \textnormal{if }p \equiv 0 \textnormal{ (mod }3), \\
1 & \textnormal{if }p \equiv 1 \textnormal{ (mod }3), \\
-1 & \textnormal{if }p\equiv 2 \textnormal{ (mod }3).
\end{cases}
\end{equation}
\end{thm}

\begin{proof}
See \cite[Appendix A]{zinn-justin-di-francesco2}.
\end{proof}

Note that if $\mathbf{a}=\pi^+ \in \noncn{n}^+$ then $\hat{\mathbf{a}}_j$ is in $\mathcal{A}_{n-1}$ but not necessarily in $\noncn{n-1}^+$, which is the reason why it was necessary to consider the family of polynomials indexed by the larger set $\mathcal{A}_n$ even though we are mainly interested in the polynomials $\Phi_\pi$. 

Like \eqref{eq:qkz-eval-recursion}, the recurrence relation \eqref{eq:phi-a-recurrence} can be solved to give an explicit formula for $C_{\mathbf{a},\sigma}$.

\begin{thm}
For $\mathbf{a}=(a_1,\ldots,a_n)\in\noncn{n}^+$ and $\sigma\in\noncn{n}$, the coefficient $C_{\mathbf{a},\sigma}$ is given by:
\begin{equation}
C_{\mathbf{a},\sigma} = \prod_{j<k,\ j\matched{\sigma}k} \chi(p_{j,k}(\mathbf{a})),
\label{eq:cmatrix-explicit}
\end{equation}
where we denote $p_{j,k}(\mathbf{a}) = \#\{ m\,:\, j\le a_m <k \}-\tfrac12 (k-j-1)$.
\end{thm}

\begin{proof}
See \cite[Appendix A]{zinn-justin-di-francesco2}.
\end{proof}

\begin{thm}
The matrix $(C_{\pi,\sigma})_{\pi,\sigma\in\noncn{n}}$ is triangular with respect to the partial order $\preceq$ on link patterns, and has $1$s on the diagonal (i.e., $C_{\pi,\pi}=1$ for all $\pi\in\noncn{n}$), hence it is invertible.
\end{thm}

\begin{proof} See \cite[Sec.~3.3.2]{fonseca-zinn-justin}.
\end{proof}

\begin{cor}
$(\Phi_\pi)_{\pi\in \noncn{n}}$ form a basis for $\wheelpoly{n}$.
\end{cor}

Denote the coefficients of the inverse matrix of $(C_{\pi,\sigma})_{\pi,\sigma\in\noncn{n}}$ by $\tilde{C}_{\pi,\sigma}$, or $\tilde{C}_{\pi,\mathbf{a}}$ if $\mathbf{a}=\sigma^+$. 
In other words, the coefficients $\tilde{C}_{\pi,\sigma}$ are defined by the relation
$$
\Psi_\pi = \sum_{\sigma\in\noncn{n}} \tilde{C}_{\pi,\sigma} \Phi_\sigma.
$$

We have defined two square matrices $\mathbf{C}_n = (C_{\pi,\sigma})_{\pi,\sigma\in\noncn{n}}$ and $\tilde{\mathbf{C}}_n=\mathbf{C}_n^{-1}=(\tilde{C}_{\pi,\sigma})_{\pi,\sigma\in\noncn{n}}$, which are the change of basis matrices transitioning between the two wheel polynomial bases we defined. 
Note that $\tilde{C}_{\pi,\sigma}$, like $C_{\pi,\sigma}$, are integers, but there does not seem to be a simple formula for them. The matrices $\mathbf{C}_n$ and $\tilde{\mathbf{C}}_n$ for $n=2,3,4$ are shown in Table~\ref{table:transition-matrices}.
\begin{table}[h]
$$
\begin{array}{c|cc}
n & \mathbf{C}_n & \tilde{\mathbf{C}}_n \\
\hline \\[-1ex]
2 & \begin{pmatrix} 1&0\\0&1 \end{pmatrix} & \begin{pmatrix} 1&0\\0&1 \end{pmatrix}
\\[12pt]
3 &
\begin{pmatrix} 1&0&0&0&0\\0&1&0&1&0\\0&0&1&0&0\\0&0&0&1&0\\0&0&0&0&1
\end{pmatrix}
&
\begin{pmatrix} 1&0&0&0&0\\0&1&0&-1&0\\0&0&1&0&0\\0&0&0&1&0\\0&0&0&0&1
\end{pmatrix}
\\[35pt]
4 &
\left(
\begin{smallmatrix}
 1 & 0 & 0 & 0 & 0 & 0 & 0 & 0 & 0 & 0 & 0 & 0 & 0 & 0 \\
 0 & 1 & 0 & 0 & 1 & 0 & 0 & 0 & 0 & 0 & 0 & 0 & 1 & 0 \\
 0 & 0 & 1 & 0 & 0 & 0 & 0 & 0 & 1 & 0 & 0 & 0 & 0 & 0 \\
 0 & 0 & 0 & 1 & 0 & 0 & 0 & 0 & 0 & 0 & 0 & 1 & 0 & 1 \\
 0 & 0 & 0 & 0 & 1 & 0 & 0 & 0 & 0 & 0 & 0 & 0 & 1 & -1 \\
 0 & 0 & 0 & 0 & 0 & 1 & 0 & 0 & 0 & 0 & 0 & 0 & 0 & 0 \\
 0 & 0 & 0 & 0 & 0 & 0 & 1 & 0 & 0 & 0 & 1 & 0 & 0 & 0 \\
 0 & 0 & 0 & 0 & 0 & 0 & 0 & 1 & 0 & 0 & 0 & 0 & 0 & 0 \\
 0 & 0 & 0 & 0 & 0 & 0 & 0 & 0 & 1 & 0 & 0 & 0 & 0 & 0 \\
 0 & 0 & 0 & 0 & 0 & 0 & 0 & 0 & 0 & 1 & 0 & 0 & 0 & 0 \\
 0 & 0 & 0 & 0 & 0 & 0 & 0 & 0 & 0 & 0 & 1 & 0 & 0 & 1 \\
 0 & 0 & 0 & 0 & 0 & 0 & 0 & 0 & 0 & 0 & 0 & 1 & 0 & 0 \\
 0 & 0 & 0 & 0 & 0 & 0 & 0 & 0 & 0 & 0 & 0 & 0 & 1 & 0 \\
 0 & 0 & 0 & 0 & 0 & 0 & 0 & 0 & 0 & 0 & 0 & 0 & 0 & 1
 \end{smallmatrix}
 \right)
&
\left(
\begin{smallmatrix}
 1 & 0 & 0 & 0 & 0 & 0 & 0 & 0 & 0 & 0 & 0 & 0 & 0 & 0 \\
 0 & 1 & 0 & 0 & -1 & 0 & 0 & 0 & 0 & 0 & 0 & 0 & 0 & -1 \\
 0 & 0 & 1 & 0 & 0 & 0 & 0 & 0 & -1 & 0 & 0 & 0 & 0 & 0 \\
 0 & 0 & 0 & 1 & 0 & 0 & 0 & 0 & 0 & 0 & 0 & -1 & 0 & -1 \\
 0 & 0 & 0 & 0 & 1 & 0 & 0 & 0 & 0 & 0 & 0 & 0 & -1 & 1 \\
 0 & 0 & 0 & 0 & 0 & 1 & 0 & 0 & 0 & 0 & 0 & 0 & 0 & 0 \\
 0 & 0 & 0 & 0 & 0 & 0 & 1 & 0 & 0 & 0 & -1 & 0 & 0 & 1 \\
 0 & 0 & 0 & 0 & 0 & 0 & 0 & 1 & 0 & 0 & 0 & 0 & 0 & 0 \\
 0 & 0 & 0 & 0 & 0 & 0 & 0 & 0 & 1 & 0 & 0 & 0 & 0 & 0 \\
 0 & 0 & 0 & 0 & 0 & 0 & 0 & 0 & 0 & 1 & 0 & 0 & 0 & 0 \\
 0 & 0 & 0 & 0 & 0 & 0 & 0 & 0 & 0 & 0 & 1 & 0 & 0 & -1 \\
 0 & 0 & 0 & 0 & 0 & 0 & 0 & 0 & 0 & 0 & 0 & 1 & 0 & 0 \\
 0 & 0 & 0 & 0 & 0 & 0 & 0 & 0 & 0 & 0 & 0 & 0 & 1 & 0 \\
 0 & 0 & 0 & 0 & 0 & 0 & 0 & 0 & 0 & 0 & 0 & 0 & 0 & 1 \\
\end{smallmatrix}
\right)
\end{array}
$$
\caption{The matrices $\mathbf{C}_n, \tilde{\mathbf{C}}_n$ for $n=2,3,4$ (the ordering of the rows and columns corresponds to decreasing lexicographic order of the Young diagrams associated with noncrossing matchings). Note that $\tilde{\mathbf{C}}_n$ can have entries larger in magnitude than $1$; the smallest value of $n$ for which this happens is $n=7$.}
\label{table:transition-matrices}
\end{table}
As we shall see, these matrices will enable us to translate certain sums of the qKZ polynomials (which are associated with certain connectivity events in loop percolation) to linear combinations of the $\Phi_{\pi}$'s. To get back to numerical results about the numbers $\psi_\pi = \Psi_\pi(\mathbf{1})$ (which are related to the probabilities $\mu_\pi$ according to Theorem~\ref{thm:psi-mu-related}), we will want to consider also the $\mathbf{1}$-evaluations of the $\Phi_\pi$'s. Denote $\phi_{\mathbf{a}} = \Phi_\mathbf{a}(\mathbf{1})$,
$\phi_{\pi} = \Phi_\pi(\mathbf{1})$. The coefficients $\phi_\pi$ are related by
\begin{align*}
\phi_\pi &= \sum_{\sigma\in\noncn{n}} C_{\pi,\sigma} \psi_\sigma, \\
\psi_\pi &= \sum_{\sigma\in\noncn{n}} \tilde{C}_{\pi,\sigma} \phi_\sigma.
\end{align*}

\begin{lem}[\cite{zinn-justin-di-francesco2}, Section 3.4] 
$\phi_\mathbf{a}$ has the following contour integral formula:
\begin{equation} \label{eq:phi-a-eval1}
\phi_\mathbf{a} = \oint\!\ldots\!\oint \prod_{1\le j<k\le n} (u_k-u_j)(1+u_k+u_j u_k) \prod_{j=1}^n \frac{1}{u_j^{a_j}} \cdot \prod_{j=1}^n \frac{du_j}{2\pi i},
\end{equation}
where the contours are circles of arbitrary radius around $0$. Equivalently, $\phi_{\mathbf{a}}$ can be written as a constant term
\begin{equation}
\label{eq:phi-a-eval1-constantterm}
\phi_\mathbf{a} = \CT_{u_1,\ldots,u_n} \left(
\prod_{1\le j<k\le n} (u_k-u_j)(1+u_k+u_j u_k) \prod_{j=1}^n \frac{1}{u_j^{a_j-1}} \right)
\end{equation}
or as a Taylor coefficient
$$
\phi_\mathbf{a} = \left[u_1^{a_1-1} \ldots u_n^{a_n-1} \right]
\left(
\prod_{1\le j<k\le n} (u_k-u_j)(1+u_k+u_j u_k)\right).
$$
\end{lem}

\begin{proof} In the integral formula \eqref{eq:phi-a-def} set $z_1=\ldots=z_{2n}=1$ and then make the substitution $w_j = \frac{1-q^{-1} u_j}{1-q u_j}$. This gives \eqref{eq:phi-a-eval1} after a short computation.
\end{proof}

\subsection{Product-type expansions and the sum rule}

We are ready to start applying the theory developed above to the problem of computing probabilities of connectivity events for the random noncrossing matching $\pistarn$. Our methods generalize those of Zinn-Justin and Di-Francesco \cite{zinn-justin-di-francesco2}, so let us recall briefly one of their beautiful discoveries. Define a set $\mathcal{B}_n \subseteq \noncn{n}^+$ consisting of $2^{n-1}$ sequences, given by
$$ \mathcal{B}_n = \{ \mathbf{a}=(a_1,\ldots,a_n)\in\mathcal{A}_n\,:\, a_1=1,\  a_j\in\{2j-2,2j-1\},\ (2\le j\le n) \}. $$
For each $\mathbf{a}=(a_1,\ldots,a_n)\in \mathcal{B}_n$ define
$$ \mathcal{L}(\mathbf{a}) = \{ \pi\in\noncn{n} \,:\, \textrm{for all } 1\le j\le n,\ \   \pi_{2j-1}=1 \textrm{ iff }a_j=2j-1 \}. $$
Note that $\noncn{n} = \displaystyle \bigsqcup_{\mathbf{a}\in\mathcal{B}_n} \mathcal{L}(\mathbf{a})$.

\begin{thm}
\label{thm:phi-a-psi-expansion}
For any $\mathbf{a}\in\mathcal{B}_n$, we have the expansion
\begin{equation} \label{eq:phi-a-psi-expansion}
\Phi_\mathbf{a} = \sum_{\pi\in\mathcal{L}(\mathbf{a})} \Psi_\pi.
\end{equation}
\end{thm}

Theorem~\ref{thm:phi-a-psi-expansion} and its elegant proof served as the inspiration for one of our main results (Theorem~\ref{thm:submatching-event-expansion} in Subsection~\ref{sec:gen-submatching}). As an aid in motivating and understanding our use of Zinn-Justin and Di Francesco's proof technique in a more complicated setting, we include their original proof, taken from \cite[Section 3]{zinn-justin-di-francesco2}, below.

\begin{proof}[Proof of Theorem~\ref{thm:phi-a-psi-expansion}]
By Theorem~\ref{thm:sufficient-evaluations}, it is enough to prove that for any $\sigma\in\noncn{n}$ we have
\begin{equation}
\label{eq:phi-psi-expansion}
\Phi_\mathbf{a}(\sigma) = \sum_{\pi\in\mathcal{L}(\mathbf{a})} \Psi_\pi(\sigma).
\end{equation}
The proof is by induction on $n$. Denote the right-hand side of \eqref{eq:phi-psi-expansion} by $\theta_{\mathbf{a}}(\sigma)$.
Let $k\matched{\sigma}k+1$ be a little arc of $\sigma$. Let $j$ be such that $k\in \{2j-2,2j-1\}$. By the definition of $\mathcal{B}_n$, either $k=a_j$ is the only occurrence of $k$ in $\mathbf{a}$, or $k$ does not occur in $\mathbf{a}$ (in which case $a_j$ is the number in $\{2j-2,2j-1\}$ of the opposite parity from $k$).
For the left-hand side of \eqref{eq:phi-psi-expansion} we have by \eqref{eq:phi-a-recurrence} that 
$$ \Phi_\mathbf{a}(\sigma) = \begin{cases}
3^{n-1} \Phi_{\hat{\mathbf{a}}_k}(\hat{\sigma}_k) & \textrm{if }a_j=k, \\
0 & \textrm{if }a_j \neq k.
\end{cases}
$$
From the definition of $\mathcal{B}_n$ it is easy to see that when $k=a_j$, $\hat{\mathbf{a}}_k \in \mathcal{B}_{n-1}$. So, to complete the inductive proof it is enough to show that $\theta_\mathbf{a}(\sigma)$ satisfies the same recurrence. First, assume that $a_j\neq k$. If $k=2j-1$ (so $a_j=2j-2$) then any $\pi\in\mathcal{L}(\mathbf{a})$ will have the property that $\pi_k=-1$, that is, $\pi$ matches $k$ with a number to the left of $k$. In particular $\pi$ does not contain the little arc $(k,k+1)$, and therefore $\Psi_\pi(\sigma)=0$. This shows that
$\theta_{\mathbf{a}}(\sigma) = 0$. Alternatively, if $k=2j-2, a_j=2j-1$, then any $\pi\in\mathcal{L}(\mathbf{a})$ will have the property that $\pi$ matches $a_j=k+1$ with a number to the right of $k+1$, and again the little arc $(k,k+1)$ cannot be in $\pi$ and $\Psi_\pi(\sigma)=0$, so the relation $\theta_{\mathbf{a}}(\sigma) = 0$ also holds.

Finally, in the case $a_j=k$, the noncrossing matchings $\pi\in\mathcal{L}(\mathbf{a})$ can be divided into those for which $(k,k+1)$ is not a little arc of $\pi$, which satisfy $\Psi_\pi(\sigma)=0$, and those for which $(k,k+1)$ is a little arc of $\pi$. The ones in the latter class satisfy $\Psi_\pi(\sigma)=3^{n-1} \Psi_{\hat{\pi}_k}(\hat{\sigma}_k)$, and furthermore the correspondence $\pi\mapsto \hat{\pi}_k$ maps them bijectively onto $\mathcal{L}(\hat{\mathbf{a}}_k)$. Thus we get the relation
$$ \theta_{\mathbf{a}}(\sigma) = 3^{n-1} \theta_{\hat{\mathbf{a}}_k}(\hat{\sigma}_k), $$
which is the same recurrence as the one for $\Phi_\mathbf{a}(\sigma)$ and therefore completes the proof.
\end{proof}

Theorem~\ref{thm:phi-a-psi-expansion} has several immediate corollaries that are directly relevant to the problem of deriving formulas for probabilities of connectivity events. First, summing \eqref{eq:phi-a-psi-expansion} over the sequences $\mathbf{a}\in\mathcal{B}_n$ gives that
$$
\sum_{\pi\in\noncn{n}} \Psi_\pi = \sum_{\mathbf{a}\in \mathcal{B}_n} \Phi_\mathbf{a}.
$$
Taking the $\mathbf{1}$-evaluation of both sides then gives the relation
$$
\sum_{\pi\in\noncn{n}} \psi_\pi = \sum_{\mathbf{a}\in \mathcal{B}_n} \phi_\mathbf{a},
$$
and the right-hand side can then be expanded as a constant term using \eqref{eq:phi-a-eval1-constantterm}, which gives that
\begin{equation} \label{eq:sum-psi-constterm}
\sum_{\pi\in\noncn{n}} \psi_\pi = 
\CT_{u_1,\ldots,u_n} \left(
\prod_{1\le j<k\le n} (u_k-u_j)(1+u_k+u_j u_k) \prod_{j=2}^n \frac{1+u_j}{u_j^{2j-2}} \right).
\end{equation}
Note that the key point that makes such an expansion possible is the fact that $\mathcal{B}_n$ decomposes as a Cartesian product, which manifests itself as the product $\prod_{j=2}^n (1+u_j)/u_j^{2j-2}$ in the Laurent polynomial. It is precisely this product structure that our more general result seeks to emulate. Finally, applying the identity \eqref{eq:asm-const-term} gives the following result.

\begin{thm}
\label{thm:sum-rule}
$ \sum_{\pi\in\noncn{n}} \psi_\pi = \asm(n)$.
\end{thm}

With this result, we are finally able to relate results about the vector $\boldsymbol{\psi}_n$ to statements about its normalized version, the probability vector $\boldsymbol{\mu}_n$. As an example (also taken from \cite{zinn-justin-di-francesco2}), letting $\mathbf{a}=(1,3,5,\ldots,2n-1)$, note that $\mathcal{L}(\mathbf{a})$ is the singleton set $\{\pi^n_{\textrm{max}}\}$, so,
applying \eqref{eq:phi-a-psi-expansion} in this case, then taking the $\mathbf{1}$-evaluation and using \eqref{eq:phi-a-eval1-constantterm} as before, gives the identity
$$
\psi_{\pi^n_{\textrm{max}}} = 
\CT_{u_1,\ldots,u_n} \left(
\prod_{1\le j<k\le n} (u_k-u_j)(1+u_k+u_j u_k) \prod_{j=1}^n \frac{1}{u_j^{2j-2}} \right).
$$
This constant term can be evaluated iteratively as
\begin{align*}
\CT_{u_1,\ldots,u_{n-1}}  & \CT_{u_n} \left(
\prod_{1\le j<k\le n} (u_k-u_j)(1+u_k+u_j u_k) \prod_{j=1}^n \frac{1}{u_j^{2j-2}} \right)
\\ &= 
\CT_{u_1,\ldots,u_{n-1}} \left(
\prod_{1\le j<k\le n-1} (u_k-u_j)(1+u_k+u_j u_k) \prod_{j=1}^n \frac{1+u_j}{u_j^{2j-2}} \right) 
\\ &= \asm(n-1)
\end{align*}
(use \eqref{eq:asm-const-term}, noting that on the right-hand side of that identity, starting the product $\prod_j (1+z_j)$ at $j=1$ instead of $j=2$ does not affect the constant term),
confirming part 3 of Theorem~\ref{thm:sum-rules}.

\subsection{$p$-nested matchings}

\label{sec:p-nested-matchings}

If $\pi\in\noncn{n}$, denote by $\lpbracket{\pi}$ the noncrossing matching of order $n+1$ obtained from $\pi$ by relabeling the elements of $[1,2n]$ as $2,\ldots,2n+1$ and adding a ``large arc'' connecting $1$ and $2n+2$. We call this the \textbf{nesting of $\pi$} or ``$\pi$-nested.'' For an integer $p\ge0$ let $\lpbracket{\pi}^p$ denote the $p$-nesting of $\pi$, defined as the result of $p$ successive nesting operations applied to $\pi$. Similarly, for the associated sequence $\mathbf{a}=\pi^+ \in \noncn{n}^+$, denote $\lpbracket{\mathbf{a}}=(\lpbracket{\pi})^+$ and $\lpbracket{\mathbf{a}}^p=(\lpbracket{\pi}^p)^+$. If $\mathbf{a}=(a_1,\ldots,a_n)$, it is easy to see that this can be written more explicitly as
$$ \lpbracket{\mathbf{a}}^p = (1,\ldots,p,p+a_1,\ldots,p+a_n). $$

The following result is due to Fonseca and Zinn-Justin \cite{fonseca-zinn-justin}.

\begin{thm}
Let $p\ge 0$. Under the $p$-nesting operation, the expansions
\begin{align*}
\Phi_\sigma &= \sum_{\pi\in\noncn{n}} C_{\sigma,\pi} \Phi_{\pi} \qquad (\pi\in\noncn{n}),\\
\Psi_\pi &= \sum_{\sigma\in\noncn{n}} \tilde{C}_{\pi,\sigma} \Phi_{\sigma} \qquad (\sigma\in\noncn{n}),
\end{align*}
turn into
\begin{align*}
\Phi_{\lpbracket{\sigma}^p} &= \sum_{\pi\in\noncn{n}} C_{\sigma,\pi} \Psi_{\lpbracket{\pi}^p} \qquad (\pi\in\noncn{n}), \\
\Psi_{\lpbracket{\pi}^p} &= \sum_{\sigma\in\noncn{n}} \tilde{C}_{\pi,\sigma} \Phi_{\lpbracket{\sigma}^p} \qquad (\sigma\in\noncn{n}).
\end{align*}
Equivalently, we have the recursive relations
\begin{align}
\label{eq:p-nesting-first}
C_{\lpbracket{\sigma}^p, \mu} &= \begin{cases} C_{\sigma,\pi} & \textrm{if }\mu=\lpbracket{\pi}^p, \\ 0 & \textrm{otherwise},
\end{cases} \\[5pt]
\label{eq:p-nesting-second}
\tilde{C}_{\lpbracket{\pi}^p, \mu} &= \begin{cases} \tilde{C}_{\pi,\sigma} & \textrm{if }\mu=\lpbracket{\sigma}^p, \\ 0 & \textrm{otherwise}.
\end{cases}
\end{align}
\end{thm}

\begin{proof}
The relation \eqref{eq:p-nesting-second} is an immediate consequence of \eqref{eq:p-nesting-first}. To prove \eqref{eq:p-nesting-first}, note that the fact that $C_{\lpbracket{\sigma}^p, \mu}=0$ if $\mu$ is not of the form $\lpbracket{\sigma}^p$ for some $\sigma\in\noncn{n}$ follows from the triangularity of the basis change matrix with respect to the order $\preceq$. The relation $C_{\lpbracket{\sigma}^p, \lpbracket{\pi}^p}=C_{\sigma,\pi}$ is easy to prove by induction from~\eqref{eq:phi-a-recurrence}.
\end{proof}

\subsection{General submatching events}

\label{sec:gen-submatching}

Fix $k\ge0$ and $\pi_0\in\noncn{k}$. For $n\ge k+1$ define
$$ \mathcal{E}_n(\pi_0) = \left\{ \pi\in\noncn{n}\,:\, \pi(j)=1+\pi_0(j-1) \textrm{ for }2\le j\le k+1 \right\}. $$
In the terminology of Subsections~\ref{sec:connectivity-infinite} and \ref{sec:loop-perc-cylinder}, this set of matchings is related to the submatching event associated with the submatching $\pi_0$, except that we require the submatching to occur on the sites in the interval $[2,k+1]$ instead of $[1,k]$. Surprisingly, this turns out to be the correct thing to do---see below.

Next, for any $n\ge k+1$ and $\mathbf{b}=(b_{k+2},\ldots,b_n)$ satisfying $b_j\in\{2j-2,2j-1\}$ for all $j$, define
\begin{align*}
\mathcal{E}_n(\pi_0, \mathbf{b}) = \Big\{ &\pi\in\mathcal{E}_n(\pi_0) \,:\, 
\pi_{2j-1}=1 \textrm{ iff }b_j=2j-1 \ \ (k+2\le j\le n)
\Big\}.
\end{align*}
Note that
$$ 
\mathcal{E}_n(\pi_0) = \bigsqcup_{\begin{array}{c} \scriptstyle \mathbf{b}=(b_{k+2},\ldots,b_n) \\[-3pt] \scriptstyle \forall j\,\, b_j\in\{2j-2,2j-1\} \end{array}} \mathcal{E}_n(\pi_0, \mathbf{b}).
$$

The following result will easily imply Theorem~\ref{thm:explicit-formulas-finite-n}, and is one of the main results of this paper.

\begin{thm}
\label{thm:submatching-event-expansion}
For any $\pi_0\in\noncn{k}$, $n\ge k+1$ and vector $\mathbf{b}=(b_{k+2},\ldots,b_n)$ satisfying $b_j\in\{2j-2,2j-1\}$ for all $j$, we have
\begin{align}
\label{eq:submatching-b-expansion}
\sum_{\pi\in\mathcal{E}_n(\pi_0,\mathbf{b})}\Psi_\pi &= 
\sum_{\mathbf{a}=(a_1,\ldots,a_k)\in\noncn{k}^+}
\tilde{C}_{\pi_0, \mathbf{a}} \Phi_{(1,1+a_1,\ldots,1+a_k,b_{k+2},\ldots,b_n)}.
\end{align}
It follows that for any $\pi_0\in\noncn{k}$ and $n\ge k+1$,
\begin{equation}
\label{eq:submatching-exansion-ecal}
\sum_{\pi\in\mathcal{E}_n(\pi_0)}\Psi_\pi = 
\hspace{-6pt} \sum_{\begin{array}{c} \scriptstyle b_{k+2},\ldots,b_n \\[-3pt] \scriptstyle \forall j\,\,b_j\in\{2j-2,2j-1\} \end{array}}
\hspace{-6pt} \sum_{\mathbf{a}=(a_1,\ldots,a_k)\in\noncn{k}^+}
\tilde{C}_{\pi_0, \mathbf{a}} \Phi_{(1,1+a_1,\ldots,1+a_k,b_{k+2},\ldots,b_n)}.
\end{equation}
\end{thm}

\begin{proof}
By \eqref{eq:p-nesting-second}, the right-hand side of \eqref{eq:submatching-b-expansion} is equal to 
$$
\sum_{\mathbf{a}=(a_1,\ldots,a_{k+1})\in\noncn{k+1}^+}
\tilde{C}_{\lpbracket{\pi_0}, \mathbf{a}} \Phi_{(a_1,\ldots,a_{k+1},b_{k+2},\ldots,b_n)}.
$$
So, instead of proving \eqref{eq:submatching-b-expansion} we will show that
\begin{equation}
\label{eq:submatching-b-expansion-secondversion}
\sum_{\pi\in\mathcal{E}_n(\pi_0,\mathbf{b})}\Psi_\pi = 
\sum_{\mathbf{a}=(a_1,\ldots,a_{k+1})\in\noncn{k+1}^+}
\tilde{C}_{\lpbracket{\pi_0}, \mathbf{a}} \Phi_{(a_1,\ldots,a_{k+1},b_{k+2},\ldots,b_n)}.
\end{equation}
We will prove by induction on $n$ the following slightly stronger claim: if we have an expansion of the form
\begin{equation} \label{eq:pinested-phi-expansion}
\Psi_{\lpbracket{\pi_0}} = \sum_{\mathbf{a}\in\mathcal{A}_{k+1}} c_{\mathbf{a}} \Phi_{\mathbf{a}}
\end{equation}
for some constants $c_{\mathbf{a}}\ (\mathbf{a}\in\mathcal{A}_{k+1})$, then for any vector $\mathbf{b}=(b_{k+2},\ldots,b_n)$ satisfying $b_j\in\{2j-2,2j-1\}$ as in the theorem, the equality
\begin{equation}
\label{eq:psi-phi-two-sums}
\sum_{\pi\in\mathcal{E}_n(\pi_0,\mathbf{b})} \Psi_{\pi} =  \sum_{\mathbf{a}\in\mathcal{A}_{k+1}} c_{\mathbf{a}} \Phi_{(\mathbf{a},\mathbf{b})}
\end{equation}
holds. (Note that this claim implies \eqref{eq:submatching-b-expansion-secondversion} by taking $c_\mathbf{a}=\tilde{C}_{\lpbracket{\pi_0}, \mathbf{a}}$ for $\mathbf{a}\in\noncn{n}^+$ and $c_\mathbf{a}=0$ otherwise, but allowing the sum on the right-hand side of \eqref{eq:pinested-phi-expansion} to range over the bigger set $\mathcal{A}_{k+1}$ helps make the induction work.)
Denote the left-hand side of \eqref{eq:psi-phi-two-sums} by $X_{\pi_0,\mathbf{b}}$ and the right-hand side by $Y_{c,\mathbf{b}}$. We now prove that $X_{\pi_0,\mathbf{b}}=Y_{c,\mathbf{b}}$.

\textbf{The induction base:} here $k=0$, $n=1$, $\pi_0=\emptyset$ (the empty matching), $\mathcal{A}_{k+1}=\mathcal{A}_1=\{ 1\}$ and the expansion necessarily takes the form
$$ \Psi_{ \lpbracket{\emptyset} } = \Phi_1. $$
Also $\mathcal{E}_n(\pi_0,\mathbf{b})=\mathcal{L}((1,b_2,\ldots,b_n))$, so the claim reduces to the relation \eqref{eq:phi-a-psi-expansion} from Theorem~\ref{thm:phi-a-psi-expansion}.

\textbf{The inductive step:} First, if $n=k+1$ then the claim \eqref{eq:psi-phi-two-sums} is the same as the assumption \eqref{eq:pinested-phi-expansion}, since $\mathbf{b}$ is a trivial vector of length $0$ and $\mathcal{E}_n(\pi_0,\mathbf{b})$ is the singleton set $\{ \lpbracket{\pi_0} \}$, so there is nothing to prove. From now on assume $n>k+1$. We need to show that for all $\sigma\in\noncn{n}$, $X_{\pi_0,\mathbf{b}}(\sigma)=Y_{c,\mathbf{b}}(\sigma)$. Let $1\le m\le 2n-1$ be a number such that $m\matched{\sigma}m+1$.

\textbf{Case 1.} Assume $m\ge 2k+2$.

\textbf{Subcase 1a.} Assume $m\notin \mathbf{b}$ (that is, $m\neq b_j$ for all $j$). Then by \eqref{eq:phi-a-recurrence} we have that
$
\Phi_{a_1,\ldots,a_{k+1},b_{k+2},\ldots,b_n}(\sigma)=0
$
for any $\mathbf{a}=(a_1,\ldots,a_{k+1})\in\mathcal{A}_{k+1}$, so $Y_{c,\mathbf{b}}(\sigma)=0$.

\textbf{Subsubcase 1a(i).} Assume $m$ is odd, i.e., $m=2j-1$ for some $k+2\le j\le n$. If $\pi\in\mathcal{E}_n(\pi_0,\mathbf{b})$ then, since we assumed $m\notin \mathbf{b}$, $b_j=2j-2$ and therefore by definition $\pi_m=-1$, that is, $\pi(m)=\pi(2j-1)<m$, so in particular $(m,m+1)$ is not a little arc of $\pi$ and therefore $\Psi_\pi(\sigma)=0$. It follows that $X_{\pi_0,\mathbf{b}}(\sigma)=0=Y_{c,\mathbf{b}}(\sigma)$.

\textbf{Subsubcase 1a(ii).} Assume $m$ is even, i.e., $m=2j-2$ for some $k+2\le j\le n$. If $\pi\in\mathcal{E}_n(\pi_0,\mathbf{b})$ then, since we assumed $m\notin \mathbf{b}$, $b_j=2j-1$ and therefore by definition $\pi(m+1)=\pi(2j-1)>m+1$ and in particular $(m,m+1)$ is not a little arc of $\pi$. Again it follows that $\Psi_\pi(\sigma)=0$ and therefore $X_{\pi_0,\mathbf{b}}(\sigma)=0=Y_{c,\mathbf{b}}(\sigma)$.

\textbf{Subcase 1b.} Assume that $m\in \mathbf{b}$, i.e., $m=b_j$ for some $k+2\le j\le n$. Then by \eqref{eq:phi-a-recurrence}, for any $\mathbf{a}=(a_1,\ldots,a_{k+1})\in\mathcal{A}_{k+1}$ we have
$$ \Phi_{(a_1,\ldots,a_{k+1},b_{k+2},\ldots,b_n)}(\sigma) = 3^{n-1} 
\Phi_{(\mathbf{a},\hat{\mathbf{b}}_m)}(\hat{\sigma}_m),
$$
where $\hat{\mathbf{b}}_m$ is obtained from $\mathbf{b}$ by deleting the $m$ and shifting all the numbers to the right of $m$ down by $2$. It follows that 
\begin{equation} \label{eq:ycb-sigma-bm-hat}
Y_{c,\mathbf{b}}(\sigma) = 3^{n-1} Y_{c,\hat{\mathbf{b}}_m}(\hat{\sigma}_m).
\end{equation}
On the other hand, we have
\begin{align*}
X_{\pi_0,\mathbf{b}}(\sigma) &= \sum_{\pi\in\mathcal{E}_n(\pi_0,\mathbf{b}),\ m\notmatchedvariant{\pi}m+1}\Psi_\pi(\sigma)
+ \sum_{\pi\in\mathcal{E}_n(\pi_0,\mathbf{b}),\ m\matched{\pi}m+1}\Psi_\pi(\sigma)
\\
&= 0+
\sum_{\pi\in\mathcal{E}_n(\pi_0,\mathbf{b}),\ m\matched{\pi}m+1} 3^{n-1} \Psi_{\hat{\pi}_m}(\hat{\sigma}_m).
\end{align*}
Relabeling $\hat{\pi}_m$ in the last sum as $\pi$ transforms it into
$$
3^{n-1} 
\sum_{\pi\in\mathcal{E}_{n-1}(\pi_0,\hat{\mathbf{b}}_m)} \Psi_{\pi}(\hat{\sigma}_m) = 3^{n-1} X_{\pi_0, \hat{\mathbf{b}}_m}(\hat{\sigma}_m),
$$
so by induction we get using \eqref{eq:ycb-sigma-bm-hat} that $X_{\pi_0,\mathbf{b}}(\sigma) = Y_{\pi_0,\mathbf{b}}(\sigma)$.

\textbf{Case 2.} Assume $m\le 2k+1$.

\textbf{Subcase 2(a).} Assume $m\matched{\pi_0}m+1$ (when thinking of $\pi_0$ as ``living'' on the interval $[2,2k+1]$), and note that in this case actually $m$ must be $\le 2k$. Then for any $\pi\in\mathcal{E}_n(\pi_0,\mathbf{b})$ we have 
$$ \Psi_\pi(\sigma) = 3^{n-1} \Psi_{\hat{\pi}_m}(\hat{\sigma}_m), $$
and consequently, denoting $\mathbf{b}'=(b_{k+2}-2,\ldots,b_n-2)$, we see that
$$ X_{\pi_0,\mathbf{b}}(\sigma) = 3^{n-1} X_{\widehat{(\pi_0)}_m,\mathbf{b}'}(\hat{\sigma}_m). $$
To evaluate $Y_{c,\mathbf{b}}(\sigma)$, we use \eqref{eq:phi-a-recurrence} to get that
\begin{align}
Y_{c,\mathbf{b}}(\sigma) &= \sum_{\mathbf{a}\in\mathcal{A}_{k+1}} c_{\mathbf{a}} \Phi_{(\mathbf{a},\mathbf{b})}(\sigma)
= \sum_{\mathbf{a}\in\mathcal{A}_{k+1},\,m\in\mathbf{a}} c_{\mathbf{a}} \Phi_{(\mathbf{a},\mathbf{b})}(\sigma)
\nonumber \\ &= \sum_{\mathbf{a}\in\mathcal{A}_{k+1},\,m\in\mathbf{a}} 3^{n-1} \chi(p(\mathbf{a},m))
c_{\mathbf{a}} \Phi_{(\hat{\mathbf{a}}_m,\mathbf{b}')}(\hat{\sigma}_m)
\nonumber \\ &=
\sum_{\mathbf{a}'\in\mathcal{A}_{k}} d_{\mathbf{a}'} \Phi_{(\mathbf{a}',\mathbf{b}')}(\hat{\sigma}_m),
\label{eq:ycb-daprime-epansion}
\end{align}
where $p(\mathbf{a},m)$ denotes the number of times $m$ appears in $\mathbf{a}$, $\hat{\mathbf{a}}_m$ is as in \eqref{eq:hat-a-def}, and in the last line the sum was reorganized as a linear combination with some new coefficients $d_{\mathbf{a}'}$.
(The fact that this sum ranges over $\mathcal{A}_k$ explains why we needed to strengthen the inductive hypothesis at the beginning of the proof.)

Note also that in a similar way, replacing the vector $\mathbf{b}$ with the empty vector $\emptyset$ of length $0$, we have for any $\mu\in\noncn{k+1}$ such that $m\matched{\mu}m+1$,
\begin{align*}
Y_{c,\emptyset}(\mu) &= \sum_{\mathbf{a}\in\mathcal{A}_{k+1}} c_{\mathbf{a}} \Phi_{\mathbf{a}}(\mu)
= \sum_{\mathbf{a}\in\mathcal{A}_{k+1},\,m\in\mathbf{a}} c_{\mathbf{a}} \Phi_{\mathbf{a}}(\mu)
\\ &= \sum_{\mathbf{a}\in\mathcal{A}_{k+1},\,m\in\mathbf{a}} 3^{n-1} 
\chi(p(\mathbf{a},m))
c_{\mathbf{a}} \Phi_{\hat{\mathbf{a}}_m}(\hat{\mu}_m)
\\ &=
\sum_{\mathbf{a}'\in\mathcal{A}_{k}} d_{\mathbf{a}'} \Phi_{\mathbf{a}'}(\hat{\mu}_m),
\end{align*}
with exactly the same coefficients $d_{\mathbf{a}'}$ (the main thing to notice is that appending the $\mathbf{b}$ at the end in the subscript of $\Phi$ has no effect on the way the recurrence \eqref{eq:phi-a-recurrence} proceeds, i.e., the $\mathbf{b}$ ``doesn't interact'' with the $\mathbf{a}$). On the other hand, the case $n=k+1$ discussed at the beginning of the proof gives that $Y_{c,\emptyset}(\mu)=X_{\pi_0,\emptyset}(\mu)=3^{n-1} X_{\widehat{(\pi_0)}_m,\emptyset}(\hat{\mu}_m)$. Combining this with \eqref{eq:ycb-daprime-epansion}, we get that
$$ X_{\widehat{(\pi_0)}_m,\emptyset}(\nu) =
3^{-(n-1)} \sum_{\mathbf{a}'\in\mathcal{A}_{k}} d_{\mathbf{a}'} \Phi_{\mathbf{a}'}(\nu)
$$
for any $\nu\in\noncn{k}$, or in other words simply that
$$ X_{\widehat{(\pi_0)}_m,\emptyset} =
3^{-(n-1)} \sum_{\mathbf{a}'\in\mathcal{A}_{k}} d_{\mathbf{a}'} \Phi_{\mathbf{a}'}.
$$
Invoking the inductive hypothesis again, we conclude that
$$ 
X_{\widehat{(\pi_0)}_m,\mathbf{b}'} =
3^{-(n-1)} \sum_{\mathbf{a}'\in\mathcal{A}_{k}} d_{\mathbf{a}'} \Phi_{(\mathbf{a}',\mathbf{b}')}.
$$
In particular, 
$$ 
3^{(n-1)} X_{\widehat{(\pi_0)}_m,\mathbf{b}'}(\hat{\sigma}_m) =
\sum_{\mathbf{a}'\in\mathcal{A}_{k}} d_{\mathbf{a}'} \Phi_{(\mathbf{a}',\mathbf{b}')}(\hat{\sigma}_m).
$$
But as we showed above, the LHS is equal to $X_{\pi_0,\mathbf{b}}(\sigma)$ and the RHS is equal to $Y_{c,\mathbf{b}}(\sigma)$, so this proves the claim.

\textbf{Subcase 2(b).} Assume $m \notmatched{\pi_0} m+1$. Then for a similar reason as in subcase 2(a) above, $X_{\pi_0,\mathbf{b}}(\sigma)=0$. We need to show that also $Y_{c,\mathbf{b}}(\sigma)=0$. The computation from subcase 2(a) is still valid, so we can write as before $Y_{c,\mathbf{b}}(\sigma) = \sum_{\mathbf{a}'\in\mathcal{A}_k} d_{\mathbf{a}'} \Phi_{(\mathbf{a}',\mathbf{b}')}(\hat{\sigma}_m)$, and also
$$ Y_{c,\emptyset}(\mu)=
\sum_{\mathbf{a}'\in\mathcal{A}_k} d_{\mathbf{a}'} \Phi_{\mathbf{a}'}(\hat{\mu}_m)
$$
for any $\mu\in\noncn{k+1}$ for which $m\matched{\mu}m+1$. But we know (again from the case $n=k+1$) that $Y_{c,\emptyset}(\mu)=X_{\pi_0,\emptyset}(\mu)=0$, so we deduce that 
$$ \sum_{\mathbf{a}'\in\mathcal{A}_k} d_{\mathbf{a}'} \Phi_{\mathbf{a}'} = 0. $$
We claim that this implies also that
\begin{equation} \label{eq:no-interaction-claim} 
\sum_{\mathbf{a}'\in\mathcal{A}_k} d_{\mathbf{a}'} \Phi_{(\mathbf{a}',\mathbf{b}')} = 0,
\end{equation}
which, if true, would also give in particular that
$$ Y_{c,\mathbf{b}}(\sigma) = \sum_{\mathbf{a}'\in\mathcal{A}_k} d_{\mathbf{a}'} \Phi_{(\mathbf{a}',\mathbf{b}')}(\hat{\sigma}_m) = 0, $$
and finish the proof. 

It remains to prove \eqref{eq:no-interaction-claim}. First, for convenience relabel $\mathbf{a}'$ as $\mathbf{a}$, $\mathbf{b}'$ as $\mathbf{b}$ and $n-1$ as $n$. The claim then becomes that if $\mathbf{b}=(b_{k+1},\ldots,b_n)$ satisfies $b_j\in\{2j-2,2j-1\}$ for all $j$ and $(d_\mathbf{a})_{\mathbf{a}\in \mathcal{A}_k}$ are coefficients such that $\sum_{\mathbf{a}\in\mathcal{A}_k} d_\mathbf{a} \Phi_{\mathbf{a}}=0$ then 
$\sum_{\mathbf{a}\in\mathcal{A}_k} d_\mathbf{a} \Phi_{(\mathbf{a},\mathbf{b})}=0$.
The proof is by induction on $n$; the idea as before is that the $\mathbf{a}'$ and $\mathbf{b}'$ components ``don't interact'' when we perform recursive computations using \eqref{eq:phi-a-recurrence}, and the proof involves similar ideas to those used above. If $n=k$ there is nothing to prove, so assume that $n>k$. Let $\nu\in \noncn{n}$, and pick some $g$ such that $g\matched{\pi}g+1$. Divide into two cases, when $g\ge 2k$ or when $g\le 2k-1$. We will show that in each case we have $\sum_{\mathbf{a}\in\mathcal{A}_k} d_\mathbf{a} \Phi_{(\mathbf{a},\mathbf{b})}(\nu)=0$, which will imply the claim for the usual reason (Theorem~\ref{thm:sufficient-evaluations}).

In the first case, if $g\notin \{ b_{k+1},\ldots,b_n \}$ then $\sum_{\mathbf{a}\in\mathcal{A}_k} d_\mathbf{a} \Phi_{(\mathbf{a},\mathbf{b})}(\nu) = 0$ by \eqref{eq:phi-a-recurrence}. Alternatively, if $g\in \{ b_{k+1},\ldots,b_n \}$
then, again by \eqref{eq:phi-a-recurrence}, we have that
\begin{equation}
\label{eq:sum-b-overline}
\sum_{\mathbf{a}\in\mathcal{A}_k} d_\mathbf{a} \Phi_{(\mathbf{a},\mathbf{b})}(\nu) = 
3^{n-1} 
\sum_{\mathbf{a}\in\mathcal{A}_k} d_\mathbf{a} \Phi_{(\mathbf{a},\overline{\mathbf{b}})}(\hat{\nu}_g),
\end{equation}
where we denote $\overline{\mathbf{b}}=(b_{k+1}-2,\ldots,b_n-2)$ (note that the $\chi(p)$ factors in \eqref{eq:phi-a-recurrence} are all equal to $1$ because of the structure of $\mathbf{b}$). The right-hand side of \eqref{eq:sum-b-overline} is equal to $0$ by the inductive hypothesis (applied with the vector $\overline{\mathbf{b}}$ replacing~$\mathbf{b}$).

Finally, in the case where $g\le 2k-1$, we have
\begin{equation}
\label{eq:sum-da-phi-ab}
\sum_{\mathbf{a}\in\mathcal{A}_k} d_\mathbf{a} \Phi_{(\mathbf{a},\mathbf{b})}(\nu) 
= 
\sum_{\mathbf{a}\in\mathcal{A}_k,\ g\in \{a_1,\ldots,a_k\}} 
d_\mathbf{a} 3^{n-1} \chi(p(\mathbf{a},g)) \Phi_{(\hat{\mathbf{a}}_g,\overline{\mathbf{b}})}(\hat{\nu}_g) 
\end{equation}
(where the notation $p(\mathbf{a},g)$ was defined immediately after \eqref{eq:ycb-daprime-epansion}),
and this can be rewritten as a sum of the form
\begin{equation} \label{eq:sum-fa-prime}
\sum_{\mathbf{a}'\in\mathcal{A}_{k-1}} 
f_{\mathbf{a}'} \Phi_{(\mathbf{a}',\overline{\mathbf{b}})}(\hat{\nu}_g) 
\end{equation}
where $(f_{\mathbf{a}'})_{\mathbf{a}'\in\mathcal{A}_{k-1}}$ are some coefficients. On the other hand, by repeating the same computation in the case $n=k$ where $\mathbf{b}$ is replaced by the empty vector of length $0$, we see that for any $\mu\in\noncn{k}$ such that $g\matched{\mu}g+1$, we have
\begin{equation} \label{eq:lin-combin-fa-prime}
\sum_{\mathbf{a}\in\mathcal{A}_k} d_\mathbf{a} \Phi_{\mathbf{a}}(\mu) = \sum_{\mathbf{a}'\in\mathcal{A}_{k-1}} f_{\mathbf{a}'} \Phi_{\mathbf{a}'}(\hat{\mu}_g), 
\end{equation}
where the key observation is that the coefficients are the same numbers $f_{\mathbf{a}'}$ as in \eqref{eq:sum-fa-prime}. But we know (by the assumption of the claim we are trying to prove) that the left-hand side of \eqref{eq:lin-combin-fa-prime} is equal to $0$. Since this is true for any $\mu\in\noncn{k}$ for which $g\matched{\mu}g+1$, we conclude that 
$\sum_{\mathbf{a}'\in\mathcal{A}_{k-1}} f_{\mathbf{a}'} \Phi_{\mathbf{a}'}=0$. The inductive hypothesis now implies that 
$\sum_{\mathbf{a}'\in\mathcal{A}_{k-1}} f_{\mathbf{a}'} \Phi_{(\mathbf{a}',\overline{\mathbf{b}})}=0$. In particular, the expression in \eqref{eq:sum-fa-prime}, which is equal to the left-hand side of \eqref{eq:sum-da-phi-ab}, is $0$.
\end{proof}

\subsection{Proof of Theorem~\ref{thm:explicit-formulas-finite-n}}

\label{sec:proof-main-thm}

We are now in a position to prove Theorem~\ref{thm:explicit-formulas-finite-n}, with an explicit description of the multivariate polynomial $F_{\pi_0}(w_1,\ldots,w_k)$ whose existence is claimed in the theorem. Let $\pi_0 \in\noncn{k}$. The polynomial $F_{\pi_0}$ simply encodes the $\pi_0$th row of the matrix $\tilde{\boldsymbol{C}}_n$; more precisely, we define it as
\begin{equation}
\label{eq:submatching-polynomial-def}
F_{\pi_0}(w_1,\ldots,w_k) = \sum_{\mathbf{a}=(a_1,\ldots,a_k)\in\noncn{k}^+} \tilde{C}_{\pi_0,\mathbf{a}} \prod_{j=1}^{k} w_j^{2j-a_j}.
\end{equation}

Denote $\Omega_n(\mathbf{z}) = \prod_{1\le i<j\le n} (z_j-z_i)(1+z_j+z_i z_j)$.
Taking the $\mathbf{1}$-evaluation of both sides of \eqref{eq:submatching-exansion-ecal} and combining the resulting equation with Theorems~\ref{thm:psi-mu-related} and \ref{thm:sum-rule}, we get that
\begin{align*}
 \asm(n) \, & \prob\left(\pistarn \in \mathcal{E}_n(\pi_0)\right) = \sum_{\pi\in \mathcal{E}_n(\pi_0)} \psi_\pi
\\ &= 
\sum_{\mathbf{a}=(a_1,\ldots,a_k)\in\noncn{k}^+}
\tilde{C}_{\pi_0, \mathbf{a}} \hspace{-14.0pt}
\sum_{\begin{array}{c} \scriptstyle b_{k+2},\ldots,b_n \\[-3pt] \scriptstyle b_j\in\{2j-2,2j-1\} \end{array}}
\hspace{-14.0pt}
\phi_{(1,1+a_1,\ldots,1+a_k,b_{k+2},\ldots,b_n)}.
\end{align*}
The left-hand side is also equal to $\asm(n) \prob\left(\pi_0 \submatching \pistarn\right)$ by the rotational symmetry of $\pistarn$. To evaluate the right-hand side, represent each summand $\phi_{(1,1+a_1,\ldots,1+a_k,b_{k+2},\ldots,b_n)}$ as a constant term using \eqref{eq:phi-a-eval1-constantterm}, to get that the last expression can be written as
\begin{align*}
&
\CT_\mathbf{z} \left[ \vphantom{\sum_{\mathbf{a}=(a_1,\ldots,a_k)\in\noncn{k}^+}
\tilde{C}_{\pi_0, \mathbf{a}} 
\Omega_n(\mathbf{z}) \prod_{j=2}^{k+1} \frac{1}{z_j^{a_{j-1}}} \prod_{j=k+2}^n \left(\frac{1}{z_j^{2j-3}}+\frac{1}{z_j^{2j-2}}\right)
} \right. 
\sum_{\mathbf{a}=(a_1,\ldots,a_k)\in\noncn{k}^+}
\tilde{C}_{\pi_0, \mathbf{a}} \hspace{-14.0pt}
\sum_{\begin{array}{c} \scriptstyle b_{k+2},\ldots,b_n \\[-3pt] \scriptstyle \forall j\,\,b_j\in\{2j-2,2j-1\} \end{array}}
\Omega_n(\mathbf{z}) \prod_{j=2}^{k+1} \frac{1}{z_j^{a_{j-1}}} \prod_{j=k+2}^n \frac{1}{z_j^{b_j-1}}
\left. \vphantom{\sum_{\mathbf{a}=(a_1,\ldots,a_k)\in\noncn{k}^+}
\tilde{C}_{\pi_0, \mathbf{a}} 
\Omega_n(\mathbf{z}) \prod_{j=2}^{k+1} \frac{1}{z_j^{a_{j-1}}} \prod_{j=k+2}^n \left(\frac{1}{z_j^{2j-3}}+\frac{1}{z_j^{2j-2}}\right)
}
\right]
\\ &=
\CT_\mathbf{z} \left[
\sum_{\mathbf{a}=(a_1,\ldots,a_k)\in\noncn{k}^+}
\tilde{C}_{\pi_0, \mathbf{a}} 
\Omega_n(\mathbf{z}) \prod_{j=2}^{k+1} \frac{1}{z_j^{a_{j-1}}} \prod_{j=k+2}^n \left(\frac{1}{z_j^{2j-3}}+\frac{1}{z_j^{2j-2}}\right)
\right]
\\ &=
\CT_\mathbf{z} \left[
\sum_{\mathbf{a}=(a_1,\ldots,a_k)\in\noncn{k}^+}
\tilde{C}_{\pi_0, \mathbf{a}} 
\Omega_n(\mathbf{z}) \prod_{j=2}^{k+1} \frac{1}{z_j^{a_{j-1}}} \prod_{j=k+2}^n \frac{1+z_j}{z_j^{2j-2}}
\right]
\\ &=
\CT_\mathbf{z} \left[
\Omega_n(\mathbf{z}) \prod_{j=k+2}^n (1+z_j) \prod_{j=1}^{n} \frac{1}{z_j^{2j-2}} 
\left( \sum_{\mathbf{a}=(a_1,\ldots,a_k)\in\noncn{k}^+}
\tilde{C}_{\pi_0, \mathbf{a}} \prod_{j=2}^{k+1} z_j^{2j-2-a_{j-1}} \right)
\right]
\\
\\ &=
\CT_\mathbf{z} \left[
\Omega_n(\mathbf{z}) \prod_{j=k+2}^n (1+z_j) \prod_{j=1}^{n} \frac{1}{z_j^{2j-2}} 
\cdot F_{\pi_0}(z_2,\ldots,z_{k+1}) \right],
\end{align*}
which is clearly equal to $\asm(n)$ times the right-hand side of \eqref{eq:submatching-polycoeff-formula}.
\qed

\section*{Appendix A. Summary of algorithms}

The research described in this paper has been significantly aided by computer-aided experimentation and numerical computations performed by the author. We believe use of such methods is likely to continue to play a role in the discovery of new results extending and building on this work. In the hope of stimulating such further research, we summarize the theory presented above in the form of explicit algorithms for computing quantities of interest related to connectivity patterns of loop percolation and the theory of wheel polynomials. 

\bigskip \noindent
\textbf{Algorithm A\ \ } Computation of the probability vector $\boldsymbol{\mu}_n=(\mu_\pi)_{\pi\in\noncn{n}}$

\begin{enumerate}[labelwidth=-30pt, label=\textnormal{\textbf{Step \arabic*.}}]

\item
Compute the matrix $S_n = (s_{\pi,\sigma})_{\pi,\sigma\in\noncn{n}}$ where
$$ s_{\pi,\sigma} = \#\{ 1\le j\le 2n\,:\, e_j(\pi)=\sigma \}  $$
(with $e_j$ defined by the equation \eqref{eq:temperley-lieb-gens}, interpreted modulo $2n$).

\item 
Compute $\boldsymbol{\mu}_n$ as the solution to the vector equation 
$$\boldsymbol{\mu}_n (S_n - 2n \mathbf{I})= \boldsymbol{0}$$ 
(where $\mathbf{I}$ is the identity matrix), normalized to be a probability vector.

\end{enumerate}

\bigskip \noindent
\textbf{Algorithm B \ \ } Computation of the polynomials $Q_{\pi_0}$ (assuming Conjecture~\ref{conj:rationality-finite-n})

\begin{enumerate}[labelwidth=-30pt, label=\textnormal{\textbf{Step \arabic*.}}]

\item Use Algorithm A above to compute the submatching event probabilities $p_n = \prob\left( \pi_0 \submatching \pistarn  \right)$ for $n=k,k+1,\ldots,k(k+3)/2$.

\item Compute the numbers $q_n = p_n \prod_{j=1}^k (4n^2-j^2)^{k+1-j}$.

\item Use the Lagrange interpolation formula to compute the unique polynomial $G_{\pi_0}(m)$ of degree $\le k(k+1)/2$ satisfying
$$ G_{\pi_0}(n^2) = q_n \qquad (k\le n\le k(k+3)/2). $$

\item $Q_{\pi_0}$ is given by $Q_{\pi_0}(n)=G_{\pi_0}(n^2)$.
\end{enumerate}

\bigskip \noindent
\textbf{Algorithm C\ \ } Computation of submatching event probabilities in the half-planar model (assuming Conjecture~\ref{conj:rationality-finite-n})

\smallskip
Use Algorithm B above to compute the polynomial $Q_{\pi_0}$. Let $q_{\pi_0}^*$ be the leading coefficient of $Q_{\pi_0}$.
The probability $\prob(\pi_0 \submatching \pistar)$ is $q_{\pi_0}^*/2^{k(k+1)}$.

\vbox{
\bigskip \noindent
\textbf{Algorithm D\ \ } Computation of the matrices $\mathbf{C}_n$, $\tilde{\mathbf{C}}_n$

\smallskip
Translating \eqref{eq:cmatrix-explicit} to a notation appropriate for the computation of $C_{\pi,\sigma}$, we 
denote $p_{j,k}(\pi) = \#\{ j\le \ell<k \,:\, \pi_\ell = 1 \}-\tfrac12 (k-j-1)$. 
The matrix $\mathbf{C}_n=(C_{\pi,\sigma})_{\pi,\sigma\in\noncn{n}}$ is now computed using the explicit formula
$$
C_{\pi,\sigma} = \prod_{j<k,\ j\matched{\sigma}k} \chi(p_{j,k}(\pi))
$$
(with $\chi(\cdot)$ defined in \eqref{eq:character-mod-three}),
and $\tilde{\mathbf{C}}_n = \mathbf{C}_n^{-1}$ is its inverse matrix.
}

\bigskip \noindent
\textbf{Algorithm E\ \ } Computation of the polynomials $F_{\pi_0}$ (see Theorem~\ref{thm:explicit-formulas-finite-n})

\begin{enumerate}[labelwidth=-30pt, label=\textnormal{\textbf{Step \arabic*.}}]
\item Compute $\tilde{\mathbf{C}}_k$ using Algorithm D above.
\item Use \eqref{eq:submatching-polynomial-def} to compute $F_{\pi_0}$.
\end{enumerate}

\section*{Appendix B. Proof of Theorem~\ref{thm:hamiltonian-eigenvector}}

We start by writing a more explicit formula for the entries of the transition matrices $T_n^{(p)}$. To do so, encode each row of plaquettes as a vector $\mathbf{a}\in\{0,1\}^{2n}$, where a $0$ coordinate corresponds to a ``type $0$'' plaquette, defined as the plaquette shown on the left in Fig.~\ref{fig:plaquettes}, and a $1$ coordinate corresponds to a ``type $1$'' plaquette, which is the one on the right-hand side of the same figure. The parameter $p$ corresponds to the probability of a type $1$ plaquette. For a given plaquette-row vector $\mathbf{a}\in\{0,1\}^{2n}$, denote by $f_\mathbf{a}(\cdot)$ an operator that takes a noncrossing matching $\pi\in\noncn{n}$ and returns a new noncrossing matching $\pi'=f_\mathbf{a}(\pi)$ which is the result of composing the diagram of  $\pi$ with the row of plaquettes, in an analogous manner to the composition of matching diagrams with Temperley-Lieb operators shown in Fig.~\ref{fig:pipe-diagram}.

With this notation, it is clear from the definition of the transition matrix $T_n^{(p)}$ that its entries $(T_n^{(p)})_{\pi,\pi'}$ are given by
$$
(T_n^{(p)})_{\pi,\pi'} = \sum_{\begin{array}{c} \scriptstyle \mathbf{a}\in\{0,1\}^{2n} \\[-0.6ex] \scriptstyle f_\mathbf{a}(\pi)=\pi' \end{array}} p^{\sum_j a_j} (1-p)^{\sum_j (1-a_j)}
\qquad (\pi,\pi'\in\noncn{n}).
$$
Differentiate this equation with respect to $p$ at $p=0$. On the right-hand side the nonzero contributions will come from the vector $\mathbf{a}=\mathbf{0}=(0,\ldots,0)$ and from the vectors of the form $\mathbf{a}=\mathbf{a}_k = (0,\ldots,0,1,0,\ldots,0)$ with only one coordinate in some position $k$ equal to $1$. This gives
$$
\frac{d}{dp}_{|p=0} (T_n^{(p)})_{\pi,\pi'} = -2n \delta_{f_{\mathbf{0}}(\pi),\pi'} + \#\{ 1\le k\le 2n\,:\, f_{\mathbf{a}_k}(\pi)=\pi' \}.
$$
It is now not difficult to check (see Fig.~\ref{fig:temperley-lieb-plaquettes}) that $f_\mathbf{0}(\pi)=\rho(\pi)$, where $\rho$ is the rotation operator on matchings (defined in Theorem~\ref{thm:qkz-poly-properties}), and that the operator $f_{\mathbf{a}_k}$ bears a simple relation to the Temperley-Lieb operator $e_k$, namely, we have $f_{\mathbf{a}_k} = \rho\circ e_k = e_{k-1} \circ \rho$ (where $e_{k-1}$ is interpreted with the usual mod $2n$ convention).

Summarizing, since the vector $\boldsymbol{\mu}_n$ is the stationary probability vector for the matrices $T_n^{(p)}$, that is, it satisfies $\boldsymbol{\mu}_n T_n^{(p)} = \boldsymbol{\mu}_n$, the matrix $Q_n = \frac{d}{dp}_{|p=0} T_n^{(p)}$ satisfies $\boldsymbol{\mu}_n Q_n = 0$, and by the above observations it is easy to see that it is related to the operator $H_n$ from \eqref{eq:def-h-inf-gen} by $H_n = - R^{-1} Q_n$, where $R$ is a matrix version of the rotation operator $\rho$, whose entries are given by $ R_{\pi,\pi'} = \delta_{\rho(\pi),\pi'}$. It remains to observe that $\boldsymbol{\mu}_n$ is invariant also under the rotation action, i.e., we have $\boldsymbol{\mu}_n R = \boldsymbol{\mu}_n$, to conclude that $\boldsymbol{\mu}_n H_n = 0$, as claimed.
\qed

\begin{figure}
\begin{center}
\begin{tabular}{ccc}
\scalebox{0.7}{\includegraphics{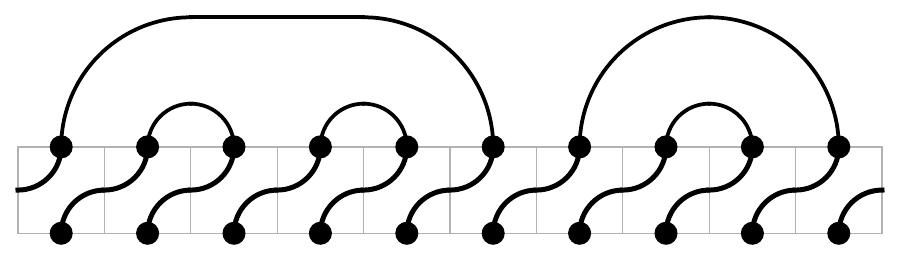}}
& &
\scalebox{0.7}{\includegraphics{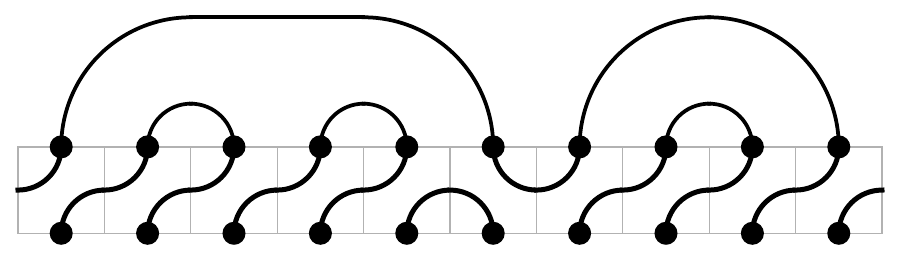}}
\\ (a)&&(b)
\end{tabular}
\caption{(a) Composing the diagram of a matching $\pi$ with a row of type $0$ plaquettes causes a simple rotation; (b) flipping the $k$th plaquette from type $0$ to type $1$ and leaving all other plaquettes as type $0$ is equivalent to the application of a Temperley-Lieb operator $e_k$ prior to the rotation.}
\label{fig:temperley-lieb-plaquettes}
\end{center}
\end{figure}

\section*{Acknowledgements}

The author thanks Omer Angel, Nathana\"el Berestycki, Jan de Gier, Christina Goldschmidt, Alexander Holroyd, Rick Kenyon, James Martin, Bernard Nienhuis, Ron Peled, David Wilson, Doron Zeilberger, and Paul Zinn-Justin for comments and helpful discussions during the preparation of the paper, and an anonymous referee for helpful suggestions.

The author was supported by the National Science Foundation under grant DMS-0955584, and by grant \#228524 from the Simons Foundation.

\bigskip \bigskip
\noindent
Dan Romik \\
Department of Mathematics \\
University of California, Davis \\
One Shields Avenue, Davis, CA 95616 \\[3pt]
Email: \texttt{romik@math.ucdavis.edu}

\end{document}